\documentclass[
  preprint, 3p, % <preprint, review, final> <1p, 3p, 5p> <twocolumn>
  number, % <authoryear, number>
  sort&compress,
]{elsarticle}
\pdfoutput=1

\usepackage[utf8]{luainputenc}
\usepackage[USenglish]{babel}
\usepackage{csquotes}

\usepackage{amsmath}
\numberwithin{equation}{section}
\allowdisplaybreaks
\usepackage{amssymb}
\usepackage{commath}
\usepackage{mathtools}
\usepackage{bbm}
\usepackage{nicefrac}
\usepackage{adjustbox}

\usepackage{siunitx}

\usepackage{amsthm}
\usepackage{thmtools}
\usepackage{etoolbox}
\makeatletter
\patchcmd{\thmt@setheadstyle}
 {\bgroup\thmt@space}
 {\thmt@space}
 {}{}
\patchcmd{\thmt@setheadstyle}
 {\egroup\fi}
 {\fi}
 {}{}
\makeatother
\declaretheoremstyle[
  bodyfont=\normalfont\itshape,
  headformat=\NAME\ \NUMBER\NOTE,
]{myplain}
\declaretheoremstyle[
  headformat=\NAME\ \NUMBER\NOTE,
]{mydefinition}
\newcommand{\envqed}{{\lower-0.3ex\hbox{$\triangleleft$}}}
\declaretheorem[style=myplain,numberwithin=section]{theorem}

\declaretheorem[style=myplain,numberlike=theorem]{proposition}

\declaretheorem[style=myplain,numberlike=theorem]{corollary}
\declaretheorem[style=mydefinition,numberlike=theorem,qed=\envqed]{definition}
\declaretheorem[style=mydefinition,numberlike=theorem,qed=\envqed]{remark}
\declaretheorem[style=mydefinition,numberlike=theorem,qed=\envqed]{example}

\usepackage[plainpages=false,pdfpagelabels,hidelinks,unicode]{hyperref}

\usepackage{color}
\usepackage{graphicx}
\usepackage[small]{caption}
\usepackage{subcaption}

\begingroup\expandafter\expandafter\expandafter\endgroup
\expandafter\ifx\csname pdfsuppresswarningpagegroup\endcsname\relax
\else
\fi

\usepackage{booktabs}
\usepackage{rotating}
\usepackage{multirow}

\usepackage{lineno}
\newcommand*\patchAmsMathEnvironmentForLineno[1]{%
  \expandafter\let\csname old#1\expandafter\endcsname\csname #1\endcsname
  \expandafter\let\csname oldend#1\expandafter\endcsname\csname end#1\endcsname
  \renewenvironment{#1}%
     {\linenomath\csname old#1\endcsname}%
     {\csname oldend#1\endcsname\endlinenomath}}%
\newcommand*\patchBothAmsMathEnvironmentsForLineno[1]{%
  \patchAmsMathEnvironmentForLineno{#1}%
  \patchAmsMathEnvironmentForLineno{#1*}}%
\AtBeginDocument{%
\patchBothAmsMathEnvironmentsForLineno{equation}%
\patchBothAmsMathEnvironmentsForLineno{align}%
\patchBothAmsMathEnvironmentsForLineno{flalign}%
\patchBothAmsMathEnvironmentsForLineno{alignat}%
\patchBothAmsMathEnvironmentsForLineno{gather}%
\patchBothAmsMathEnvironmentsForLineno{multline}%
}
\newcommand{\R}{\mathbb R}

\renewcommand{\L}{\mathcal L}

\newcommand{\ww}[1]{\underline{#1}}

\renewcommand{\vec}[1]{\ww{#1}}
\NewDocumentCommand{\mat}{mo}{%
	\IfValueTF{#2}{%
		\underline{\underline{#1}}{#2}
	}{%
		\underline{\underline{#1}}\,
	}%
}

\usepackage{bbm}
\def\bbc{\underline{\boldsymbol{\alpha}}}
\def\bc{\boldsymbol{\alpha}}

\def\bu{\boldsymbol{u}}

\def\bbd{\underline{\mathbf{d}}}
\def\bphi{\underline{\phi}}
\def\M{\underline{\underline{\mathrm{M}}}}

\def\TMM{\textsc{TMM}}
\def\DeC{\textsc{DeC}}
\def\ADER{\textsc{ADER}}
\def\sDeC{\textsc{sDeC}}
\def\eq{\textsc{EQ}}
\def\GLB{\textsc{GLB}}

\def\NODES{\textsc{NODES}}

\newtheorem{prop}{Proposition}

\newcommand{\orcid}[1]{ORCID:~\href{https://orcid.org/#1}{#1}}

\usepackage{pgfplots}
\usepackage{tikz}
\usetikzlibrary{positioning,calc}

\def\namedlabel#1#2{\begingroup
	#2%
	\def\@currentlabel{#2}%
	\phantomsection\label{#1}\endgroup
}

\begin{document}

\thispagestyle{empty}
\setcounter{page}{0}
\clearpage

\begin{frontmatter}

\title{Analysis for Implicit and Implicit-Explicit ADER and DeC Methods for Ordinary Differential Equations, Advection-Diffusion and Advection-Dispersion Equations}

\author[1,2]{Philipp~Öffner\fnref{orcidPO}}
\fntext[orcidPO]{\orcid{0000-0002-1367-1917}}
\address[1]{Institute of Mathematics, Johannes Gutenberg University Mainz, Staudingerweg 9, 55128 Mainz, Germany}
\address[2]{Institute of Mathematics, TU Clausthal, Germany}

\author[1]{Louis Petri}
\ead{lpetri01@uni-mainz.de}
\cortext[cor1]{Corresponding author}
%\address[1]{Institute of Mathematics, Johannes Gutenberg University Mainz, Staudingerweg 9, 55128 Mainz, Germany}

\author[3]{Davide Torlo\fnref{orcidDT}}
\fntext[orcidDT]{\orcid{0000-0001-5106-1629}}
\address[3]{Dipartimento di Matematica Guido Castelnuovo, Università di Roma La Sapienza,  Roma, Italy}

\begin{abstract}
%\begin{linenumbers}
	%In this paper, we develop implicit and implicit-explicit ADER and DeC methods 
%within the framework of the two-operators DeC and we study their stability as ODE solvers and in the context of linear PDE.
%We reformulate the methods as Runge-Kutta schemes to investigate their stability, discovering  that the methods strongly differ (from A-stable to bounded stability regions) according to the choice of order, method and quadrature nodes, differently from their explicit counterparts. 
%When applied to advection-diffusion and advection-dispersion equations (with finite difference spatial discretization), the von Neumann stability analysis shows that they are stable under CFL-like conditions and for the advection-diffusion equation even under a spatial-independent restriction. 
%We provide precise bounds for such coefficients and suggestion on which scheme is adequate in different problems.
In this manuscript, we present the development of implicit and implicit-explicit ADER and DeC methodologies within the DeC framework using the two-operators formulation, with a focus on their stability analysis both as solvers for ordinary differential equations (ODEs) and within the context of linear partial differential equations (PDEs).
To analyze their stability, we reinterpret these methods as Runge-Kutta schemes and uncover significant variations in stability behavior, ranging from A-stable to bounded stability regions, depending on the chosen order, method, and quadrature nodes. 
This differentiation contrasts with their explicit counterparts.
When applied to advection-diffusion and advection-dispersion equations employing finite difference spatial discretization, the von Neumann stability analysis demonstrates stability under CFL-like conditions. 
Particularly noteworthy is the stability maintenance observed for the advection-diffusion equation, even under spatial-independent constraints.
Furthermore, we establish precise boundaries for relevant coefficients and provide suggestions regarding the suitability of specific schemes for different problem.

%\end{linenumbers}
\end{abstract}

\begin{keyword}
  %% keywords here, in the form: keyword \sep keyword
 Implicit-Explicit Methods \sep
  Deferred Correction \sep
  ADER \sep
  Stability  \sep
  Advection-Diffusion  \sep
  Advection-Dispersion 
  %% MSC codes here, in the form: \MSC code \sep code
  %% or \MSC[2008] code \sep code (2000 is the default)
  \MSC
  %%% 
  65M06 \sep % NA, PDEs, IVPs, IBVPs: Finite difference methods for initial value and initial-boundary value problems involving PDEs
  65M20 \sep % NA, PDEs, IVPs, IBVPs: Method of lines for initial value and initial-boundary value problems involving PDEs
65L05 \sep %Numerical methods for initial value problems involving ordinary differential equations 
 65L06 \sep %Multistep, Runge-Kutta and extrapolation methods for ordinary differential equations 
65L07 %%% Stability 
%Numerical investigation of stability of solutions to ordinary differential equations
\end{keyword}

\end{frontmatter}

\section{Introduction} 
\label{sec:introduction} 
Many systems of time-dependent differential equations can be separated into multiple parts that differ in their stiffness. For such systems, using implicit-explicit (IMEX) time-marching methods \cite{pareschi2000implicit} is of paramount importance to guarantee stability and accuracy in many applications.

At the same time, high-order time-marching methods are sought for their efficiency and to match with the spatial discretization order in time-dependent partial differential equations (PDEs). Explicit high-order ADER and deferred correction (DeC) methods, due to their automatic construction, emerge as suitable alternatives to the traditional Runge-Kutta (RK) methods and have been extensively explored in various studies.
The explicit DeC method, introduced by Dutt et al. \cite{dutt2000dec} and then reinterpreted by Abgrall \cite{abgrall2017dec}, is an explicit, arbitrarily high-order method for ODEs. Further extensions of DeC, including implicit, semi-implicit and modified Patankar versions, are available in the literature \cite{christlieb2010integral,minion2003dec,offner2019arbitrary,abgrall2022relaxation,layton2004conservative,speck2015multi}. 
The ADER method was originally developed for hyperbolic systems exploiting the Cauchy-Kovalevskaya theorem \cite{ADERHistorical2, ADERHistorical1,titarev2002ader}, then reinterpreted  as a space-time discontinuous Galerkin (DG) method, which is solved through a fixed-point iteration procedure \cite{ADERModern,zbMATH07627644,dumbser2007FVStiff,boscheri2014direct,micalizzi2023efficient,veiga2023improving,Han_Veiga_2021}.

%In this work, we describe and analyze the implicit and IMEX version of ADER and DeC, both as ODE solvers and in the context of advection--diffusion or advection--dispersion PDEs. The description follows previous studies \cite{Han_Veiga_2021, minion2003dec, abgrall2018asymptotic}, where an IMEX description of DeC and ADER was already provided, but we extend the analysis to all possible methods, using different quadrature points and order of accuracy. Moreover, we study their IMEX stability in the spirit of \cite{Hundsdorfer,liotta2000central, minion2003dec}, noting large differences across the methods, from bounded stability areas to A-stable ones. This is in contrast with the behavior of the explicit versions.

In this research, we present an detailed investigation of both implicit and IMEX versions of ADER and DeC, investigating their efficacy as solvers for ordinary differential equations (ODEs) and in the context of linear  advection-diffusion or advection-dispersion PDEs.
Building upon prior work \cite{Han_Veiga_2021, minion2003dec, abgrall2018asymptotic, dumbser2007FVStiff}, which has explored IMEX descriptions of DeC and ADER, we expand our analysis to encompass the most used methods, employing varying quadrature points and levels of accuracy. Additionally, we explore their IMEX stability following the approach of previous studies \cite{Hundsdorfer,liotta2000central, minion2003dec}, uncovering notable discrepancies among them, ranging from bounded stability regions to A-stable ones. This diverges significantly from the behavior observed in explicit versions \cite{Han_Veiga_2021}.

Extending our investigation to the PDE case and inspired by  \cite{TanChenShu_ImEx_Stability}, 
%For the PDE case, inspired by \cite{TanChenShu_ImEx_Stability}, 
we conduct a von Neumann analysis for the presented IMEX time discretizations, paired with finite difference spatial discretizations of corresponding accuracy levels.
We find that the stability regions are bounded by CFL-type conditions as well as simple conditions on $\Delta t$ for the advection--diffusion case. 
%\PO{Emphasize this significant point: The CFL-type bounds remain invariant regardless of the spatial discretization method applied.}
% that do not involve the spatial discretization. 
 
% For the advection--dispersion, the analysis gives less clear results and only in few cases some conditions only dependent on spatial discretization. 
%All these findings are in agreement with the ODE stability regions found above.

The analysis of advection-dispersion presents less definitive outcomes, with only a few cases indicating conditions solely influenced by spatial discretization. However, these findings align with the stability regions observed in the ODE case.

%The paper is organized as follows. 
%In Section~\ref{sec:dec} and \ref{sec:ader}, we introduce the implicit and IMEX DeC and ADER methods, respectively, and we embed them into the RK framework. 
%We also provide some theoretical stability results for the pure implicit ADER method. In Section~\ref{sec:convergence}, we prove the high order accuracy of the implicit and IMEX ADER and DeC methods. In Section~\ref{sec: stability_analysis_ODE}, we describe their stability regions. In Sections~\ref{sec: advection_diffusion} and \ref{sec:PDE_adv_disp}, we extend the stability analysis to the PDE case by applying our IMEX methods to advection-diffusion and advection-dispersion equations.
%In Section~\ref{sec:numerics}, we perform few numerical examples to validate the stability and convergence analysis, and in Section~\ref{sec:conclusion} we draw some conclusions.

The structure of the paper is as follows. In Sections~\ref{sec:dec} and \ref{sec:ader}, we introduce the implicit and IMEX DeC and ADER methods, respectively, and incorporate them into the RK framework. Additionally, we present theoretical stability results for the pure implicit ADER method.
In Section~\ref{sec:convergence}, we establish the high order accuracy of both the implicit and IMEX ADER and DeC methods. Following this, in Section~\ref{sec: stability_analysis_ODE}, we delineate their stability regions.
Next, in Sections~\ref{sec: advection_diffusion} and \ref{sec:PDE_adv_disp} we extend the stability analysis to the PDE scenario by applying our IMEX methods to advection-diffusion and advection-dispersion equations.
In Section~\ref{sec:numerics}, we finally present several numerical examples aimed at validating the stability and convergence analyses, while in Section~\ref{sec:conclusion} we summarize the conclusions drawn from our deep analysis.

	\section{Deferred Correction}
	\label{sec:dec}
	%In this section, we present the DeC in its explicit, implicit and IMEX versions using the notation of \cite{Han_Veiga_2021}. 
In this section, we present the DeC in its explicit, implicit, and IMEX versions using the notation introduced in \cite{Han_Veiga_2021}.
Consider the system of ODEs
\begin{equation}\label{eq:scalarODE}
	\bc'(t)-F(\bc)=0,
\end{equation}
with $\bc:[0,T]\to \R^I$.
Given a time interval $[t_n, t_{n+1}]$ with length $\Delta t$, we subdivide  
it into $M$ subintervals  $\lbrace [t_m^{m-1},t_n^{m}]\rbrace_{m=1}^M$,
where $t_n^{0} = t_n$ and $t_n^{M} = t_{n+1}$. DeC methods are one step methods, hence we want to obtain $\bc_{n+1} \approx \bc(t_{n+1})$ from $\bc_{n} \approx \bc(t_{n})$.
We mimic for every 
subinterval $[t_n^0, t_n^m]$ the Picard--Lindel\"of theorem.
We drop the dependency on the timestep $n$ for subtimesteps $t_n^{m}$ 
and associated substates $\bc_n^{m}$, such that 
\begin{equation*}
	t_n=t_n^{0}=t^0, \ t_n^1=t^1, \ t_n^{2}=t^2, \hdots, \ t_n^{M}=t^M=t_{n+1}.
\end{equation*}

Then, we introduce the $\L^2$ operator, which represents an implicit high order discretization of ODE obtained integrating \eqref{eq:scalarODE} in each subinterval $[t^0,t^m]$ and applying a quadrature formula, 
\begin{equation}\label{eq:L2}
	\L^2(\bc^0, \dots, \bc^M) :=
	\begin{cases}
		\bc^M-\bc_n - \Delta t\sum_{r=0}^M \theta_r^M F(\bc^r)\\
		\vdots\\
		\bc^0-\bc_n - \Delta t\sum_{r=0}^M \theta_r^0 F(\bc^r)
	\end{cases}.
\end{equation}
We obtain this operator by applying to $\left(F(\bc^0), \hdots, F(\bc^M)\right)$ an interpolation polynomial of degree $M$.
The term $\theta_r^m := \frac{1}{\Delta t}\int_{t_n}^{t_n^{m}} \phi_r(s) ds$ denotes the weights, which can be obtained with Lagrange polynomials $\lbrace \phi_r \rbrace_{r=0}^M$ in the subnodes $\lbrace t^m \rbrace_{m=0}^M$, 
where $\phi_r(t^{m})=\delta_{r,m}$ and $\sum_{r=0}^M \phi_r(s) \equiv 1$ for any $s\in [0,1]$. 
We notice that $\L^2=0$ with $\bc_{n+1}=\sum_{r=0}^M\phi_r(t_{n+1})\bc^{r}$ is a collocation method and, when using Gauss--Lobatto nodes, it coincides with Lobatto IIIA schemes~\cite{veiga2023improving}.

\subsection{Explicit DeC}
To obtain an explicit versions of the DeC, we introduce $\L^1$, an explicit first order approximation of the $\L^2$ operator following \cite{abgrall2017dec, offner2023approximation}:
\begin{equation}\label{eq:L1}
\L^1(\bc^0, \dots, \bc^M) :=
\begin{cases}
\bc^M-\bc_n -  \Delta t\beta^M F(\bc_n) \\
\vdots\\
\bc^0- \bc_n -  \Delta t \beta^0 F(\bc_n)
\end{cases},
\end{equation}
with coefficients $\beta^m:= \frac{t_n^m-t_n}{\Delta t}$.
To simplify the notation, we introduce the vector of states for the variable $\bc$ at all subtimesteps
\begin{align}\label{eq:definition_bbc}
&\bbc :=  (\bc^0, \dots, \bc ^M) \in \R^{M\times I}, \text{ such that }\\
&\L^1(\bbc) := \L^1(\bc^0, \dots, \bc^M) \text{ and } \L^2(\bbc) := \L^2(\bc^0, \dots, \bc^M) .
\end{align}
Now, the DeC algorithm uses a combination of the $\L^1$ and $\L^2$ operators to recursively approximate $\bbc^*$, the high order accurate numerical solution of
the $\L^2(\bbc^*)=0$ scheme. We denote by $Y$ the order of accuracy of the solution $\bbc^*$. 
%This depends on the chosen nodes \cite{torlo2022} and we have $Y=M+1$ for equispaced nodes, and $Y=2M$ for Gauss--Lobatto ones. 
%To describe the procedure, for each variable we have to refer to both 
%the $m$-th subnode and the $k$-th iteration of the DeC algorithm by $\bc^{m,(k)} \in \R^I$.
This is contingent upon the nodes selected \cite{torlo2022}, where we have $Y=M+1$ for equispaced nodes and $Y=2M$ for Gauss-Lobatto nodes.
In detailing the process, for each variable, we need to reference both the $m$-th subnode and the $k$-th iteration of the DeC algorithm, denoted by $\bc^{m,(k)} \in \R^I$.
Finally, the DeC method can be written as
\begin{equation}\label{DeC_method}
\begin{split}
&\text{\textbf{DeC Algorithm}}\\
&\bc^{m,(0)}:=\bc_n,\quad m=1,\dots, M,\\
&\L^1(\bbc^{(k)})=\L^1(\bbc^{(k-1)})-\L^2(\bbc^{(k-1)}), \quad \text{ for }k=1,\dots,K.
\end{split}
\end{equation}
%The DeC method achieves order of accuracy $\min(K,Y)$, incrementing its accuracy by one at each iteration \cite{abgrall2017dec}. Hence, it is optimal to choose $K= Y$.
%Notice that the explicit operator $\L^1$ is the only one to solve, while $\L^2$ is only evaluated in the already computed predictions $\bbc^{(k-1)}$.
The DeC method attains an order of accuracy $\min(K,Y)$, enhancing its accuracy by one with each iteration \cite{abgrall2017dec}. Therefore, selecting $K= Y$ is optimal.
It's important to note that only the explicit operator $\L^1$ needs to be solved, while $\L^2$ is solely evaluated using the previously computed predictions $\bbc^{(k-1)}$.
\begin{remark}[Variations of DeC]
	The presented DeC algorithm is just one of a whole family of DeC methods. For example, instead of considering integrations on $[t^0,t^m]$ in the $m$'th equation, we could switch to the smaller intervals $[t^{m-1},t^m]$ like in \cite{minion2003dec,dutt2000dec,torlo2022}, changing the $m$'th line of the operators to
	\begin{align*}
	\mathcal{L}^{1,m}(\bc^0, \dots, \bc^M) &=
	\bc^m-\bc^{m-1} - \Delta t \gamma^m  F(\bc^{m-1}),\\
	\mathcal{L}^{2,m}(\bc^0, \dots, \bc^M) &=
	\bc^m-\bc^{m-1} - \Delta t\sum_{r=0}^M \delta_r^m F(\bc^r),
	\end{align*} 
	with $\gamma^m:= \frac{t_n^m-t_n^{m-1}}{\Delta t}$ and $\delta_r^m=\frac{1}{\Delta t}\int_{t_n^{m-1}}^{t_n^{m}}\phi_r(s)ds$. 
	The iteration process \eqref{DeC_method} with these $\mathcal{L}^1$ and $\mathcal{L}^2$ operators does not change. Due to the smaller steps, this method is also referred to as the sDeC algorithm. One downside of the sDeC algorithm is the necessity of $\bc^{m-1,(k)}$ to calculate $\bc^{m,(k)}$, which does not allow to use parallel computation on the different subnodes, which is possible for the DeC with larger subintervals presented above.\\
	In addition, we note that instead of using explicit Euler steps inside the $\mathcal{L}^1$ operators, other explicit RK methods can be applied inside the $\mathcal{L}^1$ operator \cite{tang2013high, christlieb2010integral}. Here, additional problems may rise. For a detailed overview of the different variations of DeC, we refer to the nice overview article \cite{dec_overview} and the references therein. 
	\end{remark}
\subsection{Implicit and IMEX DeC}
In this section, we will construct the implicit and IMEX DeC methods using the presented framework. Consider the ODE \eqref{eq:scalarODE}, where $F(\bc)$ is a stiff term. 
We proceed by taking an implicit version of the $\mathcal{L}^1$ operator, which corresponds to implicit Euler steps at each subinterval, for brevity, we will just describe the $m$'th equation for $m=0,\dots,M$
\begin{equation}\label{eq:ImL1}
\mathcal{L}^{1,m}(\bc^0, \dots, \bc^M) :=
\bc^m-\bc^0 - \beta^m \Delta t F(\bc^m)
\end{equation}
and assembling the implicit DeC method as in \eqref{DeC_method}.\\
If we have a problem, whose right-hand side can be separated into the sum of a stiff term $S(\bc)$ and a non-stiff term $G(\bc)$, i.e.
\begin{equation}\label{eq:IMEXODE}
\partial_t\bc=S(\bc)+G(\bc),
\end{equation} 
we create an implicit-explicit DeC (IMEX DeC) method, by adding up the implicit and explicit treatments of the $\mathcal{L}^1$ operator, obtaining for $m=0,\dots,M$
\begin{equation}
\mathcal{L}^{1,m}(\bc^0, \dots, \bc^M) :=
\bc^m-\bc^0 - \beta^m \Delta t S(\bc^m)- \beta^m \Delta t G(\bc^0),
\end{equation}
which we can substitute into the known correction procedure \eqref{DeC_method}.
In case of nonlinear stiff terms $S$, a further linearization of $S$ could be introduced in the definition of $\L^1$ \cite{Han_Veiga_2021}.

\subsection{Implicit and IMEX DeC as RK}
%% 
%\todo{I think this reference \cite{christlieb2010integral} is not appropriate here, as we do not talk about such methods (DeC where L1 itself is a RK)}
In this section, we rewrite the DeC methods presented above into RK methods to study their stability, as done in \cite{torlo2022}. 
The DeC has the advantage that one does not need to specify the coefficients for every order of accuracy as usually necessary in classical RK methods, see for example \cite[Remark 4.3]{abgrall2018asymptotic}.
This is done automatically, through the quadrature weights $\Theta$ and $\beta$, which fully determine the coefficients of the corresponding Butcher Tableau. 
Let us consider our version of DeC and rewrite it in a Runge--Kutta method with $Z$ stages defined by its Butcher tableau
\begin{equation}
	\begin{cases}
		\bu^{(s)} = \bc^n +\Delta t \sum_{i=1}^{Z} A_i^s G(\bu^{(i)}),\quad s=1,\dots, Z,\\
		\bc^{n+1} = \bc^n +\Delta t \sum_{i=1}^Z b_i  G(\bu^{(i)}),
	\end{cases} \qquad \begin{array}{c | c}
	c& \mat{A}
	\\ \hline
	& b
	\end{array}.
\end{equation}
%\begin{equation*}
%\label{eq:butcher}
%\begin{array}{c | c}
%c& A
%\\ \hline
%& b
%\end{array}.
%\end{equation*}
%The details of rewriting the DeC as RK can be found in \cite{torlo2022,abgrall2022relaxation}, we highlight here the final form to compare it with the implicit  and IMEX version\footnote

The process of rewriting DeC as Runge-Kutta schemes is elaborated in detail in \cite{torlo2022,abgrall2022relaxation}. Here, we emphasize the final form for comparison with the implicit and IMEX versions\footnote{Note that the re-interpretation of DeC as a RK scheme is not new and it was applied in various contexts, e.g. \cite{zbMATH06856490}.}.
We introduce $\vec{\beta}  = \lbrace \beta^i \rbrace_{i=0}^M$ as the vector of the $\L^1 $ operator coefficients, and we define the matrix $\mat{\theta} = \lbrace \theta_{r}^m \rbrace_{m=0,\dots,M; r = 0,\dots,M}$. % into $\mat{\theta} = (\vec{\theta}_0|\mat{\tilde{\theta}}),$ where $\vec{\theta}_0=(\theta_0^1, \dots, \theta_0^M)^T$ and $\mat{\tilde{\theta}} \in \mathbb R^{M\times M}$ is the remaining part of $\mat{\theta}$, as the initial subtimenode is never updated, see \eqref{DeC_method}. 
Finally, we introduce the vector $\vec{\theta}^M = (\theta^M_0, \dots, \theta^M_M)$ for the final update.
%
%Here, $A$ is block diagonal after the first column which includes the explicit Euler steps 
%from the $y^0$ to $y^i$, i.e. the $\beta^i$ coefficients with $\beta^0=0$. This results out of the first iteration, using the properties 
%\begin{equation}\label{eq:sumoverthetaeqbeta}
%\sum_{r=0}^{M}\phi_r(s)=1, \quad	\bc^{m,(0)}=\bc^0, \ \ m=0,\hdots ,M.
%\end{equation}
%Next, the $\theta_r^m$ are written 
%in a matrix form where the superscript $m$ describes the rows and $r$ the columns. We 
%have $\mat{\theta}=(\theta_r^m)$ and we split the matrix using $\mat{\theta}=(\vec{\theta}_0|\mat{\tilde{\theta}})$ ,
%where $\vec{\theta}_0=(\theta_0^1, \theta_0^2, \cdots, \theta_0^M)^T$ represents the first  column vector and
%$\mat{\tilde{\theta}}$ the remaining matrix. The first column of the Butcher tableau then includes   $\vec{\theta}_0$ 
%repeatedly to the number of corrections steps minus one before the final step is reached.  
%The remaining matrix 
%is written on the right side of it. But in each correction step it is shifted to the right depending on the used number of 
%subinterval. The open positions are filled with zeros until the last correction.
%Here, the correction does not affect the intermediate  values and has to
%be evaluated only at the last one. This means that it defines $b$ in the Butcher tableau  
%and $A$ is not affected anymore. We obtain $b_1= \theta_0^M$ and all 
%of the following $b_i$ values are zero with respect to the blocks of $A$. In connection,
%we include the vector  $ \vec{\theta}^{M} =(\theta_1^M, \dots, \theta_M^M )$.
The Butcher tableau for an arbitrary DeC approach is given by
\begin{equation}\label{eq:DeC_RK}
\begin{aligned}
\begin{array}{c|ccccccc}
	0 &  &&&&&&\\
\vec{\beta} & \vec{\beta} &  &   & & & & \\
\vec{\beta}  & \vec{0} &   \mat{{\theta}} & & & & & \\
\vdots &  \vec{0}& \mat{0}  &\mat{{\theta}}   && & & \\
\vdots & \vec{0} &  \mat{0}   &   \mat{0}  & \mat{{\theta}}  &  &  & \\
\vdots &  \vdots  &  \vdots &  \vdots &  \ddots  &  \ddots& &  \\
\vec{\beta} & \vec{0} &  \mat{0}  &  \hdots &  \hdots & \mat{0}  &  \mat{{\theta}} \\
\hline
&  0 & \vec{0}^T    & \hdots  &   &   \hdots &    \vec{0}^T& \vec{\theta}^{M} 
\end{array}.
\end{aligned}
\end{equation}
Let us notice that the first iteration is derived by a simplification that can be obtained noticing that $\sum_{r=0}^M \theta^m_r = \beta^m$, thanks to the properties
\begin{equation}\label{eq:sumoverthetaeqbeta}
\sum_{r=0}^{M}\phi_r(s)=1, \quad	\bc^{m,(0)}=\bc^0, \ \ 
m=0,\hdots ,M.
\end{equation}
Moreover, since we have chosen $t^0_n=t_n$, there are several trivial stages where the corresponding lines of $A$ is composed only of 0s.  These lines can be simplified as in \cite{Han_Veiga_2021,torlo2022}.
\begin{remark}[Number of stages]
	To obtain order $p$, we require $Y\stackrel{!}{=}K\stackrel{!}{=}p$, which means $M=p-1$ subtimesteps for equispaced nodes and $M=\lfloor \frac{p}{2}\rfloor$ for Gauss--Lobatto nodes, and $K=p$ corrections. 
	%Summing up all the stages, ignoring the final iteration where only the last subtimenode is relevant, we obtain $Z=M(K-1)+1$ stages that are $Z=(p-1)^2+1$ stages for equispaced and $Z=\lfloor \frac{p}{2}\rfloor(p-1)+1$ for Gauss--Lobatto nodes \cite{veiga2023improving}.
	In total, considering all nontrivial stages and excluding the final iteration where only the last sub-time node is pertinent, we arrive at $Z=M(K-1)+1$ stages, which equates to $Z=(p-1)^2+1$ stages for equispaced nodes and $Z=\lfloor \frac{p}{2}\rfloor(p-1)+1$ for Gauss-Lobatto nodes \cite{veiga2023improving}.
\end{remark}
For the implicit case, the DeC iteration process for every $m=0,\dots,M$ and every iteration $k=1,\dots , K$ reads
\begin{equation}\label{eq_help}
\begin{aligned}
&\mathcal{L}^{1,m}(\vec{\bc}^{(k)})=\mathcal{L}^{1,m}(\vec{\bc}^{(k-1)})-\mathcal{L}^{2,m}(\vec{\bc}^{(k-1)})\\
\Leftrightarrow  \quad &
%\begin{cases} \bc^{M,(k)}=\bc^0+\Delta t \left[\left(\sum_{r=0}^{M-1}\Theta_r^MF(\bc^{r,(k-1)})\right)+\left(\Theta_M^M-\beta^M\right)F(\bc^{M,(k-1)})+\beta^MF(\bc^{M,(k)}) \right]\\
%\vdots\\
\bc^{m,(k)}=\bc_n+\Delta t \left[\left(\sum\limits_{r=0, \,r\neq m}%{\substack{r=0, \\ r\neq m}}
^{M}\theta_r^mF(\bc^{r,(k-1)})\right)
+\left(\theta_m^m-\beta^m\right)F(\bc^{m,(k-1)})+\beta^mF(\bc^{m,(k)})\right].
%-\beta^m F(\bc^{m,(k-1)})+\beta^mF(\bc^{m,(k)}) \right].
%\vdots\\
%\bc^{1,(k)}=\bc^0+\Delta t \left[\left(\sum_{r=1}^{M}\Theta_r^1F(\bc^{r,(k-1)})\right)+\left(\Theta_1^1-\beta^1\right)F(\bc^{1,(k-1)})+\beta^1F(\bc^{1,(k)}) \right].
%\end{cases}
\end{aligned}
\end{equation}
From \eqref{eq_help}, we derive the blocks of the RK matrix $\mat{A}$, keeping in mind that for the first iteration $k=1$ many terms  simplify due to \eqref{eq:sumoverthetaeqbeta}. We obtain 
\begin{equation*}
\bc^{m,(1)}=
\bc^0+\Delta t \beta^mF(\bc^{m,(1)}).
\end{equation*}
%\\
The final update is  given by $\bc^{n+1}:=\bc^{M,(K)}$.
%To calculate the $K$th correction, which delivers $\bc^{n+1}$, we just need the $M$th equation 
%\begin{equation*}
%\bc^{M,(K)}=\bc^0+\Delta t \left[\left(\sum_{r=0}^{M-1}\Theta_r^M F(\bc^{r,(K-1)})\right) +\left(\Theta_M^M-\beta^M\right) F(\bc^{M,(K-1)})+\beta^MF(\bc^{M,(K)}) \right].
%\end{equation*}
%To finally receive the vector of coefficients $b$ and thereby a fitting, full Butcher tableau, we need to add a fake step by using the last equation
%\begin{align*}
%\bc^{n+1}:&=\bc^0+\Delta t \left[\left(\sum_{r=0}^{M-1}\Theta_r^M F(\bc^{r,(K-1)})\right) +\left(\Theta_M^M-\beta^M\right) F(\bc^{M,(K-1)})+\beta^MF(\bc^{M,(K)}) \right]\\
%&=\bc^{M,(K)},
%\end{align*}
This leads to the following RK Butcher tableau
\begin{align}\label{eq:ImDeC_RK_ButcherTableu}
\begin{array}{c|ccccccccc}
0 & 0 &   & &  & & &  & \\
\vec{\beta} & \vec{0} & \mat{B}  &   & & & & &  \\
\vec{\beta}  & \vec{0} &   \mat{{\theta}}-\mat{B} &\mat{B} & & & & &\\
\vdots & \vec{0} & \mat{0}  &\mat{{\theta}}-\mat{B} &\mat{B}& & & &\\
\vdots & \vec{0} &  \mat{0}   &   \mat{0}  & \mat{{\theta}}-\mat{B} &\mat{B}  &  & &\\
\vdots &  \vdots  &  \vdots &  \vdots &  \ddots  &  \ddots& \ddots & & \\
\vec{\beta} & \vec{0} &  \mat{0}  &  \hdots &  \hdots & \mat{0}  &  \mat{{\theta}}-\mat{B} &\mat{B} & \\
1&  0 & \vec{0}^T    & \hdots  &   &   \hdots &    \vec{0}^T& \vec{{\theta}}^M-\vec{B}^M &\beta^M \\
\hline
&  0 & \vec{0}^T    & \hdots  &   &   \hdots &\vec{0}^T &   \vec{{\theta}}^M-\vec{B}^M &\beta^M
\end{array},
\end{align}
with $B_{mr}= \delta_{mr} \beta^m$ for $m,r=0,\dots, M$ and $\delta_{mr}$ the Kronecker delta.
We notice that the ImDeC method is diagonally implicit and that the last stage coincide with the final update which makes the method stiffly accurate \cite[Proposition 3.8]{wanner1996solving}.\\
Next, we  describe the IMEX DeC in the notation of an IMEX RK scheme, i.e.,
\begin{equation}
\begin{cases}
\bu^{(s)} = \bc^n +\Delta t \sum_{i=1}^s A_i^s S(\bu^{(i)})+\Delta t \sum_{i=1}^{s-1} \hat{A}_i^s G(\bu^{(i)}),\quad s=1,\dots, Z,\\
\bc^{n+1} = \bc^n +\Delta t \sum_{i=1}^Z b_i S(\bu^{(i)})+\Delta t \sum_{i=1}^{Z} \hat{b}_i G(\bu^{(i)}),
\end{cases}
\end{equation}
where the matrices  $\mat{A},\mat{\hat{A}}$ and vectors $b,\hat{b}$  are provided in two Butcher tableaux
\begin{equation}\label{eq:IMEX_Butcher}
\begin{array}{c|c}
c & \mat{A}\\\hline
& b
\end{array}, \qquad
\begin{array}{c|c}
c & \mat{\hat{A}}\\\hline
& \hat{b}
\end{array}
\end{equation}
and directly correspond to the simpler implicit and explicit cases presented, with an extra stage in the explicit case to match the implicit description. Therefore, we obtain the following Butcher tableaux
\begin{equation}\label{eq:IMEX_DeC_RK}
\footnotesize
\begin{aligned}
\begin{array}{c|ccccccccc}
0 & 0 &   & &  & & &  & \\
\vec{\beta} & \vec{0} & \mat{B}  &   & & & & &  \\
\vec{\beta}  & \vec{0} &   \mat{{\theta}}-\mat{B} &\mat{B} & & & & &\\
\vdots & \vec{0} & \mat{0}  &\mat{{\theta}}-\mat{B} &\mat{B}& & & &\\
\vdots & \vec{0} &  \mat{0}   &   \mat{0}  & \mat{{\theta}}-\mat{B} &\mat{B}  &  & &\\
\vdots &  \vdots  &  \vdots &  \vdots &  \ddots  &  \ddots& \ddots & & \\
\vec{\beta} & \vec{0} &  \mat{0}  &  \hdots &  \hdots & \mat{0}  &  \mat{{\theta}}-\mat{B} &\mat{B} & \\
1&  0 & \vec{0}^T    & \hdots  &   &   \hdots &    \vec{0}^T& \vec{{\theta}}^M-\vec{B}^M &\beta^M \\
\hline
&  0& \vec{0}^T    & \hdots  &   &   \hdots &\vec{0}^T &   \vec{{\theta}}^M-\vec{B}^M &\beta^M
\end{array},
\quad
\begin{array}{c|ccccccccc}
0 & 0 &   & &  & & &  & \\
\vec{\beta} & \vec{\beta} &  &   & & & & &  \\
\vec{\beta}  & \vec{0} &   \mat{{\theta}} & & & & & &\\
\vdots & \vec{0} & \mat{0}  &\mat{{\theta}}   && & & &\\
\vdots & \vec{0} &  \mat{0}   &   \mat{0}  & \mat{{\theta}}  &  &  & &\\
\vdots &  \vdots  &  \vdots &  \vdots &  \ddots  &  \ddots& & & \\
\vec{\beta} & \vec{0} &  \mat{0}  &  \hdots &  \hdots & \mat{0}  &  \mat{{\theta}} & \\
1&  0 & \vec{0}^T    & \hdots  &   &   \hdots &    \vec{0}^T& \vec{\theta}_r^{M}& 0 \\
\hline
&  0 & \vec{0}^T    & \hdots  &   &   \hdots &    \vec{0}^T& \vec{\theta}_r^{M} & 0
\end{array}.
\end{aligned}
\end{equation}
\begin{remark}[Runge-Kutta for sDeC]
	Using the same properties and techniques, we can construct Butcher tableaux respectively for the sDeC, ImsDeC and IMEX sDeC. Their construction in the explicit case is performed in \cite{torlo2022} and we proceed similarly for the implicit and IMEX cases.
\end{remark}
Also in the implicit and IMEX cases, several trivial stages can be simplified in the Butcher tableaux \cite{torlo2022}.
	
	\section{ADER}
	\label{sec:ader}
	%\DTT{Restructure this section! ADER, IMEX ADER, IMEX ADER as DeC and RK, stability regions}

%\subsection{ADER method}
%Approaching the ADER method, at first it is worth to mention, that there are two different main formulations of the ADER algorithm. The first ADER method was developed for linear hyperbolic equations in \cite{schwartzkopff2002ader}, \cite{ADERHistorical2}  and extended for nonlinear hyperbolic systems in \cite{titarev2002ader}. It is based on a reconstruction of given values from cell averages, for which generalized Riemann problems have to be solved at the cell interfaces.\\
%The second, one on which we will focus, was presented in \cite{ADERModern} and is known as the \textit{modern} ADER approach. For simplicity, we will stay just with the term ADER. Its main idea is based on the variational formulation in the finite element context for an ODE, leading to a fixed-point problem.\\
%So starting with an $I$-dimensional system of ordinary differential equations as in \eqref{eq:scalarODE}, we consider the interval $T^n:=[t^n,t^{n+1}]$ where we represent $\bc(t)$ as a linear combination of $(M+1)$ basis functions (to receive order $M$ in time):

In this section, we present the modern ADER introduced as space-time DG solver for hyperbolic PDEs in \cite{ADERModern} and adapted for ODEs in \cite{Han_Veiga_2021,veiga2023improving}. Starting with an $I$-dimensional system of ODEs \eqref{eq:scalarODE}, we consider the interval $T^n:=[t^n,t^{n+1}]$ where we represent $\bc(t)$ as a linear combination of $(M+1)$ basis functions
\begin{equation}\label{eq:basis_reconstruction}
\bc(t) = \sum_{m=0}^M \phi_m(t)\bc^m = \bphi(t)^T\bbc,
\end{equation}
where 
$\bphi = [ \phi_0, \dots , \phi_M ]^T:T^n\to \R^{M+1}$ is the vector of Lagrangian basis functions in some given nodes $\{t_m\}_{m=0}^M\subset T^n$, e.g. equispaced or Gauss-Lobatto nodes, and $\bbc=[ \bc^0, \dots, \bc^M]^T \in \R^{(M+1)\times I}$ is the vector of coefficients of the basis functions respectively. \\
As described in \cite{Han_Veiga_2021}, ADER is derived multiplying \eqref{eq:scalarODE} with a smooth test function and integrating over $T^n$ to obtain the weak formulation. We insert the reconstruction \eqref{eq:basis_reconstruction} in the weak formulation, we use as test functions the basis functions, we apply integration by parts in time and using the quadrature $\lbrace (w_q, x_q)\rbrace_{q=0}^Q$ in $[t_n,t_{n+1}]$ results in a system
\begin{equation}\label{eq:ADER_fix_point}
\M \bbc = \vec{r}(\bbc) \Longleftrightarrow \mathcal{L}^2(\bbc):=\M\bbc-\vec{r}(\bbc)\stackrel{!}{=}0.
\end{equation}
%which can be expressed again with an $\L^2$ operator. 
Thereby, the mass-matrix $\M \in \R^{(M+1)\times(M+1)}$ and the (nonlinear) right-hand side functional $\vec{r}(\bbc):\R^{(M+1)\times I}\to \R^{(M+1)\times I}$ are given by
\begin{align}\label{eq:MassmatrixAder}
\M_{m,l} :&= \phi_m(t_{n+1})\phi_l(t_{n+1})- \sum_{q=0}^Q \partial_t \phi_m(x_q) \phi_l(x_q) w_q\\
\label{eq:aderrhsMatrix}
\vec{r}(\bbc)_m:&=\phi_m(t_{n})\bc_n
+ \Delta t \underbrace{\sum_{q=0}^{Q} w_q \phi_m(x_q) \bphi(x_q)^T}_{=:\mat{R}^m_z} F(\bbc) , \quad m=0,\dots,M+1.
\end{align}
The right hand side can also be written in matricial form as $\vec{r}(\bbc) = \vec{\phi}(t^{n})\bc(t^n) + \Delta t \mat{R} F(\bbc).$
\subsection{Explicit ADER}
To obtain an explicit approximation of the system \eqref{eq:ADER_fix_point} for the unknown $\bbc$, ADER resorts to a fixed-point problem, whose solution will give us a high order $Y$ accurate solution in $t$. In particular, the order of accuracy will be $Y=M+1$ for equispaced nodes and quadrature, order $Y=2M$ for Gauss--Lobatto nodes and quadrature and, order $Y=2M+1$ for Gauss--Legendre nodes and quadrature \cite{veiga2023improving}. 
%If $F(\bbc)$ is nonlinear, so is $r(\bbc)$ and therefore the fixed-point problem is in need to be solved approximately. 
We will use the following iterative procedure and then we reconstruct $\bc_{n+1}\approx\bc(t^{n+1})$ as %\\
%\centerline{\textbf{ADER Algorithm}} 
\begin{equation}\label{eq:fixpoint_iteration}
\begin{aligned}
&\text{\textbf{ADER Algorithm}}\\
&\bbc^{(0)}:=[\bc_n,\dots, \bc_n ]^T,\\
&\bbc^{(k)}=\M^{-1} \vec{r}(\bbc^{(k-1)}), \quad k=1,\dots,K, \\
&\bc_{n+1}=\bc_n+\Delta t \sum_{m=0}^M \int_{t_n}^{t_{n+1}}\phi_m(t) F(\bc^{(K),m})dt=\bc_n+\Delta t \underline{b}^TF(\bbc^{(K)}),
\end{aligned}
\end{equation}
with $b_i:=\int_{t^n}^{t^{n+1}} \phi_i(t)dt$.

%Several questions arise automatically for \eqref{eq:fixpoint_iteration},
%such as how is the convergence of the method influenced by the mass matrix $\M$ and, hence, by the node placement. In the next section, we build a bridge to the DeC algorithm as in \cite{Han_Veiga_2021}, to analyze and answer these questions.
\begin{remark}[ADER as DeC and order of accuracy of \eqref{eq:ADER_fix_point}]\label{rmk:ADERasDeC}
	As shown in \cite{Han_Veiga_2021,veiga2023improving}, 
	%it was shown, that the ADER and DeC can be expressed in the framework of one another by just small adaptations: For a different choice of test functions, we can express DeC as ADER.\\
	%The other way around, we can embed 
	ADER can be written into the DeC formalism by defining
	\begin{equation}\label{eq:L1ADER}
	\mathcal{L}^1(\vec{\bc}):=\mat{M}\vec{\bc}-\vec{r}(\vec{1} \bc_n ),
	\end{equation}
	so that the DeC iterations \eqref{DeC_method} coincide with the fixed point iterations of \eqref{eq:fixpoint_iteration}, i.e.,
%	and the $\L^2$ operator of the ADER framework in the DeC recursion \eqref{DeC_method}:
%	\begin{equation}\label{eq:iterativeADERasDeC}
%	\mathcal{L}^1(\vec{\bc}^{(k)})=\mathcal{L}^1(\vec{\bc}^{(k-1)})-\mathcal{L}^2(\vec{\bc}^{(k-1)}).
%	\end{equation}
%	 This is just to the ADER recursion
	 \begin{equation}
	 \mat{M}\vec{\bc}^{(k)}-\vec{r}(\vec{\bc}^{(k-1)})=0.
	 \end{equation}
	 If we set the starting values as $\bc^{(0),m}=\bc(t^n)$ for every $m$, like it is done in DeC, we obtain exactly the fixed-point iteration \eqref{eq:fixpoint_iteration} and therefore an equivalent definition of the ADER method.
	  This also tells us that the order of accuracy of \eqref{eq:fixpoint_iteration} with respect to the solution of $\L^2(\bbc^*)=0$ is $K$.
\end{remark}

\subsection{Implicit and IMEX ADER}
We start describing the implicit version of ADER (ImADER), considering \eqref{eq:scalarODE} with $F(\bc)$ stiff, by modifying the iterative process to
\begin{equation}\label{eq:imADER}
\vec{\bc}^{(k)} = \M^{-1}\vec{r}(\vec{\bc}^{(k)}) %= \vec{1} \bc_n + \Delta t \M^{-1}\mat{R} F(\vec{\bc}^{(k)}).
\end{equation}
or, as explained in Remark \ref{rmk:ADERasDeC}, with the ADER-DeC notation:
\begin{equation}\label{eq:L1IMADER}
\mathcal{L}^1(\bbc):=\M\bbc-\vec{\phi}(0)\bc_n-\Delta t \mat{R} F(\bbc).
\end{equation}
We soon realize that performing multiple corrections does not give us any advantages as \eqref{eq:imADER} does not depend on the previous iteration.
%, because one has to use a numerical approximation for the implicit corrections again, resulting for instance in the explicit fixed-point iteration, which solves for every $k$ the same equation. 
Hence, the construction of ImADER  does not seem purposeful, but it will become useful for the IMEX case, as presented in \cite{dumbser2007FVStiff} for PDE with stiff source terms or in \cite{Han_Veiga_2021} for ODEs. % and has to be taken in mind in the subsequent analysis.
%Nevertheless, we will take use of this approach to build the IMEX ADER method, like presented in \cite{dumbser2007FVStiff} but for the ODE case:

We consider the separated ODE system \eqref{eq:IMEXODE}. To construct the IMEX ADER $\L^1$ operator, we combine the implicit \eqref{eq:L1IMADER} and explicit \eqref{eq:L1ADER} treatments of the ADER iteration for the stiff and non-stiff terms, and get
\begin{equation}
\label{eq:IMEXADERL1}
\mathcal{L}^1(\vec{\bc}) :=  \M \vec{\bc} - \vec{\phi}(0)\bc_n -\Delta t \mat{R} S(\vec{\bc})-\Delta t \mat{R} G(\vec{1} \bc_n).
\end{equation}
This leads to the iterative process \eqref{eq:fixpoint_iteration} with the iteration given by
\begin{equation}
\begin{aligned}
&\M \vec{\bc}^{(k)} - \vec{\phi}(0)\bc_n -\Delta t \mat{R} S(\vec{\bc}^{(k)}) - \Delta t \mat{R} G(\vec{\bc}^{(k-1)})=0\\
\Longleftrightarrow\qquad &\vec{\bc}^{(k)} = \vec{1}\bc_n -\Delta t \M^{-1}\mat{R} S(\vec{\bc}^{(k)}) - \Delta t \M^{-1}\mat{R} G(\vec{\bc}^{(k-1)}),
\end{aligned}
\end{equation}
which is in fact just an additive combination of the implicit and explicit parts.
We apply the following proposition demonstrated in \cite[Proposition 2.4]{veiga2023improving} to write the IMEX ADER algorithms.
\begin{prop}[ADER right-hand side]\label{prop:RHS_ADER}
	Given the definition of $\underline{\underline{M}}$ in \eqref{eq:MassmatrixAder} and defining with $\vec{1}=[1,\dots,1]^T \in \R^{M+1}$, we have that 
	\begin{equation}
		\M^{-1}\vec{\phi}(t_n) = \vec{1}.
	\end{equation}
\end{prop}

If we take this under consideration in the whole ADER process, it leads to the %\\
%\centerline{\textbf{IMEX ADER Algorithm}} 
\begin{equation}
\begin{aligned}
&\text{\textbf{IMEX ADER Algorithm}}\\
&\bbc^{(0)}:=[\bc_n,\dots, \bc_n ]^T,\\
&\bbc^{(k)}=\vec{1}\bc_n -\Delta t \M^{-1}\mat{R} S(\vec{\bc}^{(k)}) - \Delta t \M^{-1}\mat{R} G(\vec{\bc}^{(k-1)}), \quad k=1,\dots,K,\\
&\bc_{n+1}=\bc_n+\Delta t \sum_{m=0}^M\int_{t_n}^{t_{n+1}} \left(S(\bc^{(K),m})+G(\bc^{(K),m})\right)dt=\bc_n+\Delta t \underline{b}^T\left(S(\bbc^{(K)})+G(\bbc^{(K)})\right). %\\ &\qquad\,\,
\end{aligned}
\end{equation}
We remark that the iteration process contains a (nonlinear) system of equations in $\bbc^{(k)}$ of dimension $(M+1)\times I$ for every $(k)$. If the stiff term is linear, the system becomes linear, otherwise, it is possible to linearize the stiff term in the definition of $\L^1$ \cite{Han_Veiga_2021}.
 
\subsection{IMEX ADER as RK}
In order to put the ADER method into a RK form, in \eqref{eq:fixpoint_iteration} we have to multiply the right-hand side by the inverse of the mass matrix. This will show the explicit dependence on $\bc_n$ for every iteration and subtimestep. 
Proposition~\ref{prop:RHS_ADER} allows us to rewrite the explicit iteration process \eqref{eq:fixpoint_iteration} as
\begin{equation}\label{eq: ADER_RK_step}
\begin{aligned}
\vec{\bc}^{(k)} = \M^{-1}\vec{r}(\vec{\bc}^{(k-1)})& =\M^{-1}\vec{\phi}(t_n) \bc_n + \Delta t \M^{-1}\mat{R} F(\vec{\bc}^{(k-1)})
=\vec{1} \bc_n + \Delta t \M^{-1}\mat{R} F(\vec{\bc}^{(k-1)}).
\end{aligned}
\end{equation}

Let us define the matrix $\mat{Q}:=\M^{-1} \mat{R}$ and the vector $\vec{P}$ such that $P_{m=0}^M=\sum_l Q_{ml}$. This last equation is directly used in the first iteration of the ADER process where all coefficients are initialized as $\bc_n$:
\begin{align*}
\vec{\bc}^{(1)}=\vec{1}\bc_n+\Delta t \mat{Q} F(\vec{\bc}^{(0)})
=\vec{1}\bc_n+\Delta t \mat{Q} F(\vec{1}\bc_n) 
=\vec{1}\bc_n+\Delta t \vec{P} F(\bc_n),
\end{align*}
representing the non-zero entries in the first column of the Butcher matrix $\mat{A}$. The further $(K-1)$ iterations just use the previous steps as in \eqref{eq: ADER_RK_step}, which give the entries of $\mat{Q}$. To achieve order $p$, we choose $p=K=Y$, i.e., $M=p-1$ for equispaced, $M=\lceil \frac{p}{2} \rceil$ for Gauss--Lobatto and $M = \lfloor \frac{p}{2} \rfloor$ for Gauss--Legendre, % to receive with the $K$ iterations over the $M+1$ subtimesteps and the initial stage 
for a total amount of $Z=K\times (M+1)+1$ stages, which for example for equispaced nodes is equal to $Z=p^2+1$ \cite{veiga2023improving} (some stages can be avoided when the row is identically 0). We can write the ADER, ImADER and IMEX ADER method of order $p$ as RK methods in a blockdiagonal matrix structure. Then, the Butcher tableaux of the IMEX ADER, which is composed of the explicit and implicit ones, is given by
\begin{equation}\label{eq:ImExADER_RK_ButcherTableu}
\begin{array}{c|ccccccc}
0 & 0 &   & &  & & &   \\
\vec{P} & \vec{0} &\mat{Q}  &   & & &    \\
\vec{P}  & \vec{0} &\mat{0} &   \mat{Q} & & & &  \\
\vdots & \vec{0} & \mat{0}  &\mat{0} &\mat{Q}   && & \\
\vdots & \vec{0} &  \mat{0}   &   \mat{0}  &\mat{0} & \mat{Q}  &  &   \\
\vdots &  \vdots  &  \vdots &  \vdots &  \ddots  &  \ddots&\ddots &  \\
\vec{P} & \vec{0} &  \mat{0}  &  \hdots &  \hdots & \mat{0}  & \mat{0} & \mat{Q}  \\
\hline
&  \vec{0}^T  & \vec{0}^T    & \hdots  &   &   \hdots &    \vec{0}^T& \vec{b}^{T} 
\end{array},\qquad
\begin{array}{c|ccccccc}
0 & 0 &   & &  & & &   \\
\vec{P} & \vec{P} &  &   & & & &   \\
\vec{P}  & \vec{0} &   \mat{Q} & & & & & \\
\vdots & \vec{0} & \mat{0}  &\mat{Q}   && & & \\
\vdots & \vec{0} &  \mat{0}   &   \mat{0}  & \mat{Q}  &  &  & \\
\vdots &  \vdots  &  \vdots &  \vdots &  \ddots  &  \ddots& &  \\
\vec{P} & \vec{0} &  \mat{0}  &  \hdots &  \hdots & \mat{0}  &  \mat{Q} & \\
\hline
&  \vec{0}^T  & \vec{0}^T    & \hdots  &   &   \hdots &    \vec{0}^T& \vec{b}^{T} 
\end{array}.
\end{equation}
Note that the implicit matrix $\mat{A}$ is particularly sparse and that it is block diagonal. Nevertheless, the method is not diagonally implicit, and also the final update does not have the same coefficients of the last stage.

\subsection{A-Stability of ImADER methods}
\label{subsec:ADER_A-stability}
First of all, let us notice that for all ImADER, the only determining iteration is the last one as $\bbc^{(k)} = \vec{1}\bc_n -\Delta t \mat{Q} S(\bbc^{(k)})$ for all $k$ does not depend on previous iterations in the purely implicit case.
Hence, the stability function reduces for all ImADER to the $\L^2=0$ methods denoted as ADER-IWF-RK in \cite{veiga2023improving}. It is given by the following Butcher tableau
\begin{equation}
	\begin{array}{c|c}
		\vec{P} & \mat{Q} \\ \hline 
		&b^T 
	\end{array}.
\end{equation}

For the Gauss--Lobatto nodes, it has been proven in \cite[Theorem A.3]{veiga2023improving} that the ADER-IWF-RK method coincide with the Lobatto IIIC method, hence, its stability function is the Pad\'e$(M-1,M+1)$ approximation, with $M+1$ the number of stages, and for \cite[Theorem 4.12]{wanner1996solving} it is A-stable.
This means that all ImADER with Gauss--Lobatto nodes are A-stable.

For Gauss--Legendre nodes, we need some theorems that can be found in \cite{wanner1996solving} to prove the same results. 
First a classical result of \cite{stetter1973analysis,scherer1979necessary} on the stability function of a RK method.
\begin{theorem}[{\cite[Proposition 3.2]{wanner1996solving}}]
	The stability function of a RK scheme satisfies
	\begin{equation}\label{eq:stability_function_form}
	R(z) = \frac{\det(\mat{Id}-z\mat{A}+z\vec{1} b^T)}{\det(\mat{Id}-z\mat{A})} = \frac{N(z)}{Q(z)}.
	\end{equation}
\end{theorem}

Then, we introduce the Pad\'e approximations and the following result.
\begin{theorem}[{\cite[Theorem 3.11]{wanner1996solving}}]\label{th:pade_order}
	The $(k,j)$-Pad\'e approximation to $e^z$ is given by
	\begin{equation}
	R_{kj}(z) = \frac{N_{kj}(z)}{Q_{kj}(z)}
	\end{equation}
	is the unique rational approximation to $e^z$ of order $j+k$, such that the degrees of the numerator and denominator are $k$ and $j$, respectively. 
\end{theorem}

Finally, we need the A-stability result for the Pad\'e approximations.
\begin{theorem}[{\cite[Theorem 4.12]{wanner1996solving}}]\label{th:pade_A_stab}
	A $(k,j)$-Pad\'e approximation $R_{kj}(z)$ to $e^z$ is A-stable if and only if $k\leq j \leq k+2$. All zeros and all poles are simple.
\end{theorem}

With these theorems, we can proceed studying the A-stability of the implicit ADER methods for Gauss--Legendre nodes.
We  prove that the stability function of the ADER-IWF-RK with $M+1$ Gauss--Legendre nodes is the Pad\'e$(M,M+1)$.
% Before doing it, we need an auxiliary result on the matrices involved in the numerator of the stability function.
Prior to proceeding, we require an additional outcome concerning the matrices present in the numerator of the stability function.
\begin{proposition}[Zero determinant of $\mat{A}-\vec{1}b^T$]\label{prop:zero_det}
	For the ADER-IWF-RK, we have $\det(\mat{A}-\vec{1}b^T)=0$.
\end{proposition}
\begin{proof}
	Without loss of generality, we consider the interval $[t_n,t_{n+1}]=[0,1]$ for simplicity. %With a simple rescaling it can be proven on any interval.
	To prove the result, let us recall the definition of the matrices. $\mat{A}=\mat{Q} = \M^{-1} \mat{R} $ and
	\begin{align*}
	b_i = w_i = \int_0^1 \phi_i(t)dt, \qquad  M_{ij} = \phi_i(1)\phi_j(1) - \int_0^1 \phi_i'(t)\phi_j(t) dt, \qquad R_{ij} = \int_0^1 \phi_i(t) \phi_j(t) dt = \delta_{ij} w_j.
	\end{align*} 
	Proving that $\det(\mat{A}-\vec{1}b^T)=\det(\mat{M}^{-1}\mat{R}-\vec{1}b^T)=0$ is equivalent to show that $\det(\mat{R}-\mat{M}\vec{1}b^T)=0$. 
	
	First, we study the matrix $\mat{M}\vec{1}b^T$. It can be rewritten as follows:
	\begin{align*}
	(\mat{M}\vec{1}b^T)_{ij}  = \sum_k M_{ik} 1_k b_j = \sum_k \left(\phi_i(1)\phi_k(1) - \int_0^1 \phi_i'(t)\phi_k(t) dt \right) w_j = \left(\phi_i(1) - \int_0^1 \phi_i'(t)dt \right) w_j = \phi_i(0) w_j.
	\end{align*}
	
	Then, we define $\mat{E}:=\mat{R}-\mat{M}\vec{1}b^T$ and we show that the sum of its rows is identically 0, hence, its rows are linearly dependent. It is
	\begin{align}
	\sum_{i} E_{ij} = \sum_{i} \left(R_{ij}-\phi_i(0) w_j\right) =  \sum_{i} \left(\delta_{ij} w_j-\phi_i(0) w_j\right) =  \sum_{i} \left(\delta_{ij}-\phi_i(0) \right) w_j  = (1-1)w_j = 0, \quad \forall j.
	\end{align} 
	This proves the statement.
\end{proof}

\begin{theorem}[Stability function of ADER-IWF-RK Gauss--Legendre]\label{th:pade_GLG}
	The stability function  of ADER-IWF-RK Gauss--Legendre is  the Pad\'e$(M,M+1)$ approximation.
\end{theorem}
\begin{proof}
	In \cite[Theorem 3.9]{veiga2023improving}, it has been shown that ADER-IWF-RK Gauss--Legendre is of order $2M+1$.
	Now, since the determinant of $\mat{A}-\mathbbm{1}b^T$ is zero, see Proposition~\ref{prop:zero_det}, there exists a unitary matrix $\mat{W} \in \R^{(M+1)\times (M+1)}$ such that $\mat{W}\mat{W}^T=\mat{Id}$ and
	\begin{equation}
	\mat{W} (\mat{A} -\vec{1}b^T )\mat{W}^T = \begin{pmatrix}
	0 & 0 & \dots &0\\
	* & * &  \dots &*\\
	\vdots & * &  \dots &*\\
	* & * &  \dots &*
	\end{pmatrix}. 
	\end{equation}
	%Hence, the numerator of \eqref{eq:stability_function_form} has the following expression
	Therefore, the numerator of \eqref{eq:stability_function_form} can be expressed as follows:
	\begin{align}
	\det\left(\mat{Id}-z\mat{A}+z\vec{1}b^T \right) = \det\left(\mat{Id}+z \mat{W} \left( \mat{A}-\vec{1}b^T\right)\mat{W}^T \right)= \det\left(\begin{pmatrix}
	1 & \underline{0}^T \\
	\underline{*} & \mat{Id}-z \mat{G} 
	\end{pmatrix}\right) = \det\left(\mat{Id}-z \mat{G} \right)
	\end{align}
	with $\mat{G}\in \R^{M\times M}$. Te degree of $N(z)$ is smaller or equal to $M$. Since the order of the ADER-IWF-RK Gauss--Legendre is $2M+1$, it must be that the degree of $N(z)$ is $M$ and the degree of $Q(z)$ is $M+1$.
	%From Theorem~\ref{th:pade_order} we know that the only approximation of $e^z$ with order $2M+1$, numerator with degree $M$ and denominator with degree $M+1$ is the Pad\'e$(M,M+1)$, it must be the one of ADER-IWF-RK Gauss--Legendre.
	Theorem~\ref{th:pade_order} establishes that the unique approximation of $e^z$ with an order of $2M+1$, a numerator of degree $M$, and a denominator of degree $M+1$ is the Padé$(M,M+1)$, implying it aligns with the ADER-IWF-RK Gauss--Legendre method.
\end{proof}

\begin{corollary}[A-stability of ImADER Gauss--Legendre]
	ADER-IWF-RK Gauss--Legendre and all ImADER Gauss--Legendre are A-stable.
\end{corollary}

%For equispaced ADER methods it is not so easy to derive similar results. Numerically, from order 5 on, we observe some instabilities on the imaginary axis and unstable regions in $\mathbb C^-$ with width close to machine precision and we think that the stability region may cross that axis. 

Deriving similar results for equispaced ADER methods poses a challenge. Numerically, starting from the fifth order onward, we notice instabilities along the imaginary axis and unstable regions within $\mathbb{C}^-$, with widths approaching machine precision. This observation suggests the possibility of the stability region intersecting that axis.

	\section{Convergence Analysis}
	\label{sec:convergence}
	As seen above, the order of accuracy of the DeC procedure \eqref{DeC_method} is $\textrm{min}\{K,Y\}$ where $K$ is the number of iteration and $Y$ the order of the $\L^2$ operator.
The proof of this statement, as stated in \cite{abgrall2017dec, Han_Veiga_2021,veiga2023improving}, requires some hypotheses that must be checked for the implicit and IMEX cases.
\begin{proposition}[DeC iterative method]\label{DeC_prop}
	Let $\L^1$ and $\L^2$ be two operators defined on $\mathbb{R}^{(M+1)\times I}$, 
	which depend on the discretization scale $\Delta = \Delta t$, such that
	\begin{itemize}
		\item[\namedlabel{itm:coercivity}{\bf{C1.}}] $\L^1$ is coercive with respect to a norm, i.e., 
		$\exists\, \gamma_1 >0$ independent of $\Delta$, such that for any $\bbc,\bbd$
		$$\gamma_1||\bbc-\bbd||\leq ||\L^1 (\bbc)-\L^1 (\bbd)||,$$
		\item[\namedlabel{itm:Lipschitz}{\bf{C2.}}] $\L^1 - \L^2$ is Lipschitz with constant $\gamma_2>0$
		uniformly with respect to $\Delta$, i.e., for any $\bbc,\bbd$
		$$
		||(\L^1(\bbc)-\L^2(\bbc))-(\L^1(\bbd)-\L^2(\bbd))||\leq \gamma_2 \Delta ||\bbc-\bbd||,
		$$
		\item[\namedlabel{itm:existence}{\bf{C3.}}] there exists a 
		unique $\bbc^*$ such that $\L^2(\bbc^*)=0$. 
	\end{itemize}
	Then, if $\eta:=\frac{\gamma_2}{\gamma_1}\Delta<1$,
	the DeC is converging to $\bbc^*$ and after $k$ iterations
	the error $||\bbc^{(k)}-\bbc^*||$ is smaller than $\eta^k||\bbc^{(0)}-\bbc^{*}||$.
\end{proposition}

Proofs of this proposition and of the hypotheses of the proposition for operators $\L^1$ and $\L^2$ for explicit DeC and ADER can be found in \cite{abgrall2017dec,abgrall2018asymptotic,offner2019arbitrary}.
The condition for $\eta$ comes from the fixed--point theorem and it is needed to converge.
As foreshadowed in Proposition~\ref{DeC_prop}, we need to prove the conditions~\ref{itm:coercivity}, \ref{itm:Lipschitz} and \ref{itm:existence} also in the implicit and IMEX cases, as for the explicit cases this was shown in \cite{Han_Veiga_2021}.
Here, we want to extend this proof to the IMEX $\mathcal{L}^1$ operators. 
The proof of~\ref{itm:existence} is as in the explicit case, because just the $\mathcal{L}^1$ operator changes in the implicit and IMEX cases. 
The arguments to prove \ref{itm:Lipschitz} are also exactly the same as in the explicit case, because it all boils down to the Lipschitz continuity of the right hand side of the ODE. \\
%For the general case, it turns out we can not prove the coercivity, so we proceed with a linearization of the stiff term in $\L^1$, see \cite{Han_Veiga_2021}, which is still a first order approximation to the implicit terms.
In the general scenario, we find ourselves unable to prove coercivity. Therefore, we opt to linearize the stiff term in $\mathcal{L}^1$, as outlined in \cite{Han_Veiga_2021}. This linearization still provides a first-order approximation to the implicit terms.
We substitute $S(\bbc)$ with $S'(\vec{1}\bc_n)\bbc$, where $S'$ is the Jacobian of $S$. Note that this simplification is exact for linear systems where $S(\bbc)=\mat{S}'(\vec{1}\bc_n)\bbc$. 
Moreover, this formulation can also be used to incorporate the nonlinear solver inside the DeC iteration method, without the need of further nonlinear solvers \cite{Han_Veiga_2021}.
\begin{prop}[IMEX DeC: \ref{itm:coercivity}]\label{prop: ImDeC_Coercivity}
	Assume that we apply a first order approximation of the IMEX DeC method linearizing the implicit terms, i.e., $\mathcal{L}^1$ is defined as
	\begin{equation}
	\label{eq: L1_linearized_DeC}
	\tilde{\mathcal{L}}^1(\bbc):=\bbc -\vec{1}\bc_n-\Delta t\vec{\beta}S'(\vec{1}\bc_n)(\bbc-\vec{1}\bc_n)  - \Delta t\vec{\beta}\left(S(\vec{1}\bc_n)+F(\vec{1}\bc_n)\right).
	\end{equation}
	Let
	\begin{equation}\label{eq: timestep_restriction_IMDeC}
	\Delta t< \frac{1}{2\tilde{\beta} L},
	\end{equation}
	where $\tilde{\beta}:=\max\limits_{1\le m\le M}\{\beta^i\}\leq 1$ and $L$ is the Lipschitz constant of $S$.
	Then, given any $\bbc, \bbd \in \mathbb{R}^{(M+1)\times I}$, there exists a positive $C_0$, such that
	\begin{equation*}
	C_0||\bbc-\bbd||\leq ||\tilde{\mathcal{L}}^1 (\bbc)-\tilde{\mathcal{L}^1} (\bbd)||
	\end{equation*}
	is fulfilled for the $\tilde{\mathcal{L}}^1$ IMEX DeC operator.
\end{prop}
\begin{proof}
	First, we recall some basic properties for eigenvalues, which we will use in the proof.
	\begin{itemize}	
		\item[\namedlabel{itm:eigen_bound}{i)}] 	Let $\lambda$ be an eigenvalue of $\mat{A}\in \mathbb{C}^{n\times n}$. Then, it holds
		$\lvert \lambda \rvert \le \norm{\mat{A}}$.
		%\item[ii)] Let $\lambda$ be an eigenvalue of $\mat{A}\in \mathbb{C}^{n\times n}$, $k\in \mathbb{C}$, then $k\cdot \lambda$ is an eigenvalue of $k\cdot \mat{A}$.
		%for an arbitrary, vector induced matrix norm $\norm{\cdot}$.
		\item[\namedlabel{itm:eigen_plus_id}{ii)}] Let $\mat{A}\in\mathbb{C}^{n\times n}$, $\mat{Id}$ the $n$-dimensional identity matrix and $\gamma, \delta \in \mathbb{C}$ and $\lambda \in \mathbb C$ an eigenvalue of $\mat{A}$.
		Then, $\gamma \lambda+\delta$ is an eigenvalue of $\gamma  \mat{A}+\delta \mat{Id}$.
	\end{itemize}
	We know that  $\norm{S'(\bc)}$ is bounded by $L$, the Lipschitz continuity constant of $S$, and, by property~\ref{itm:eigen_bound}, all the absolute values of the eigenvalues $\lambda_{S'}(\bc)$ of $S'(\bc)$ are bounded by $L$ for every $\bc\in \mathbb{R}^{I}$. 
	Now, using property \ref{itm:eigen_plus_id} and the restriction \eqref{eq: timestep_restriction_IMDeC}, we have that for every eigenvalue $\lambda_{\beta_m}$ of $\Delta t \beta_mS'(\bc)$ it holds $\lvert \lambda_{\beta_m}(\bc)\rvert<\frac{1}{2}$.
	Using property~\ref{itm:eigen_plus_id}, we can estimate for each $m=0,\dots,M$ the absolute value of the eigenvalues  of 
	\begin{equation*}
	\mat{Z_m}:=\mat{Id}-\Delta t \beta^m S'(\bc),
	\end{equation*}
	which are therefore all larger than $\frac{1}{2}$ for every $\bc \in \mathbb{R}^{I}, \; 1\le m  \le M$, leading to the invertibility of $\mat{Z_m}$. \\
	We consider the block-diagonal matrix $\mat{Z}$ with $\mat{Z_m}$ on each block-entry, that correspond to the system matrix of $\tilde{\L}^1$.
	The eigenvalues of $\mat{Z}^{-1}$ are the reciprocal of the eigenvalues of $\mat{Z}$, hence, all smaller than 2, and so $\lVert \mat{Z}^{-1}\rVert \leq 2$. Note that for any $\bbc, \bbd \in \mathbb{R}^{(M+1)\times I}$, it holds
	\begin{align}
	\label{eq: dec_coercivity_conclusion}
		\lVert \bbc -\bbd  \rVert = \lVert \mat{Z}^{-1} \mat{Z}(\bbc -\bbd ) \rVert \leq  \lVert \mat{Z}^{-1}\rVert \lVert \mat{Z}(\bbc -\bbd ) \rVert \leq 2 \lVert \mat{Z}(\bbc -\bbd ) \rVert.
	\end{align}
	By considering the $\tilde{\mathcal{L}}^1$ operator of the ImDeC, we expand
	\begin{align*}
	\tilde{\mathcal{L}}^1(\bbc)-\tilde{\mathcal{L}}^1(\bbd)&=\bbc-\bbc^0-\Delta t\vec{\beta}S'(\bbc^0)\bbc -\Delta t\vec{\beta}F(\bbc^0) - \bbd+\bbd^0+\Delta t\vec{\beta}S'(\bbc^0)\bbd +\Delta t\vec{\beta}F(\bbc^0)\\
	%&=\bbc-\bbd - \Delta t \vec{\beta}S'(\bbc^0)\left(\bbc-\bbd\right)\\
	&=\left(\mat{Id}-\Delta t \vec{\beta}S'(\bbc^0)\right)\left(\bbc-\bbd\right)%\\
	%&
	=Z\left(\bbc-\bbd\right).
	\end{align*}
	Finally, we conclude that
\begin{equation}
\norm{\tilde{\mathcal{L}}^1 (\bbc)-\tilde{\mathcal{L}^1} (\bbd)} = \norm{\mat{Z}(\bbc -\bbd )} \geq \frac12 \norm{\bbc -\bbd },
\end{equation} 
thanks to  \eqref{eq: dec_coercivity_conclusion}.
\end{proof}

\begin{prop}[IMEX ADER: \ref{itm:coercivity}]\label{prop:coercivity_ADER_IMEX}
	Assume we apply a first order linear approximation of the IMEX ADER method, i.e., we change the $\mathcal{L}^1$ operator to
	\begin{equation}
	\label{eq: L1_linearized_ADER}
	\tilde{\mathcal{L}}^1(\bbc):=\bbc-\bbc^0-\Delta t \M^{-1}\mat{R} S'(\bbc^0)(\bbc -\bbc^0) -\Delta t \M^{-1}\mat{R} (S(\bbc^0)+F(\bbc^0)).
	\end{equation}
	Let
	\begin{equation}\label{eq: timestep_restriction_IMADER}
	\Delta t< \frac{1}{2CL},
	\end{equation}
	where $C=\norm{\M^{-1}\mat{R}}=\mathcal{O}(1)$ and $L$ is the Lipschitz constant of $S$. 
	Then, given any $\bbc, \bbd \in \mathbb{R}^{(M+1)\times I}$, there exists a positive $C_0$, such that
	\begin{equation*}
	C_0||\bbc-\bbd||\leq ||\tilde{\mathcal{L}}^1 (\bbc)-\tilde{\mathcal{L}^1} (\bbd)||
	\end{equation*}
	is fulfilled for the $\tilde{\mathcal{L}}^1$ IMEX ADER operator.
\end{prop}
\begin{proof}
	Also for the $\tilde{\mathcal{L}}^1$ operator of the ImADER, the proof is similar to the ImDeC one. 
	We know that $\norm{S'(\bc)}\le L$ holds for every $\bc$ and $C=\norm{\M^{-1}\mat{R}}>0$, because $\mat{M}$ and $\mat{R}$ are constant. 
	Therefore, we can deduce that the eigenvalues $\lambda_{M^{-1}R}(\bbc)$ of $\M^{-1}\mat{R}S'(\bbc)$ can be estimated by 
	\begin{equation*}
	\lvert \lambda_{M^{-1}R}(\bbc) \rvert \le \norm{(\M^{-1}\mat{R})S'(\bbc)} \le \norm{\M^{-1}\mat{R}}\norm{S'(\bbc)} \le C \cdot L
	\end{equation*}
	for every $\bc$. 
	This, combined with condition \eqref{eq: timestep_restriction_IMADER}, leads us again to the property that all the absolute values of the eigenvalues of
	\begin{equation*}
	\mat{Id}-\Delta t \M^{-1}\mat{R}S'(\bbc)
	\end{equation*} 
	are bigger than $\frac{1}{2}$.
	With the same arguments as in the proof of Proposition~\ref{prop: ImDeC_Coercivity}, we conclude that $\tilde{\mathcal{L}}^1$ is coercive.
\end{proof}
In many applications, the hypothesis on the time-step as assumed in Propositions~\ref{prop: ImDeC_Coercivity} and \ref{prop:coercivity_ADER_IMEX} are too restrictive, especially when considering stiff equations. So, we present a variation of this proof, which does not require these time-step restrictions but uses, instead, another assumption on $S$ that is typical for damping/diffusion operators.
\begin{prop}[IMEX DeC/ADER: Variation of \ref{itm:coercivity} for diffusion terms]
	Consider again the first order approximations $\tilde{\mathcal{L}^1}$ as in \eqref{eq: L1_linearized_DeC} for the IMEX DeC and in \eqref{eq: L1_linearized_ADER} for IMEX ADER. 
	Assume additionally that the Jacobian $S'$ is symmetric negative definite.
	Then, given any $\bbc, \bbd \in \mathbb{R}^{(M+1)\times I}$, there exists a positive $C_0\geq 1$ independent of $\Delta t$, such that
	\begin{equation*}
	C_0||\bbc-\bbd||\leq ||\tilde{\mathcal{L}}^1 (\bbc)-\tilde{\mathcal{L}^1} (\bbd)||
	\end{equation*}
	is fulfilled for both $\tilde{\mathcal{L}}^1$ operators.
\end{prop}
\begin{proof}
	%\todo{You're right, we can skip some parts}
	We start from the IMEX DeC. Using again property~\ref{itm:eigen_plus_id} from the proof of Proposition \ref{prop: ImDeC_Coercivity}, we can deduce immediately that $\mat{Z}:=\mat{Id}-\Delta t \beta^m S'$ is symmetric positive definite and any eigenvalue $\lambda_{Z}$ is larger than $1+ \Delta t \beta^m |\lambda_{S'}|>1$. Hence, the eigenvalues of $\mat{Z}^{-1}$ are all smaller than one. 
	Therefore, as in Proposition~\ref{prop: ImDeC_Coercivity}, we have that 
	\begin{equation}\label{eq:coerc_L1_pos_def}
		\norm{\tilde{\L}^1(\bbc)-\tilde{\L}^1(\bbd)} = \norm{\mat{Z}(\bbc - \bbd) } \geq \frac{1}{\norm{\mat{Z}^{-1}}}\norm{\bbc-\bbd} > \norm{\bbc-\bbd}. 
	\end{equation}
	This result holds independently of the size of $\Delta t$.	
	For IMEX ADER the matrix in consideration is $$\mat{Z}:= \mat{Id}-\Delta t \mat{M}^{-1}\mat{R} S'(\bc_n),$$
	where, implicitly, there is a Kronecker product between $\mat{M}^{-1}\mat{R}\in \mathbb R^{(M+1)\times (M+1)}$ and $S'(\bc_n)\in\mathbb R^{I\times I}$. We know that $-S'(\bc_n)$ is positive definite, if also $\mat{M}^{-1}\mat{R}$ is positive define. Then, we have that their Kronecker product is positive definite \cite[Section 1]{vanloan1993approximation}. In \cite{veiga2023improving}, it has been shown that $\mat{M}$ is invertible and it is equivalent to a Hilbert matrix with an extra column of zeros on the left and an extra row of ones on the top.  
	Since, Hilbert matrices are positive definite, also $\mat{M}$ is positive definite and its inverse is positive definite as well.
	Now, $\bbc^T\mat{R}\bbc = \int_{0}^{1} \bc(t)^2 \geq 0$ when computed exactly, and for Gauss--Lobatto nodes, where the quadrature is not exact \cite{veiga2023improving}, it is $\bbc^T\mat{R}\bbc = \sum_{m=0}^M w_m (\bc^{m})^2 \geq 0$ as $w_m>0$. Hence,  $\mat{R}$ is positive definite and all eigenvalues of $\mat{Z}$ are $\lambda_Z>1$. We can then proceed as in \eqref{eq:coerc_L1_pos_def} to show that $\tilde{\L}^1$ is coercive with $C_0\geq 1$.
\end{proof}

	\section{Numerical Stability Analysis}
	\label{sec: stability_analysis_ODE}
	In this section, we study the stability of the presented method for the linear Dahlquist equation $u' = -\lambda u$ or $u' = -\lambda_I u - \lambda_E u$ for IMEX methods. All methods can be rewritten as $u_{n+1} = R(z)u_n$ or $u_{n+1}= R(z_I,z_E)$ being $R$ stability functions and $z,z_I,z_E\in \mathbb C$. 
We will study the stability regions $\lbrace |R|\leq 1\rbrace \subset \mathbb C$ for the implicit ADER/DeC, while for the IMEX schemes we consider different approaches. The explicit cases the stability functions for DeC and ADER have already been investigated in \cite{Han_Veiga_2021} and reported, while we show in the repository \cite{ourrepo} the sDeC for completeness. 
We will use Gauss-Lobatto (GLB) nodes as quadrature nodes, while we compare equispaced nodes also in the repository \cite{ourrepo} and we highlight here only the main differences.
We will numerically compute the stability regions obtained from the stability functions of ADER/DeC that are defined through their Butcher tableaux \eqref{eq:ImDeC_RK_ButcherTableu}, \eqref{eq:IMEX_DeC_RK}, \eqref{eq:ImExADER_RK_ButcherTableu}, see \cite{hairer1987solving}.
In detail, we compute on $200\times 200$ grid points with an offset of $+0.01$ from the origin for both axes to avoid singularities. 
The plot bounds are dependent on the type of scheme and their stability regions.
We decreased the offset in Figures~\ref{fig: ODE_minor_instabilitys} to a fraction of $10^{-2}$ of the largest real value displayed when zooming on small areas.

%All of the following pictures have been obtained by plugging the above described Butcher tableaux into the stability functions, using property \eqref{eq:ImExStabilityFunction} and evaluating them numerically.\\
%While the explicit cases have already been investigated in \cite{Han_Veiga_2021}, we will investigate the implicit stability regions for the first time at this point. 
%A slight extension to these results can be found in the appendix.
To distinguish the different orders, we apply different colors and line styles to the outer and inner bounds according to the legend that are plot next to each stability region plot.
%in figure~\ref{fig: legend_standard}.
%\begin{figure}[!h]
%	\centering
%		\begin{minipage}[t]{0.4\textwidth}\centering
%			\includegraphics[height=0.4\textwidth, trim={469 280 30 23}, clip]{pdf/odepics/colors_a-d_new_2-7.pdf}
%			\includegraphics[height=0.4\textwidth, trim={461 280 30 23}, clip]{pdf/odepics/colors_a-d_new_8-13.pdf}
%		\end{minipage}
%		\caption{Legend for the classic RK and the IMEX stability regions}
%	\label{fig: legend_standard}
%\end{figure}
\subsection{Implicit schemes}
In the following, we plot the contour lines of the bounds of the stability regions of various implicit methods. We start from ImDeC and ImADER schemes in Figure~\ref{fig: ODEIMDeCADER}. 
Clearly, all these stability regions are unbounded, but they are not all A-stable, as we will see soon. Moreover, we can observe a great variability changing the scheme or the nodes, in opposition to the explicit case \cite{Han_Veiga_2021}. In most of the cases, ImADER have larger stability regions than ImDeC.
%We already see a different behavior in stability, contrary to the explicit counterparts. 
%We also observe a slightly different behavior while changing the interpolation nodes to Gauss-Lobatto points in figure~\ref{fig: ODEIMDeCADERGLB}. Nevertheless, all the results coincide with what expected from implicit Runge-Kutta schemes and all the methods seem to be A-stable. Furthermore, the stability regions of the ImADER differ to the ImDeC for both types of nodes, differently from the explicit case \cite{Han_Veiga_2021}.  
\begin{figure}
	\centering
	\includegraphics[width=0.465\textwidth,trim={215 340 32 22}, clip]{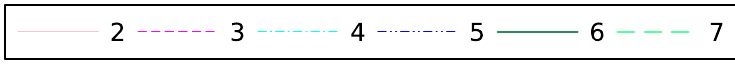}\!\!
	\includegraphics[width=0.515\textwidth,trim={179 340 30 22}, clip]{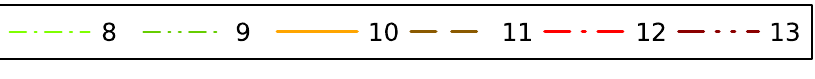}\\
%	\begin{minipage}[t]{0.32\textwidth}
%		\centering
%		\includegraphics[width=\textwidth, trim={0 0 0 22}, clip]{pdf/odepics/IMDeC_eq_ord13-crop.pdf}\\
%		ImDeC eq
%	\end{minipage}
%	\begin{minipage}[t]{0.32\textwidth}
%		\centering
%	\includegraphics[width=\textwidth, trim={0 0 0 22}, clip]{pdf/odepics/IMsDeC_eq_ord13-crop.pdf}\\
%	ImsDeC eq
%	\end{minipage}
%	\begin{minipage}[t]{0.32\textwidth}
%		\centering
%		\includegraphics[width=\textwidth, trim={0 0 0 22}, clip]{pdf/odepics/IMADER_eq_ord13-crop.pdf}\\
%		ImADER eq
%	\end{minipage}\\
%	\includegraphics[width=0.08\textwidth, trim={491 140 30 23}, clip]{pdf/odepics/colors_a-d_new_2-13_no_order.pdf}\\
	\begin{minipage}[t]{0.32\textwidth}
		\centering
		\includegraphics[width=\textwidth, trim={0 0 0 0}, clip]{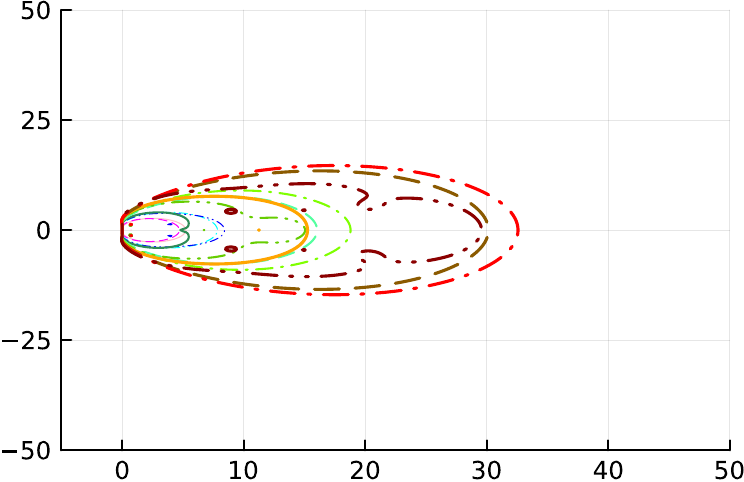}\\
		ImDeC GLB
	\end{minipage}
	\begin{minipage}[t]{0.32\textwidth}
		\centering
	\includegraphics[width=\textwidth, trim={0 0 0 0}, clip]{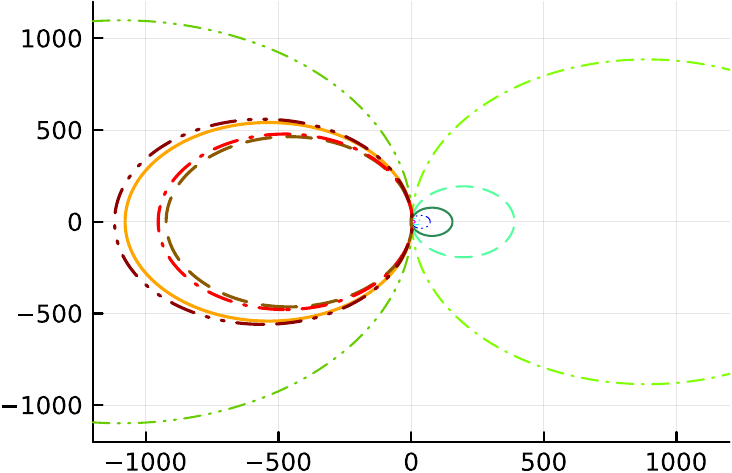}\\
	ImsDeC GLB
	\end{minipage}
	\begin{minipage}[t]{0.32\textwidth}
		\centering
		\includegraphics[width=\textwidth, trim={0 0 0 0}, clip]{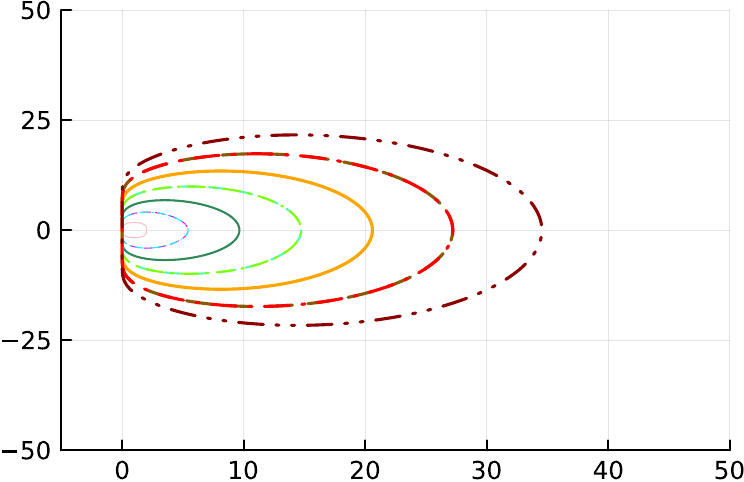}\\
		ImADER GLB
	\end{minipage}
	\caption{Implicit DeC (left), sDeC (center) and ADER (right) with equispaced (top) and Gauss-Lobatto (bottom) nodes for orders 2 to 13}
	\label{fig: ODEIMDeCADER}
\end{figure}
%\begin{figure}
%	\centering
%	\includegraphics[width=0.08\textwidth, trim={491 180 30 23}, clip]{pdf/odepics/colors_a-d_new_2-13_no_order.pdf}
%	\caption{Implicit sDeC for orders 2 to 13}
%	\label{fig: ODEIMsDeC}
%\end{figure}

\begin{figure}
	\centering
	\begin{minipage}[t]{0.33\textwidth}
		\centering
		\includegraphics[width=\textwidth, trim={0 0 0 0}, clip]{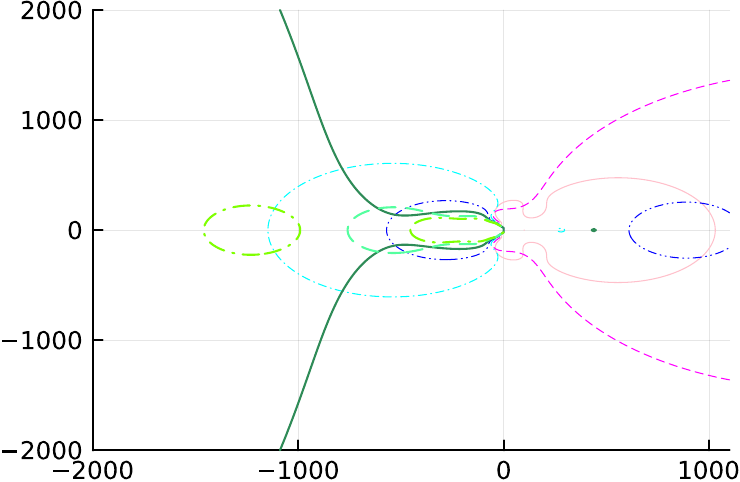}\\
		ImsDeC eq
	\end{minipage}
	\begin{minipage}[t]{0.33\textwidth}
		\centering
		\includegraphics[width=\textwidth, trim={0 0 0 0}, clip]{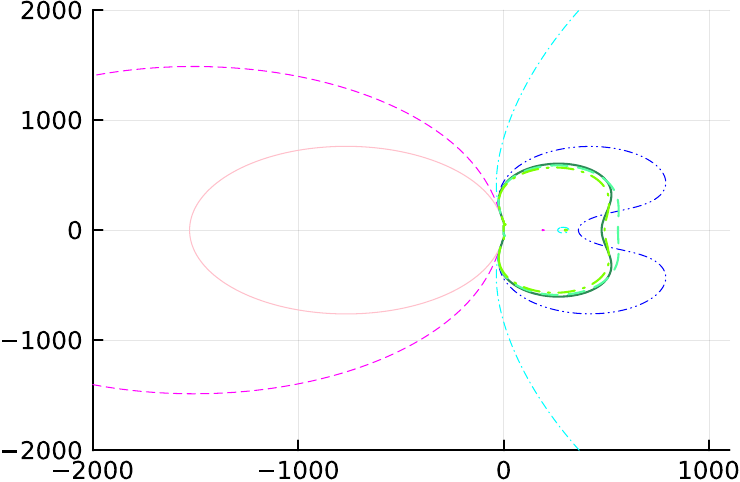}\\
		ImsDeC GLB
	\end{minipage}
	\includegraphics[width=0.1\textwidth, trim={491 230 30 23}, clip]{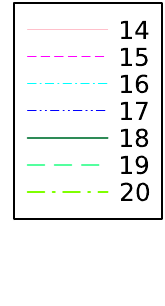}
	\caption{Implicit sDeC for orders 14 to 20}
	\label{fig: exaImsDeC_high}
\end{figure}
The stability regions for the implicit sDeC (ImsDeC) are also shown in Figure~\ref{fig: ODEIMDeCADER} and, surprisingly, do not behave like the other methods. 
Up to a certain order (i.e. sDeC8 with Gauss-Lobatto), the stability regions are unbounded and \textit{seem} A-stable, but for higher orders, we lose this property, obtaining large, but finite, stability regions. 
This behavior is not uniform and, at certain orders, the stability region will be unbounded again, as for example shown in Figure~\ref{fig: exaImsDeC_high} for very high order ImsDeC. 
%
%A legend for these plots can be found in figure~\ref{fig: legend_high_order}.
%\begin{figure}[!h]
%	\centering
%	\begin{minipage}[t]{0.18\textwidth}
%		\includegraphics[width=\textwidth, trim={461 270 30 23}, clip]{pdf/odepics/colors_a-d_new_14-20.pdf}
%	\end{minipage}
%	\caption{Legend for the higher order Im sDeC methods.}
%	\label{fig: legend_high_order}
%\end{figure}
%
%\begin{table}
%	\centering
%	\caption{Table with orders of the ImsDeC methods that have a bounded stability region on the negative half-plane}
%	\label{tab:imsDeC_bounded}
%	\begin{tabular}[h]{|c|c|}
%		\hline
%		ImsDeC nodes specification & orders of the method\\
%		\hline
%		Gauss-Lobatto & $9, 10, 11, 12, 13, 14, 15 $\\
%		\hline	equispaced & $12, 13, 16, 17, 18, 19, 20 $\\
%		\hline
%	\end{tabular}
%\end{table}
%Remark that the outlines of these stability regions can display inner or outer bounds as described before. Thereby, the shapes of these bounded stability regions mostly seem to be elliptic shaped. 
%A detailed list of the bounded methods until order 20 is given in Table~\ref{tab:imsDeC_bounded}. 
%Why we use the term unbounded instead of A-stable will be discussed shortly afterwards.\\
In detail, the methods with bounded stability region up to order 20 are the ImsDeC with Gauss--Lobatto nodes with orders 9, 10, 11, 12, 13, 14 and 15 and with equispaced nodes with orders 12, 13, 16, 17, 18, 19 and 20. 
For the sDeC, we can conclude that the choice of an implicit version does not guarantee an unbounded stability region. Nevertheless, even these implicit sDeC methods have larger stability regions than their explicit counterparts and therefore may be applicable to mildly stiff problems.
We notice again that this odd loss of stability in the left half plane could not be found in the ImDeC and ImADER methods. We checked it numerically up to order 50.

\begin{figure}
	\centering
	\includegraphics[width=0.465\textwidth,trim={215 340 33 22}, clip]{pdf/odepics/colors_a-d_new_horiz_2-7_no_order.pdf}\!\!
	\includegraphics[width=0.515\textwidth,trim={179 340 30 22}, clip]{pdf/odepics/colors_a-d_new_horiz_8-13_no_order.pdf}\\
	\begin{minipage}[t]{0.325\textwidth}
		\includegraphics[width=\textwidth]{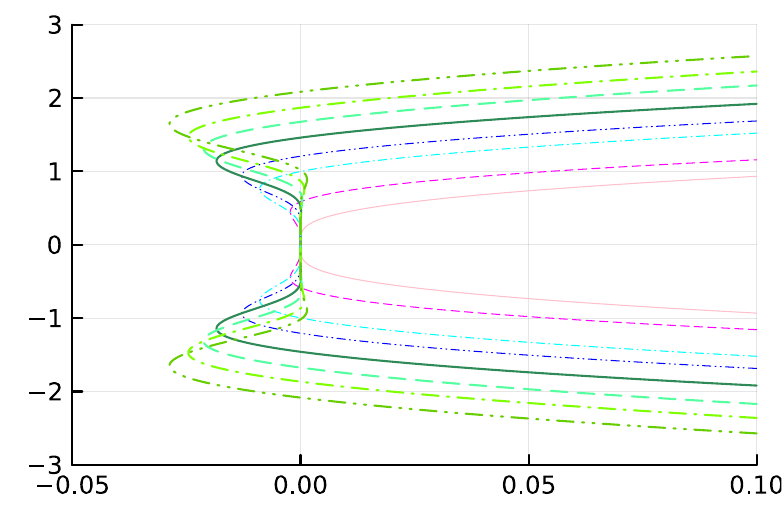}
		\centering
		ImDeC GLB from order 2 to 9
	\end{minipage}
%	\begin{minipage}[t]{0.325\textwidth}
%		\includegraphics[width=\textwidth]{pdf/odepics/IMEXDeC_equispaced_zoom.pdf}
%		\centering
%		ImDeC eq from order 2 to 9
%	\end{minipage}
%	\includegraphics[width=0.08\textwidth, trim={491 140 30 23}, clip]{pdf/odepics/colors_a-d_new_2-13_no_order.pdf}\\	
	\begin{minipage}[t]{0.325\textwidth}
		\includegraphics[width=\textwidth]{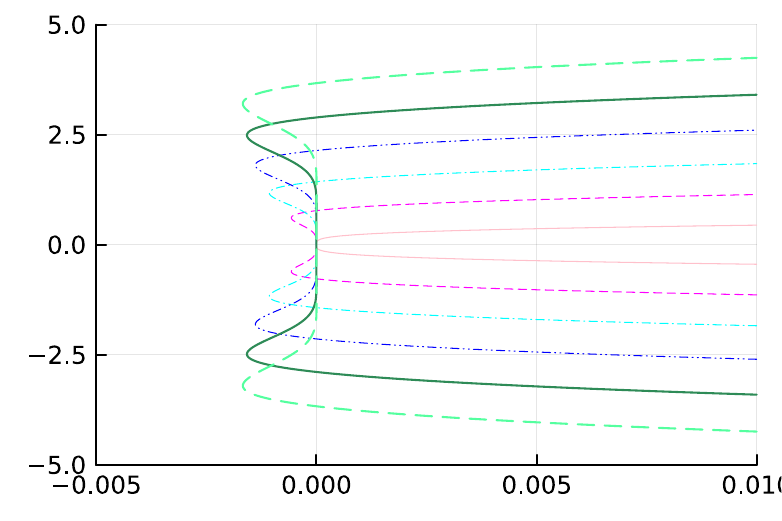}
		\centering
		ImsDeC GLB from orders 2 to 7
	\end{minipage}
	\begin{minipage}[t]{0.325\textwidth}
		\includegraphics[width=\textwidth]{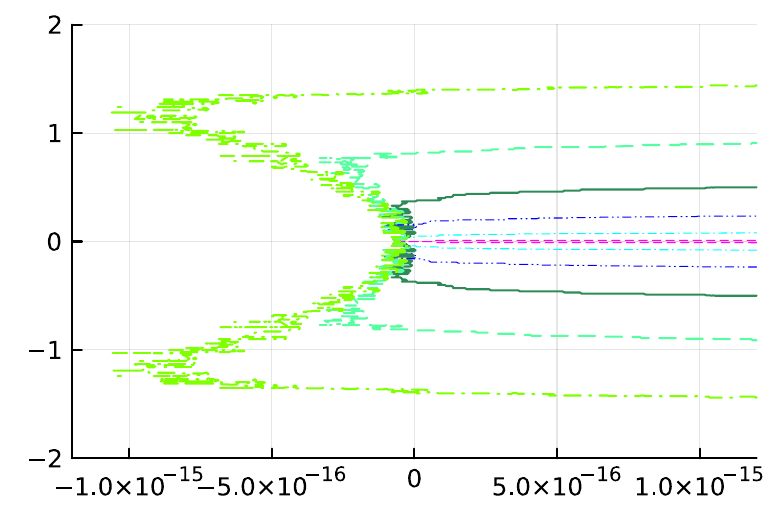}
		\centering
		ImADER eq from orders 3 to 8
	\end{minipage}
	\caption{Zoomed stability region of various implicit schemes}
	\label{fig: ODE_minor_instabilitys}
\end{figure}

Taking a closer look at the implicit methods, we additionally detect some minor instability regions on the negative half-plane, see Figure~\ref{fig: ODE_minor_instabilitys}.
It turns out that these instabilities appear for all ImDeC and ImsDeC methods of orders larger than 2 and both types of nodes. We display the sDeC only with GLB nodes, as the equispaced nodes version is very similar to that.
For a very small scale, the same can be seen for the ImADER methods with equispaced nodes of orders at least larger than 4, displayed on the bottom right of Figure~\ref{fig: ODE_minor_instabilitys}. 
Notice that the sizes of the unstable regions are close to machine precision, which results in non-smooth boundaries and it is unclear if the ImADER with equispaced nodes are not A-stable or the visualization of the unstable area is given by machine precision errors.

As proved in Section~\ref{subsec:ADER_A-stability}, also numerically we observe that the ImADER methods with Gauss-Lobatto nodes are A-stable for all orders. 

%Hence, we classify the implicit methods as A-stable, \textit{almost A-stable}, when the stability region is unbounded and it \textit{almost} includes the whole left half--plane, or with bounded stability region.
%Remark that for \textit{almost A-stable} methods these minor instabilities do not influence the behavior of the scheme on many stiff problems. Nevertheless, when the eigenvalues $\lambda$ of the system are (almost) purely imaginary, they might encounter instabilities for some discretizations.
%, so that for certain step-sizes $\Delta t$ it is possible to get a $z=\lambda \Delta t$ located in the unstable region.

Summarizing, we can categorize our methods in 3 different classes: 
\begin{itemize}
	\item The A-stable schemes: all ImADER GLB and all second order implicit methods;
	\item The \textit{almost A-stable} schemes, when the stability region is unbounded and it \textit{almost} includes the whole left half--plane: high order ImDeC, some ImsDeC and ImADER equispaced;
	\item The bounded stability schemes: some ImsDeC.
\end{itemize}

Remark that for \textit{almost A-stable} methods these minor instabilities do not influence the behavior of the scheme on many stiff problems. Nevertheless, when the eigenvalues $\lambda$ of the system are (almost) purely imaginary (typical for high order advection operators), they might encounter instabilities for some discretizations.

\subsection{IMEX schemes}
To study the stability of IMEX schemes, we will use the RK stability function 
%We will now evaluate our constructed IMEX methods, by using their Butcher tableaux 
%\begin{equation}\label{eq:ImExButcherTableu}
%\begin{array}{c|c}
%c & A\\\hline
%& b^T
%\end{array}, \qquad
%\begin{array}{c|c}
%c & \hat{A}\\\hline
%& \hat{b^T}
%\end{array}.
%\end{equation}
%to receive their stability functions
\begin{equation}\label{eq:ImExStabilityFunction}
R(z_I, z_E)=1+\left(z_I\vec{b}^T+z_E\vec{\hat{b}}^T\right)\vec{u}=1+\left(z_I\vec{b}^T+z_E\vec{\hat{b}}^T\right)\left(\mat{Id}-z_I\mat{A}-z_E\mat{\hat{A}}\right)^{-1} \vec{1}
\end{equation}
that uses the matrices defined by the Butcher tableau of an IMEX RK \eqref{eq:IMEX_Butcher}.
Remark that the standard approach of A-stability cannot be used anymore.
Indeed, the region of absolute stability $$S=\left\{(z_I,z_E)\in \mathbb{C}^2 \ : \ \lvert{R(z_I,z_E)}\rvert\le1\right\}$$ lays in a larger space, with respect to classical RK schemes, therefore, its study, computation and visualization are challenging.
%Like in the classic case, we do not need (and do not want) the whole region $\mathbb{C}^2$ for satisfying results besides the issue of problematic visualization. Here arises the question which subsets of stability we want to consider. \\
Hence, we need to rely on some simplifications.
In \cite{minion2003dec}, Minion simplifies the Dahlquist equation by imposing
\begin{equation*}
\lambda_I \in \mathbb{R}, \quad \lambda_E =i\lambda_E', \ \lambda_E' \in \mathbb{R}.
\end{equation*} 
This procedure neglects respectively the imaginary or real part of the coefficients in the Dahlquist equation to display a two-dimensional region. This idea is lead by classical PDE discrete operators, where typically the diffusion is symmetric negative definite, while the advection is mainly with imaginary eigenvalues. 
A second approach where for each $\lambda_E$ the A-stability is required for the implicit part of the scheme was originally studied in \cite{zhong1996additive,caflisch1997uniformly} and formalized in \cite{liotta2000central}. Another approach studies, instead, the stability for each $\lambda_I$ requiring at least the stability region of the explicit Euler method to the explicit part \cite{Hundsdorfer}.
We collect these definitions of stability region in the following.
\begin{definition}[Stability regions]
	Consider the modified test equation with stability function \eqref{eq:ImExStabilityFunction}. Then, we define multiple approaches for IMEX stability regions by
	\begin{itemize}
		\item $S:=\left\{(z_I,z_E)\in \mathbb{C}^2 \ :\ \lvert{R(z_I,z_E)}\rvert\le1\right\}$ (Region of absolute stability),
	\item  $\mathcal{D}_M:= \left\{(z_I,z_E)\in \mathbb{R}^2 \ :\ \lvert{R(z_I,iz_E)}\rvert\le1\right\}$  (Minion's stability region) \cite{minion2003dec}, 
	\item  $\mathcal{D}_0:=\left\{z_E \in \mathbb{C}\ :\ \lvert{R(z_I,z_E)}\rvert \le 1 \textrm{ for any } z_I \in \mathbb{C}^- \right\}$ \cite{liotta2000central},
	\item $\mathcal{D}_1:=\left\{z_I \in \mathbb{C}\ :\ \lvert{R(z_I,z_E)}\rvert \le 1 \textrm{ for any } z_E \in \mathcal{S}_0 \right\}$ \cite{Hundsdorfer},
	\end{itemize}
	where $\mathcal{S}_0=\left\{z_E \in \mathbb{C}\ :\ \lvert 1+z_E\rvert \le 1 \right\}$ is the stability region of the explicit Euler method. 
\end{definition}
%Remark that $\mathbb{C}^-$ is the smallest stability region to fulfill the A-stability. 
%Hundsdorfer's definitions \cite{Hundsdorfer} $\mathcal{D}_0$ and $\mathcal{D}_1$ give statements about the explicit or implicit part of the equation, respectively assuming that the other part fulfills these known stability properties.
$\mathcal D_0$ is a very strict condition of IMEX stability, in particular for the considered high order schemes. 
%We will refer to $\mathcal{D}_M$ as Minion's stability approach and to $\mathcal{D}_0$ /\mathcal{D}_1$ as Hundsdorfer's stability approach.
Theoretically, the terms A-stability and A($\alpha$)-stability may be applied for all 3 of these subsets of $\mathbb{C}$ analogously to the classical cases, so we will make use of this terminology too. 

\subsubsection*{$\mathcal{D}_M$ stability region}
\begin{figure}
	\centering
	\includegraphics[width=0.465\textwidth,trim={215 340 32 22}, clip]{pdf/odepics/colors_a-d_new_horiz_2-7_no_order.pdf}\!\!
	\includegraphics[width=0.515\textwidth,trim={179 340 30 22}, clip]{pdf/odepics/colors_a-d_new_horiz_8-13_no_order.pdf}\\
%	\begin{minipage}[t]{0.32\textwidth}
%		\centering
%		\includegraphics[width=\textwidth, trim={0 0 0 22}, clip]{pdf/odepics/Minion_IMEXDeC_eq_ord13-crop.pdf}
%		IMEXDeC eq orders 2 to 13
%	\end{minipage}	
%	\begin{minipage}[t]{0.32\textwidth}
%		\centering
%		\includegraphics[width=\textwidth, trim={0 0 0 22}, clip]{pdf/odepics/Minion_IMEXsDeC_eq_ord13-crop.pdf}
%		IMEXsDeC eq orders 2 to 13
%	\end{minipage}
%	\begin{minipage}[t]{0.32\textwidth}
%		\centering
%		\includegraphics[width=\textwidth, trim={0 0 0 22}, clip]{pdf/odepics/Minion_IMEXADER_eq_ord6-crop.pdf}
%		IMEXADER eq orders 2 to 6
%	\end{minipage}\\
	\begin{minipage}[t]{0.32\textwidth}
		\centering
		\includegraphics[width=\textwidth, trim={0 0 0 0}, clip]{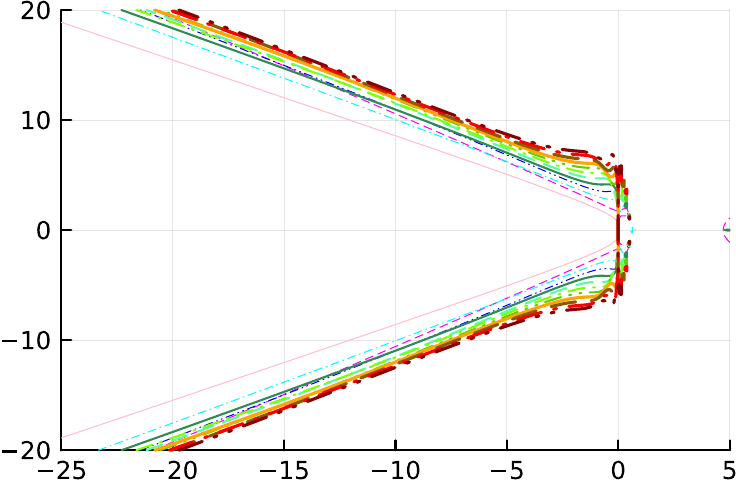}
		IMEXDeC GLB orders 2 to 13
	\end{minipage}
	\begin{minipage}[t]{0.32\textwidth}
		\centering
		\includegraphics[width=\textwidth, trim={0 0 0 0}, clip]{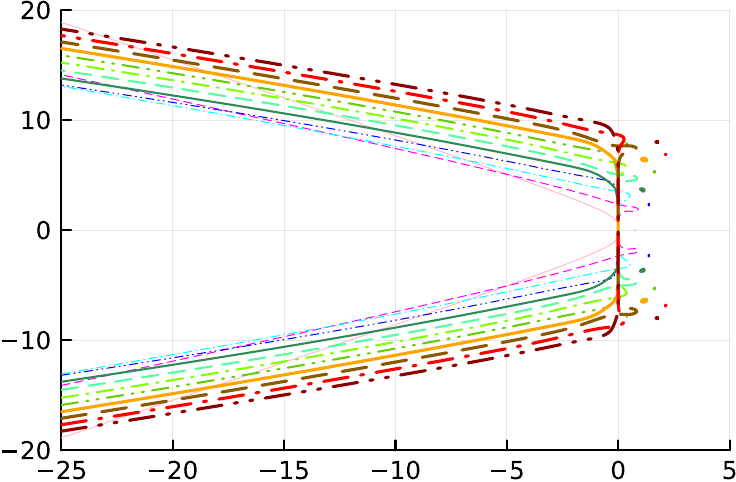}
		IMEXsDeC GLB orders 2 to 13
	\end{minipage}
	\begin{minipage}[t]{0.32\textwidth}
		\centering
		\includegraphics[width=\textwidth, trim={0 0 0 0}, clip]{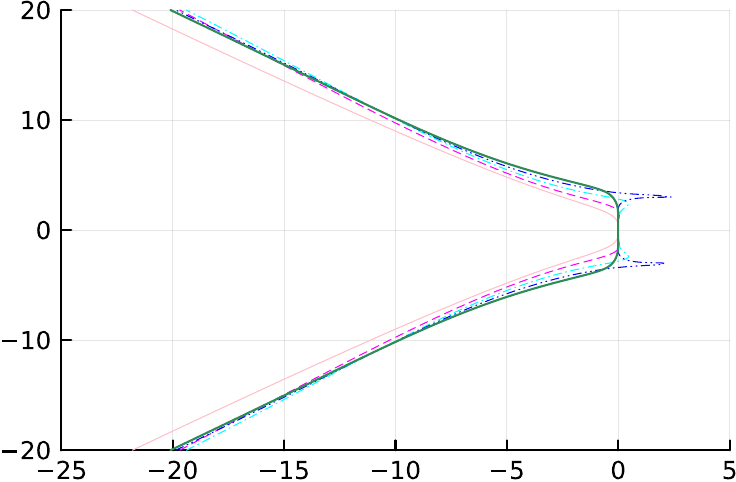}
		IMEXADER GLB orders 2 to 6
	\end{minipage}
	\caption{Minion's stability region for IMEX DeC (left), sDeC (center) and ADER (right) with equispaced (top) and GLB (bottom) nodes}
	\label{fig: ODEIMEX}
\end{figure}
We recall that the plots of the stability regions have very different meaning according to the chosen approach. 
%Considering the previously developed IMEX methods, we want to recall that the view on the regions of stability changes, depending on the approach we choose. Firstly, we want to consider stability like in \cite{minion2003dec}. 
%It is possible to write the method used there into the IMEX sDeC method, so we expect the same results as him in \cite[fig. 4.1]{minion2003dec}. 
Starting from Minion's approach \cite{minion2003dec}, we evaluate the IMEX stability function \eqref{eq:ImExStabilityFunction} numerically to calculate the respective stability regions. 

For the IMEX DeC, we can observe in Figure~\ref{fig: ODEIMEX} (left) that the choice of nodes change the regions on some details but the qualitative behavior is the same. We can also conclude on an $A(\alpha)-$stability for approximately $\alpha=35^\circ$.

Going on to the IMEX ADER, we can see in Figure~\ref{fig: ODEIMEX} (right) a similar behavior, even if the stability regions differ in small details, we observe $A(\alpha)-$stability for at least $\alpha=35^\circ$.

Finally, for the IMEX sDeC method in Figure~\ref{fig: ODEIMEX} (center) we see a slightly different behavior, still resulting in an $A(\alpha)-$stability, but for significantly smaller angles, approximately $\alpha=18^\circ$. 
%Remark that this is significantly lower than for both the IMEX ADER and IMEX DeC. 
Nevertheless, the result on the bottom center in Figure~\ref{fig: ODEIMEX} with GLB nodes coincides with the one in \cite{minion2003dec}, as expected. 
%So we can conclude that the IMEX variations for all of these methods also have a somehow similar $A(\alpha)-$stability for some $\alpha>10^\circ$, at least for this view on the stability for IMEX methods.\\
It is also noticeable that the IMEX sDeC stability region of order 2 is $A(\alpha)$-stable with larger $\alpha$ as it coincides with the IMEX DeC2.

\subsubsection*{$\mathcal{D}_0$ stability region}

Now, we want to evaluate $\mathcal{D}_0$ stability for our IMEX methods. 
We want to emphasize that the requirements here are stricter than in Minion's approach. 
Indeed, for $\mathcal{D}_0$ we require the method to be at least fully A-stable for the implicit part and we look at the stability of the explicit part.
The IMEX DeC and IMEX sDeC have $\mathcal{D}_0=\emptyset$ and this is probably related to the fact that their implicit counterpart is not A-stable.
For the IMEX ADER, only few orders have non-empty $\mathcal{D}_0$ stability region. In Figure~\ref{fig: HundsdorferD0_IMEXADER}, we show the few stability regions, which eventually vanish when increasing the order of accuracy.
\begin{figure}
	\centering
	\begin{minipage}[t]{0.4\textwidth}
		\includegraphics[width=\textwidth]{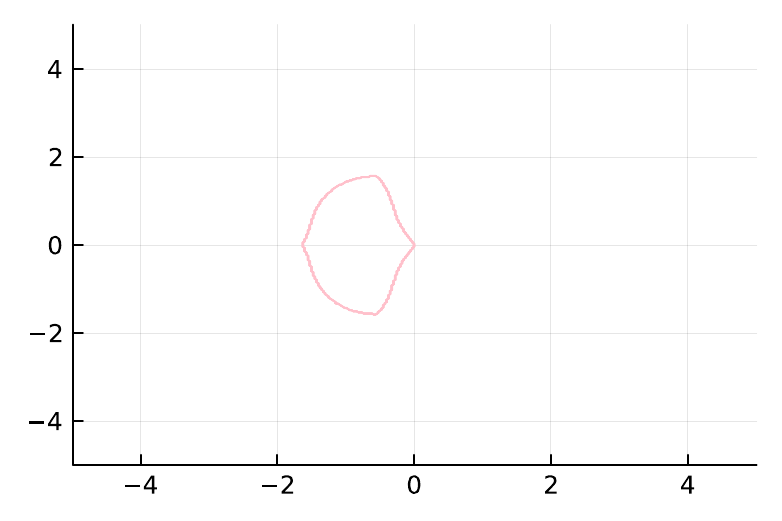}
		\centerline{Order 2 with equispaced nodes}
	\end{minipage}\,\,
	\begin{minipage}[t]{0.4\textwidth}
		\includegraphics[width=\textwidth]{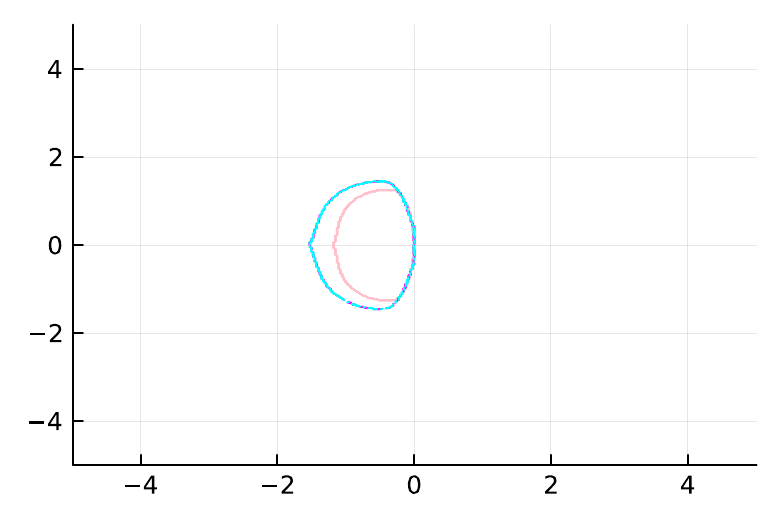}
		\centerline{Orders 2 to 4 with Gauss-Lobatto nodes}
	\end{minipage}
	\includegraphics[width=0.12\textwidth, trim={491 230 30 23}, clip]{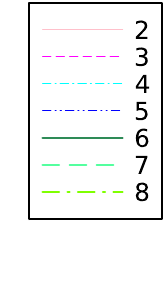}
	\caption{$\mathcal{D}_0$ Stability regions for IMEX ADER. The smaller stability region displays order 2, while the larger one displays order 3 and 4 (right)}
\label{fig: HundsdorferD0_IMEXADER}
\end{figure}

\subsubsection*{$\mathcal{D}_1$ stability region}
We plot the $\mathcal{D}_1$ stability regions for DeC and sDeC methods in Figures~\ref{fig: HundsdorferD1_IMEX}, where we require the explicit part to cover the stability region of the explicit Euler method and we look at the stability of the implicit part.
Contrary to the $\mathcal{D}_0$ cases, we observe non-empty, limited regions of stability for every order for the IMEX DeC methods. Moreover, there is no regularity in their shape and their size grow significantly as the order of accuracy increases.
Notice that the plots do not show the full stability regions of higher orders, for example for orders 6, 7, and 8 with equispaced nodes, but they are anyway bounded regions. \\
For the IMEX sDeC, see Figure~\ref{fig: HundsdorferD1_IMEX} (center), we observe some remarkable differences. In the case of equispaced nodes, even orders just show the small bounded stability regions in the negative half-plane nearby the origin, odd orders smaller than 6 show large stability regions, while they are unstable starting from order 7. 
Also in the GLB case, we do not observe much regularity. 
We notice that the largest stability region is obtained for order 5, while, for higher orders, the stability region almost fit in the plot.
\begin{figure}
	\centering
	\includegraphics[width=0.515\textwidth,trim={158 340 30 22}, clip]{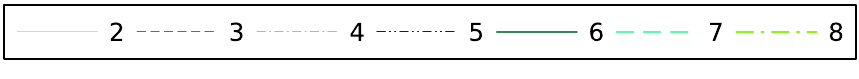}\\
	\begin{minipage}[t]{0.32\textwidth}
		\includegraphics[width=\textwidth]{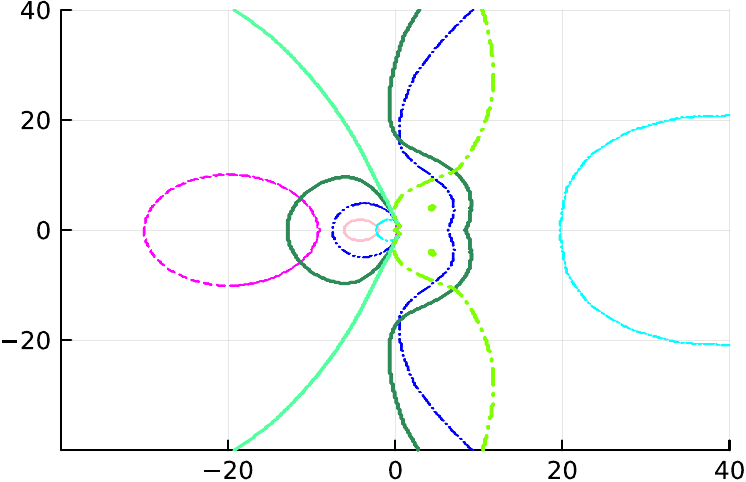}
		\centerline{IMEX DeC eq}
	\end{minipage}
	\begin{minipage}[t]{0.32\textwidth}
	\includegraphics[width=\textwidth]{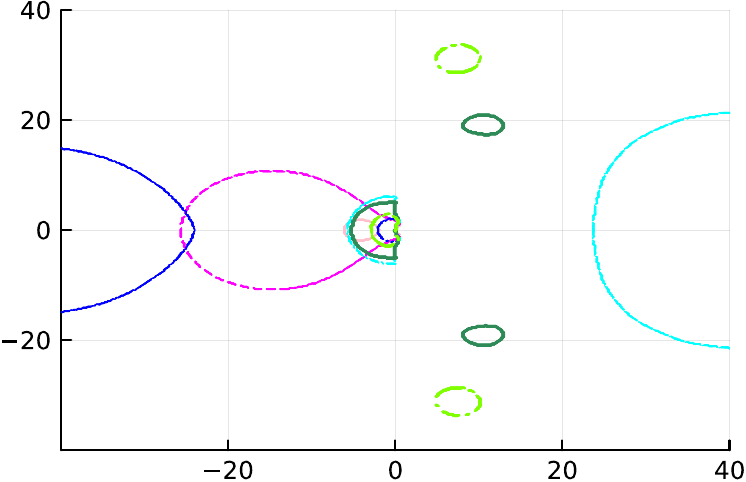}
	\centerline{IMEX sDeC eq}
	\end{minipage}
	\begin{minipage}[t]{0.32\textwidth}
	\includegraphics[width=\textwidth]{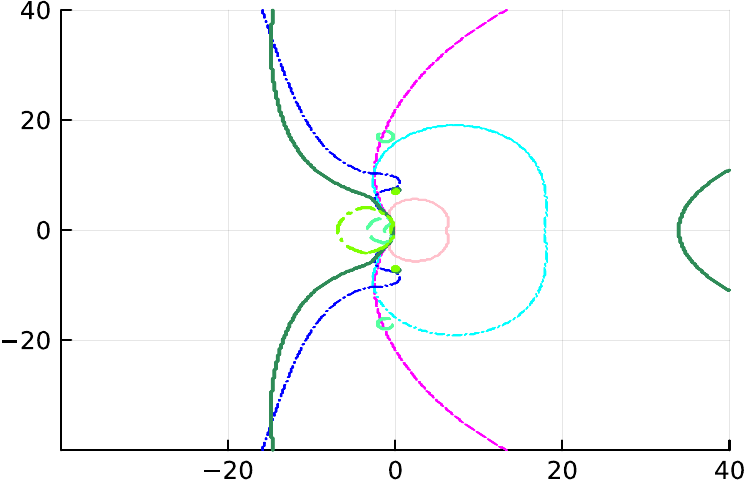}
	\centerline{IMEX ADER eq}
	\end{minipage}\\
	\begin{minipage}[t]{0.32\textwidth}
	\includegraphics[width=\textwidth]{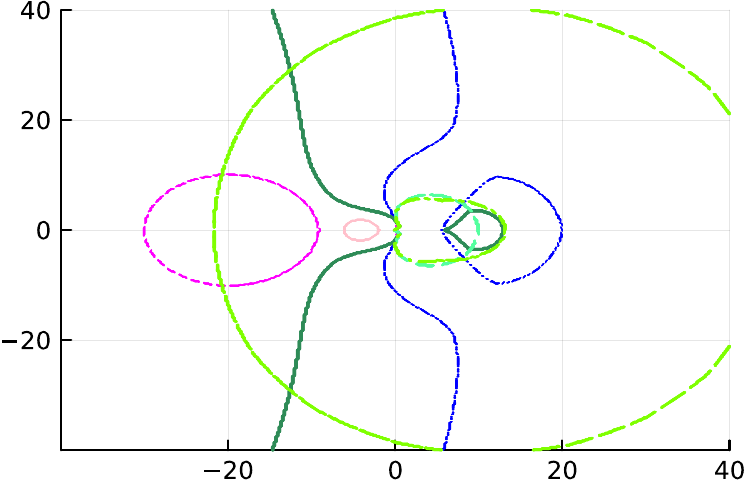}
	\centerline{IMEX DeC GLB}
	\end{minipage}
	\begin{minipage}[t]{0.32\textwidth}
	\includegraphics[width=\textwidth]{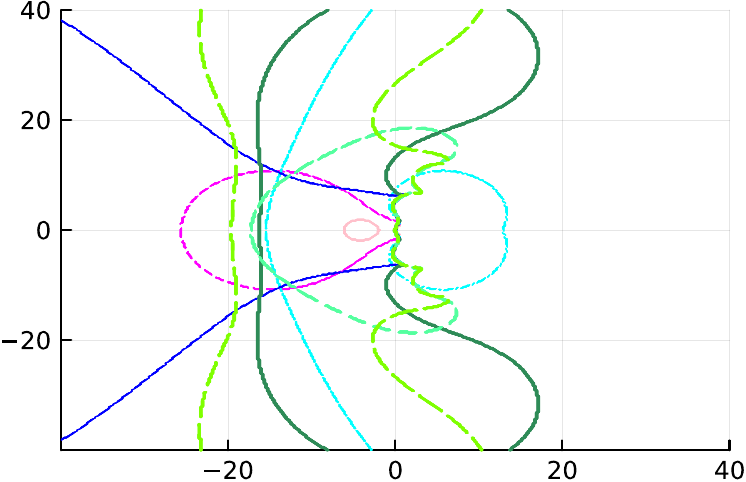}
	\centerline{IMEX sDeC GLB}
	\end{minipage}
	\begin{minipage}[t]{0.32\textwidth}
	\includegraphics[width=\textwidth]{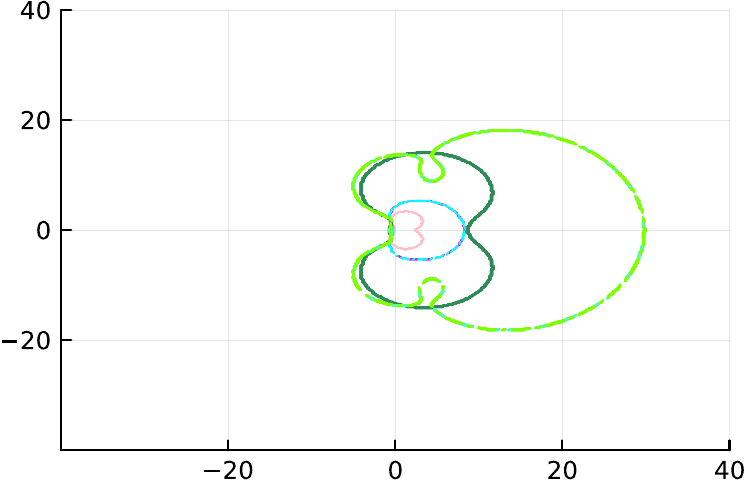}
	\centerline{IMEX ADER GLB}
	\end{minipage}
	\caption{$\mathcal{D}_1$ Stability Region for IMEX DeC (left), sDeC (center) and ADER (right) with equispaced (top) and GLB (bottom) nodes: orders 2 to 8}
\label{fig: HundsdorferD1_IMEX}
\end{figure}

%\begin{figure}
%	\centering
%	\begin{minipage}[t]{0.45\textwidth}
%		\includegraphics[width=\textwidth]{pdf/odepics/ImExD1_IMEXDeC_subtimesteps_equispaced_range=40.pdf}
%		\centerline{equispaced nodes}
%	\end{minipage}
%	\begin{minipage}[t]{0.45\textwidth}
%		\includegraphics[width=\textwidth]{pdf/odepics/ImExD1_IMEXDeC_subtimesteps_gaussLobatto_range=40.pdf}
%		\centerline{Gauss-Lobatto nodes}
%	\end{minipage}
%	\includegraphics[width=0.08\textwidth, trim={491 180 30 23}, clip]{pdf/odepics/colors_a-d_new_2-8_no_order.pdf}
%	\caption{$\mathcal{D}_1$ Stability Region for IMEX sDeC: orders 2 to 8}
%\label{fig: HundsdorferD1_IMEXsDeC}
%\end{figure}
In Figure~\ref{fig: HundsdorferD1_IMEX} (right), we show the results for the IMEX ADER methods. 
We note that most of the methods fulfill nearly A-stability by almost covering the negative half-plane. %, but most of it except for relatively small, limited sets. 
We highlight that for the equispaced case we do not show the full outreach of the stability regions. 
While in most of the cases, for the almost A-stable cases, these contour lines represent the inner bounds of unlimited stability regions, in the cases of order 5 and 8 we just have large, limited stability regions, as we also could observe for example in the $\mathcal{D}_1$ IMEX DeC case for equispaced nodes. 
Therefore, it seems like we can not guarantee this almost A-stability for the IMEX ADER, but just for some of the orders of accuracy.
%\begin{figure}
%	\centering
%	\begin{minipage}[t]{0.45\textwidth}
%		\includegraphics[width=\textwidth]{pdf/odepics/ImExD1_IMEXADER_equispaced_range=40.pdf}
%		\centerline{equispaced nodes}
%	\end{minipage}
%	\begin{minipage}[t]{0.45\textwidth}
%		\includegraphics[width=\textwidth]{pdf/odepics/ImExD1_IMEXADER_gaussLobatto_range=40.pdf}
%		\centerline{Gauss-Lobatto nodes}
%	\end{minipage}
%	\includegraphics[width=0.08\textwidth, trim={491 180 30 23}, clip]{pdf/odepics/colors_a-d_new_2-8_no_order.pdf}
%	\caption{$\mathcal{D}_1$ Stability Region for IMEX ADER: orders 2 to 8}
%\label{fig: HundsdorferD1_IMEXADER}
%\end{figure}

%Remark that for the $\mathcal{D}_0$ case, an order of magnitude of the explicit Euler stability region can be considered as good as it gets controlled by the non-stiff term, while for the $\mathcal{D}_1$ case A-stability would be desired due to stiffness. 
Therefore, we can conclude, that some of the IMEX ADER stability regions cover the areas in the complex plane, that we assumed for a stable method in the context of $\mathcal{D}_0$ and $\mathcal{D}_1$, while the DeC and sDeC methods have their limitations with $\mathcal{D}_0= \emptyset$ in every case and bounded stability regions for most of the cases in the scope of $\mathcal{D}_1$. Nevertheless, we need to keep in mind that the set conditions are very strict, so,  the methods might still be applicable to some stiff equations. % \\
These results reflect what we have seen for the respective implicit methods.

	\section{PDE: analysis of advection-diffusion}
	\label{sec: advection_diffusion}
	In this section, we want to extend our stability analysis to the one-dimensional advection-diffusion equation
\begin{equation}\label{eq: A-D_equation}
u_t(x,t) + au_x(x,t) = du_{xx}(x,t), \quad a\ge0, \ d \ge 0, \qquad x \in \Omega \subset \mathbb R,
\end{equation}
where $a$ is the coefficient of the advection term and $d$ the coefficient of the diffusion term, using the von  Neumann stability analysis. %principles.
After the space discretization, we discretize the advection part with an explicit time-integration scheme and the diffusion with an implicit one.
The linear stability of the DeC method in PDE contexts was studied for explicit methods for advection equations with FEM spatial discretizations and various stabilization techniques in \cite{michel2021spectral,michel2023spectral}, while in the IMEX context for FEM methods applied to kinetic models in \cite{torlo2020hyperbolic}.
For the ADER method, a von Neumann stability analysis was applied to the original formulation \cite{titarev2007analysis,dematte2020ader}, but not on the modern version that we are studying. We close this gap with our investigation in the following.
\subsection{Finite Difference discretization}
\label{sec: spatial_discretization}
We apply spatial discretizations to the spatial derivative operators, namely $\partial_x$ and $\partial_{xx}$. 
We consider a uniformed grid $\Omega_{\Delta x}=\left\{x_j \ : \ x_j = x_0 + j\Delta x, \ j \in \{0,\hdots , J\}\right\}$ with periodic boundary conditions and we denote the approximation of $u(x_j)=u(x_j,t)$ by $w_j$.

To discretize the advection term in \eqref{eq: A-D_equation}, i.e., the first spatial derivative $\partial_xu(x)$, we make usage of the stable finite difference stencils introduced in \cite{Iserles1982}. Assume we discretize $\partial_xu$ at $x_j$ by an $[r,s]$-discretization
\begin{equation}\label{eq: r-s_scheme}
\partial^{[r,s]}_{\Delta x}(u(x_j)) = \frac{1}{\Delta x} \sum\limits_{k=-r}^s \alpha_k w_{j+k},
\end{equation}
with $r,\ s$ such that $\alpha_{j-r}, \alpha_{j+s} \neq 0$.
The maximum order we can achieve with an $[r,s]$-discretization is $q=r+s$ and this discretization actually reaches order $q$ and is unique by setting the coefficients in \eqref{eq: r-s_scheme} as
	\begin{align}\label{eq:def_advection_optimal_stencil}
	\alpha_0&=
	\begin{cases}
	\sum\limits_{k=r+1}^s\frac{1}{k}, & s\ge r+1, \\
	0, & s=r, \\
	\sum\limits_{k=s+1}^r\frac{1}{k}, & r\ge s+1, 
	\end{cases} \qquad
	\alpha_k = \frac{(-1)^{k+1}}{k}\cdot \frac{r!s!}{(r+k)!(s-k)!}, \quad -r\le k \le s, \ k\neq0.
	\end{align}
 It is also proven in \cite{Iserles1982} that these so-called $\textit{optimal-order}$ schemes of order $q$ are stable if and only if $s\le r \le s+2$ for $a>0$.
We  involve these stable $\textit{optimal-order}$ schemes into our analysis. 
We introduce upwinding in the choice of the stencils, in particular, we will consider $[r, r +1]$ stencils for odd optimal-order scheme and $[r, r+2]$ stencils for an even optimal-order scheme for the advection part.

For the diffusion term, we will just use a central finite difference discretization of the second spatial derivative $\partial_{xx} u(x)$ given in Table~\ref{tab: CFD-schemes_first_deriv}.
\begin{table}
	\centering
	\caption{Central finite difference discretizations of $\partial_{xx}$ applied onto $w$ centered in $j$  \cite{fornberg_finite_difference}}
	\label{tab: CFD-schemes_first_deriv}
	\small
	\begin{tabular}[h]{|c|c|}
		\hline
		order & finite difference for $\partial_{\Delta x}^2(u(x_j))$\\
		\hline
		2 & $\frac{1}{\Delta x^2}\left(w_{j-1}-2w_j+w_{j+1}\right)$\\
		\hline
		4 & $\frac{1}{\Delta x^2}\left( -\frac{1}{12}w_{j-2} +\frac{4}{3}w_{j-1} -\frac{5}{2}w_j   +\frac{4}{3}w_{j+1}  -\frac{1}{12}w_{j+2}\right)$\\
		\hline
		6 & $\frac{1}{\Delta x^2}\left(\frac{1}{90}w_{j-3} - \frac{3}{20}w_{j-2} + \frac{3}{2}w_{j} - \frac{49}{18}w_{j} + \frac{3}{2}w_{j} - \frac{3}{20}w_{j+2} + \frac{1}{90}w_{j+3}\right)$\\
		\hline
		8 & $\frac{1}{\Delta x^2}\left(
		-\frac{1}{56}w_{j-4} + \frac{1}{420}w_{j-3} - \frac{1}{5}w_{j-2} + \frac{8}{5}w_{j-1} - \frac{205}{72}w_{j} + \frac{8}{5}w_{j+1} - \frac{1}{5}w_{j+2} + \frac{1}{420}w_{j+3} - \frac{1}{56}w_{j+4}
		\right)$\\
		\hline
	\end{tabular}
\end{table}

\subsection{von Neumann analysis}
\label{sec: stability_theory_PDE}
To analyze the stability of the described methods, we make use of the von Neumann stability analysis for linear partial differential equations \cite{leveque2007finite}.
%\subsubsection{von Neumann Stability}
Briefly summarized, we investigate the behavior inside the numerical scheme of the Fourier modes
\begin{equation}\label{eq: fourier_modes}
w^n_j = v^ne^{ikx_j},
\end{equation}
where $w^n_j$ is the discretization of $u(x_j, t_n)$ and $k$ is the wavenumber and we focus on the representation coefficient $v^n$. 
Indeed, $e^{ikx}$ are eigenfunctions of the differential operator $\partial_x$ and therefore for any linear differential operator. 
If we use \eqref{eq: fourier_modes} in our discretized system, we obtain a system of the form
\begin{equation}
\label{eq: ampfactor}
v^{n+1}=G(k, \Delta x, \Delta t, a, d)v^n.
\end{equation}
with the amplification factor $G\in \mathbb{C}$ independent on the mesh point $x_j$.
%Now, to check the stability of the scheme, we need that $\lvert{v^{n+1}}\rvert= \lvert{G(k, \Delta x, \Delta t, a, d)v^n}\rvert \leq \lvert v^n\rvert$. 
Stability means in our context that $\lvert{v^{n+1}}\rvert= \lvert{G(k, \Delta x, \Delta t, a, d)v^n}\rvert \leq \lvert v^n\rvert$ holds.
In practice, we check that $\lvert G (k, \Delta x, \Delta t, a, d) \rvert \le 1$. This implies stability for the related method and due to the  Lax-Richtmeyer theorem \cite{leveque2007finite} convergence can be ensured.
Note that for consistency, we included the parameters $a$ and $d$ into the dependency of $G$ to cover all advection-diffusion equations. \\
% This theory concludes in the following condition: 
%\begin{prop}\textbf{Richtmyer condition}\\
%	A finite difference scheme equipped with a time-stepping method is stable in the stability area $\Lambda$ if and only if there exists a constant $K$, such that
%	\begin{equation*}
%	\lvert G (k, \Delta x, \Delta t, a, d) \rvert \le 1+ K \Delta t
%	\end{equation*}
%	with $(k, \Delta x, \Delta t, a, d)\in \Lambda$. 
%\end{prop}
%Remark that for consistency, we included the parameters $a$ and $d$ into the dependency of $G$ to cover all advection-diffusion equations. When we evaluate the amplification factors numerically later on, we will probe the slightly more strict von Neumann condition $\lvert G (k, \Delta x, \Delta t, a, d) \rvert \le 1$ for stability. \\
%This whole analysis gains its relevance based on the following statement.
%\begin{theorem}\textbf{Lax-Richtmyer}\\
%	A consistent finite difference scheme for a partial differental equation with a well-posed inital value problem is convergent if and only if its stable.
%\end{theorem}
%%%%%%%%%%%%%%%%%
%Displaying Stability
%%%%%%%%%%%%%%%%%
Typically, in order to estimate the stability of the advection-diffusion equation, an analytical study of $G$ in all the parameters should be performed. 
This is not feasible when considering high order schemes as ADER and DeC.
Hence, we will evaluate the amplification factor numerically, similarly to what we did with the stability functions in the ODE case.

Before running all the simulations, we need to understand what are the free variables of the function $G$. First of all, the wavenumbers  should be bounded $k \in \lbrace -n_0-1 , \dots, n_0+1 \rbrace \subset \mathbb Z$ and the maximum wavenumber $n_0+1$ is strongly related with the discretization scale $\Delta x$. Indeed, by  Nyquist–Shannon sampling theorem, only functions with frequency less than $\frac{|x_J-x_0|}{2\Delta x}$ can be represented on our discretization. Hence, we will choose $n_0=10^3$ to take in consideration fine grids.

Then, we have further variables $a,d,\Delta x,\Delta t $ that are actually coupled together: in the advection term ($a\Delta t/\Delta x$) and in the diffusion term ($d\Delta t/\Delta x^2$). We keep this in mind when studying the behavior of $G$, as we can recast few methods to the same coefficients.

\subsubsection{Displaying stability}
It is stated and numerically shown in \cite{TanChenShu_ImEx_Stability, WangShuZhang_LDG1_2015,WangShuZhang_LDG_2016} that several schemes as the local discontinuous Galerkin scheme \cite{WangShuZhang_LDG1_2015,WangShuZhang_LDG_2016} and other finite difference schemes combined with an IMEX RK method are stable if the time step is upper bounded by some $\tau_0$. This $\tau_0$ is proportional to $\frac{d}{a^2}$, i.e., if $\Delta t\le\tau_0= E_0\cdot \frac{d}{a^2}$ for some $E_0>0$.
%%%%%%%%%%%%%%%%%
%C, D, E Approach
%%%%%%%%%%%%%%%%%
Considering the before mentioned parameters $\Delta x, \Delta t, a,d$, we introduce two new coefficients
\begin{equation}
C=\frac{a\Delta t}{\Delta x}, \quad D=\frac{d\Delta t}{{(\Delta x)}^2}.
\end{equation}

Moreover, using the coefficients $C$ and $D$ reveals an equivalent condition to \cite{TanChenShu_ImEx_Stability}, if we assume that the quotient
\begin{equation*}
E:=\frac{C^2}{D}= \frac{ \Delta t ^2 a^2 }{\Delta x^2} \frac{\Delta x^2}{d \Delta t} = \frac{a^2}{d}\Delta t
\end{equation*}
is bounded by some constant $E_0$, indeed,
\begin{equation*}
E= \frac{a^2}{d}\Delta t\le E_0 \quad \Longleftrightarrow \quad \Delta t \le E_0 \cdot \frac{d}{a^2}=\tau_0.
\end{equation*}
Therefore, for a given method solving the advection-diffusion equation, we can rewrite the amplification factor  in \eqref{eq: ampfactor} as
\begin{equation}
	g(k,C,E)=G(k,\Delta x, \Delta t, a,d).
\end{equation}

\begin{definition}[Scheme notation]	\label{defi: notation_pde_method}
	To shorten the notation, we denote the considered method for the advection-diffusion equation by $[\TMM,\NODES,N, A_n,D_m],$ where
	\begin{itemize}
		\item $\TMM$ stands for the respective IMEX time-marching method, among $\DeC$, $\ADER$, $\sDeC$,
		\item $\NODES$ stands for the used quadrature nodes for the \TMM, among {\eq} or $\GLB$,
		\item $N$ stands for the order of the considered time-marching method $\TMM$,
		\item $A_n$ denotes the optimal first derivative stencil of order $n$ defined in \eqref{eq:def_advection_optimal_stencil} used for the advection term,
		\item $D_m$ denotes the central second derivative stencil of order $m$ in Table~\ref{tab: CFD-schemes_first_deriv} used for the diffusion term.
	\end{itemize}
\end{definition}
\begin{example}[Stability of IMEX DeC3 with $A_1$ and $D_2$ operators]
	\begin{figure}[!h]
		\centering
		\begin{minipage}[t]{0.42\textwidth}
			\includegraphics[width=\textwidth]{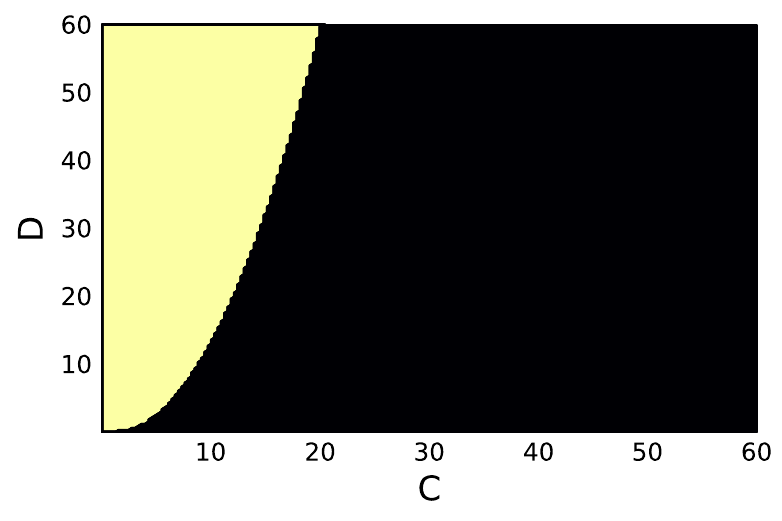}
			\centering
			Coefficients $C$ and $D$
		\end{minipage} 
		\begin{minipage}[t]{0.42\textwidth}
			\includegraphics[width=\textwidth]{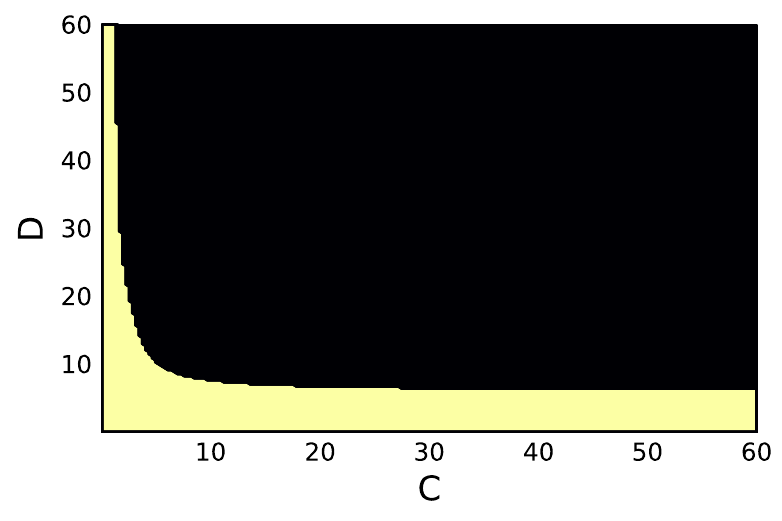}
			\centering
			Coefficients $C$ and $E$
		\end{minipage}
		\caption{Stability areas (yellow) for the $[DeC,eq,3, A_1,D_2]$.}
		\label{fig: exa_ImExDeC3_diff2_adv1}
	\end{figure}
	To give an example of how to study the stability region, we show in Figure~\ref{fig: exa_ImExDeC3_diff2_adv1} the stability areas $\lbrace |g(k,C,E)|\leq 1 ,\, \forall k \in [-n_0-1,n_0+1]\rbrace$ for the $[\DeC,\eq,3, A_1,D_2]$. On the left, we plot the stability area as a function of $C$ and $D$, while on the right as a function of $C$ and $E$. The black area is associated to the unstable area, while the yellow displays the stable region.
	We recognize two sufficient conditions to obtain stability:
	\begin{itemize}
		\item the well known CFL-condition, i.e., if $C$ is lower than some constant $C_0$ only dependent on the method, then the method is stable.
		\item the new numerically obtained condition, designated as the $E_0$-condition: If $E$ is lower than some constant $E_0$ dependent on the method, then the method is stable.
	\end{itemize}	
	The two parameters $C$ and $E$ include all the remaining ones and are enough to characterize the whole scheme.
\end{example}

As we can see in Figure~\ref{fig: exa_ImExDeC3_diff2_adv1}, the unstable area of this specific method seems to be bounded by the linear constraints $C\geq C_0$ and $E\geq E_0$. We will observe numerically that these unstable regions are indeed similarly bounded in most of our methods.
\begin{definition}[Stability parameters $C_0$ and $E_0$]
	Given the amplification factor $g(k,C,E)$ of a discretization of the advection-diffusion equation, we define the two stability parameters $C_0$ and $E_0$ by
	\begin{itemize}
		\item $C_0:=\max_{C\in \mathcal{S}} C$ with $\mathcal{S}=\lbrace C: \lvert g(k,C,E)\rvert \leq 1, \, \forall E>0, \,\forall k \in [-n_0-1,n_0+1] \rbrace$,
		\item $E_0:=\max_{E\in \mathcal{R}}E$ with $\mathcal{R}= \lbrace E: \lvert g(k,C,E)\rvert \leq 1,\, \forall C>0,\,\forall k \in [-n_0-1,n_0+1]\rbrace .$
	\end{itemize}
\end{definition}
Therefore, the strategy we want to follow is to look at the areas of stability by evaluating the amplification factor $\max_k |g(k,C,E)|$ like in Figure~\ref{fig: exa_ImExDeC3_diff2_adv1} and numerically calculating the parameters $C_0$ and $E_0$ for our methods, when possible.

Note, that the condition $E<E_0$ does not depend on $\Delta x$ and avoids CFL restrictions. We should always keep in mind that our numerical evaluations can just cover finite ranges of $C$ and $E$. 
Hence, we checked that the displayed limits for $E$ and $C$ are actually bounds also for larger domains of $C$ and $E$.
Moreover, we also observed that the considered results do not vary much for large values of $n_0$, hence, we set $n_0=10^3$. Due time-efficiency, we will use this value for every evaluation in the von Neumann stability analysis context for the rest of this work. Further, all plots in this section will be evaluated and displayed at $400\times 400$ grid points.
\subsubsection{Numerical analysis}
%To distinguish the different orders, we use different colors to the  bounds according to the legend figure~\ref{fig: legend_a-d}. 
%In the following plots, we remark that the colors are assigned to different orders, indeed, we will vary alternatively the order of the advection method, of the time-marching, of the diffusion method or of all of them. We will specify for each plot what is the object of the study.
%\begin{figure}
%	\centering
%	\includegraphics[width=0.17\textwidth, trim={461 256 30 23}, clip]{pdf/odepics/colors_a-d_new_1-8.pdf}
%	\caption{Legend for the different orders.}
%	\label{fig: legend_a-d}
%\end{figure}
In the following plots, we will study various configurations of methods and the variation of the stability region changing the order of the methods. In particular, we will focus on changing the order of the time discretization only and varying the order of all operators. In the repository \cite{ourrepo}, we include other variations of the order of only the advection or the diffusion operator that we discard here for brevity. Moreover, we will focus on Gauss--Lobatto nodes for the time discretization only mentioning the peculiarity of the equispaced ones that can be found, again, in the repository \cite{ourrepo}.

We start by varying only the time scheme order. 
In Figure~\ref{fig: grp_adv1_diff2}, the stability areas of several methods are shown, i.e., $[\TMM,\NODES,k, A_1,D_2]$ varying the time scheme and the nodes. 
As in Figure~\ref{fig: exa_ImExDeC3_diff2_adv1}, the plotted lines separates the stable region in the lower left part from the unstable region in the upper right side of the $C$-$E$ plane. 
We see a similar behavior for mostly all methods. 
Increasing the order of the time marching method results in a larger stability region and in larger values for $E_0$ and $C_0$. 

\begin{figure}
	\centering
	\includegraphics[width=0.6\textwidth,trim={160 340 30 22}, clip]{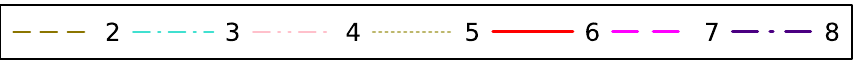}\\		\begin{minipage}[t]{0.32\textwidth}
		\includegraphics[width=\textwidth]{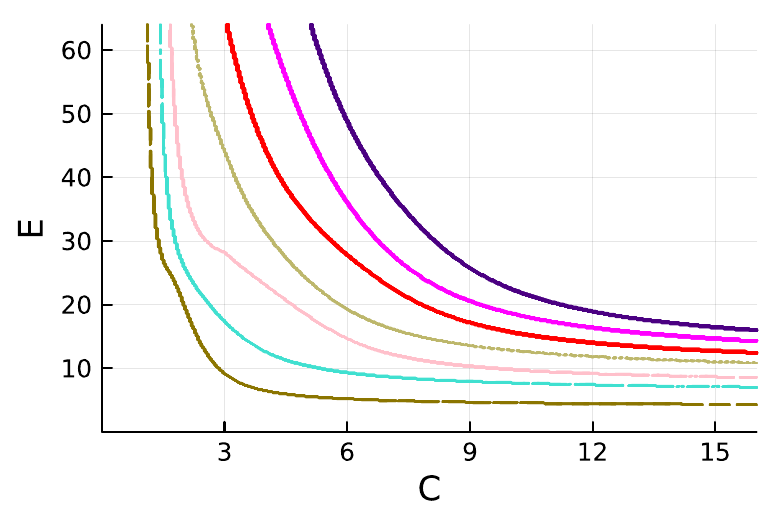}
		\centering
		$[\DeC,\GLB,k, A_1,D_2]$
	\end{minipage} 
	\begin{minipage}[t]{0.32\textwidth}
		\includegraphics[width=\textwidth]{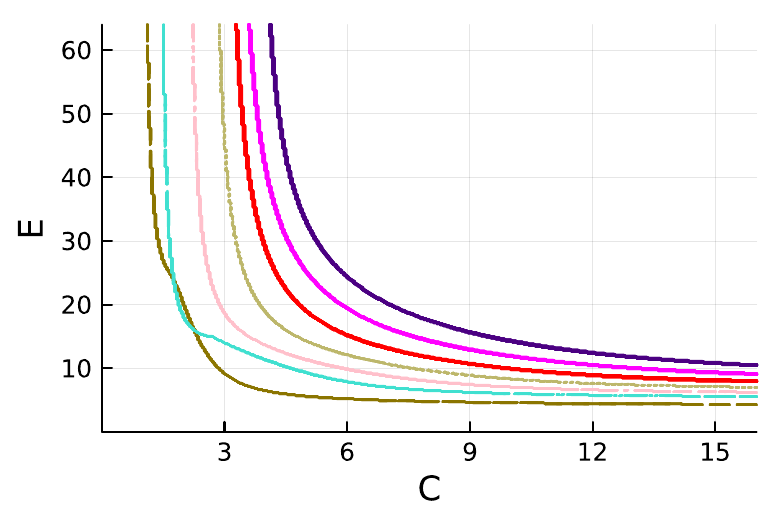}
		\centering
		$[\sDeC,\GLB,k, A_1,D_2]$
	\end{minipage}
	\begin{minipage}[t]{0.32\textwidth}
		\includegraphics[width=\textwidth]{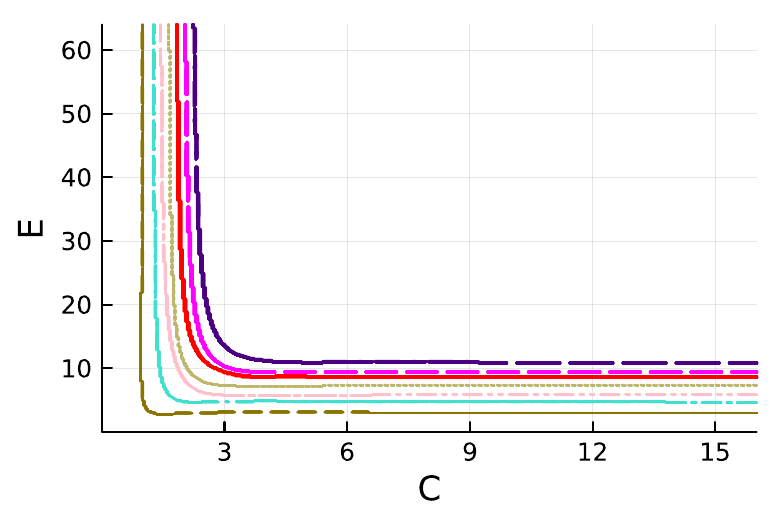}
		\centering
		$[\ADER,\GLB,k, A_1,D_2]$
	\end{minipage} 
	\caption{Stability areas for $[\TMM,\GLB,k, A_1,D_2]$ for $\TMM$ as IMEX DeC (left), IMEX sDeC (center) and IMEX ADER (right) with eq (top) and GLB (bottom) nodes: $\TMM$ order $k$ from 2 to 8}
	\label{fig: grp_adv1_diff2}
\end{figure}

In the equispaced case we do not observe major differences for the DeC and sDeC methods, 
while, for the ADER cases, the usage of equispaced nodes results in an irregular reduction to $C_0=0$ for some orders (7 and 8), meaning that we can not ensure stability as we did in the other cases for high order methods, see Figure~\ref{fig: ADER_eq_adv_diff_stab} (left). 
This is probably due to the numerical cancellation of Newton-Cotes quadratures with more than 8 points that occur for orders larger than $6$, used in the considered IMEX ADER method. 
%The reason for this catastrophic cancellation effect lies in negative weights of the quadrature which firstly appear for 8 steps. 
%This hypothesis is also supported by the IMEX ADER in figure~\ref{fig: grp_adv1_diff2_GLB} with Gauss-Lobatto nodes, where this phenomenon does not appear.
%For this reason, we will mainly analyze this method with Gauss-Lobatto nodes for the other cases.

\begin{figure}
	\centering
	\includegraphics[width=0.6\textwidth,trim={160 340 30 22}, clip]{pdf/pdepics/legends/colors_a-d_new_horiz_2-8_no_order.pdf}\\
	\begin{minipage}[t]{0.325\textwidth}
		\includegraphics[width=\textwidth]{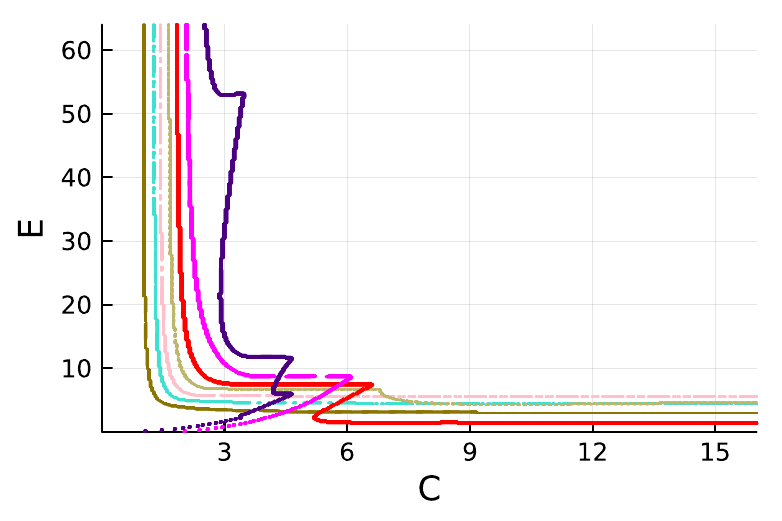}
		\centering
		$[\ADER,\eq, k, A_1,D_2]$
	\end{minipage}
	\begin{minipage}[t]{0.325\textwidth}
		\includegraphics[width=\textwidth]{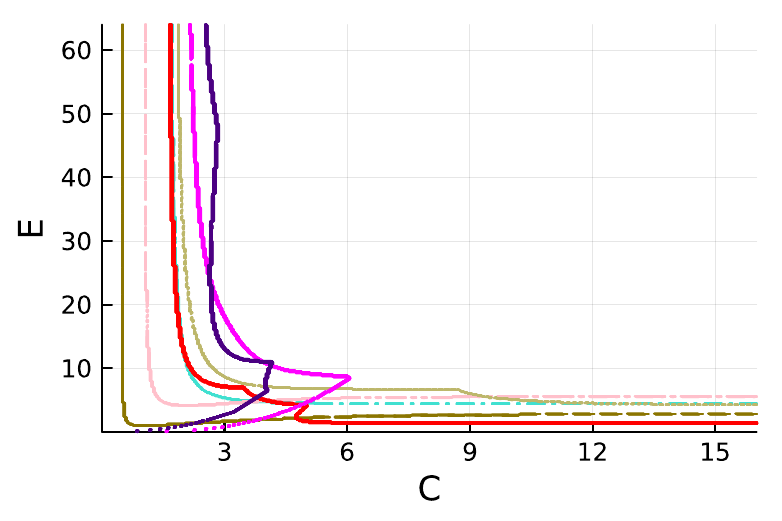}
		\centering
		$[\ADER,\eq, k, A_k,D_{2\lceil k/2 \rceil}]$
	\end{minipage}
	\begin{minipage}[t]{0.325\textwidth}
		\includegraphics[width=\textwidth]{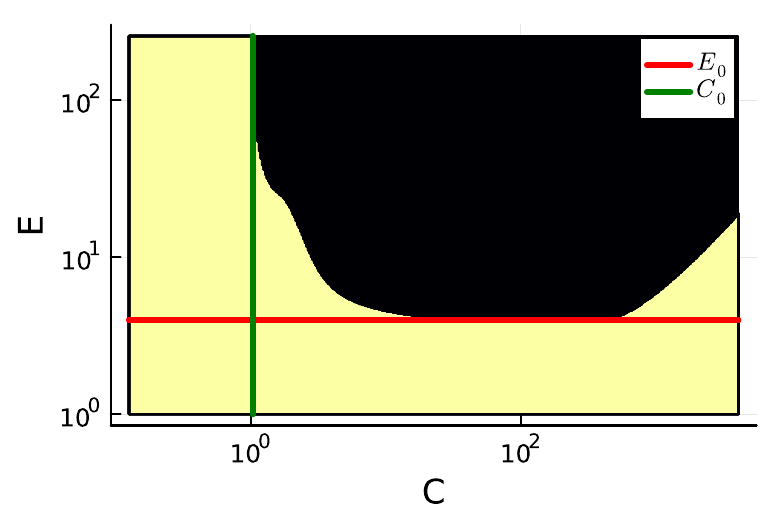}
		\centering
		$[\sDeC,\GLB, 2, A_1,D_2]$
	\end{minipage}
	\caption{Stability areas of ADER with equispaced nodes (left and center): varying the order of the time method (left) and varying the order of all discretizations (center). Bounds by $C_0$ and $E_0$ of the stability region (in yellow) for an sDeC method (right).}
	\label{fig: ADER_eq_adv_diff_stab}
\end{figure}

In Figure~\ref{fig: exa_difftermsim}, we study the space--time discretizations matching the orders of the time scheme with the order of the spatial terms, i.e.  $[\TMM,\GLB, k, A_k,D_{2\lceil k/2 \rceil}]$, where $\TMM$ is one of the 3 considered methods and $k \in \{2, \hdots, 8\}$.
Here, we can observe for all cases that the higher order terms results in slightly larger stability areas and also in bigger $C_0$ and $E_0$, which leads to the behavior we have already seen in Figure~\ref{fig: grp_adv1_diff2} varying only the order of the time scheme. Again, for high order IMEX ADER with equispaced nodes we lose the $E_0$ bound from below as shown in Figure~\ref{fig: ADER_eq_adv_diff_stab}.

\begin{figure}
	\centering
	\includegraphics[width=0.6\textwidth,trim={160 340 30 22}, clip]{pdf/pdepics/legends/colors_a-d_new_horiz_2-8_no_order.pdf}\\
	\begin{minipage}[t]{0.32\textwidth}
		\includegraphics[width=\textwidth]{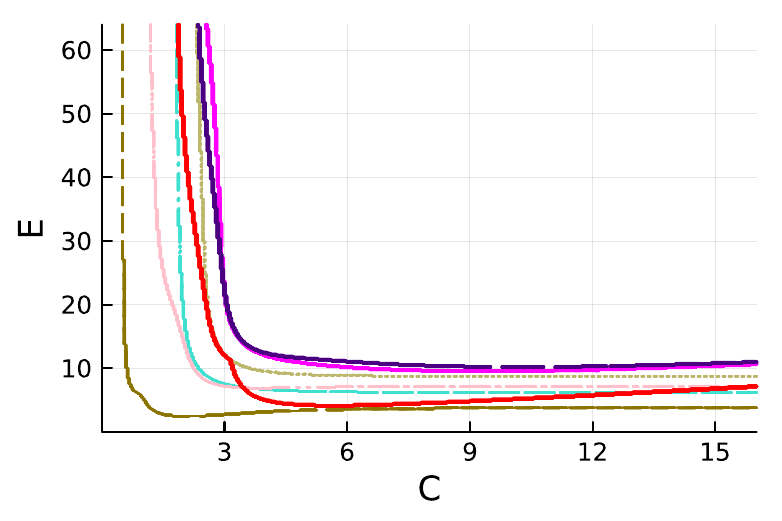}
		\centering
		$[\DeC,\GLB, k, A_k,D_{2\lceil k/2 \rceil}]$
	\end{minipage} 
	\begin{minipage}[t]{0.32\textwidth}
		\includegraphics[width=\textwidth]{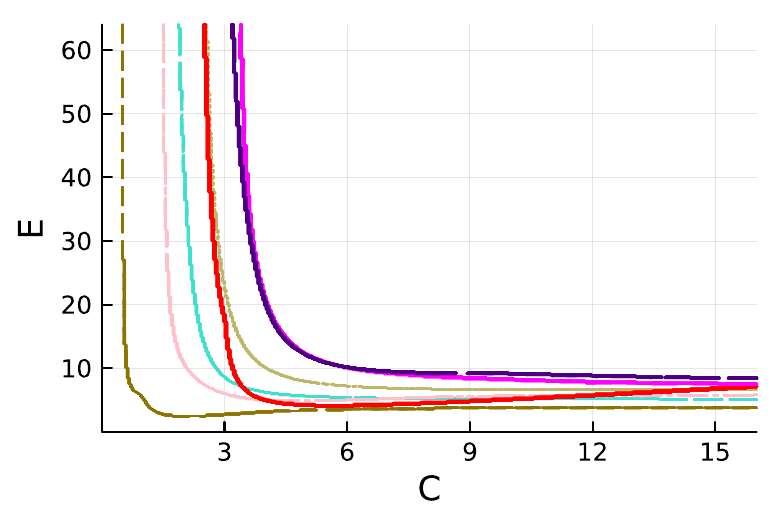}
		\centering
		$[\sDeC,\GLB, k, A_k,D_{2\lceil k/2 \rceil}]$
	\end{minipage}
	\begin{minipage}[t]{0.32\textwidth}
		\includegraphics[width=\textwidth]{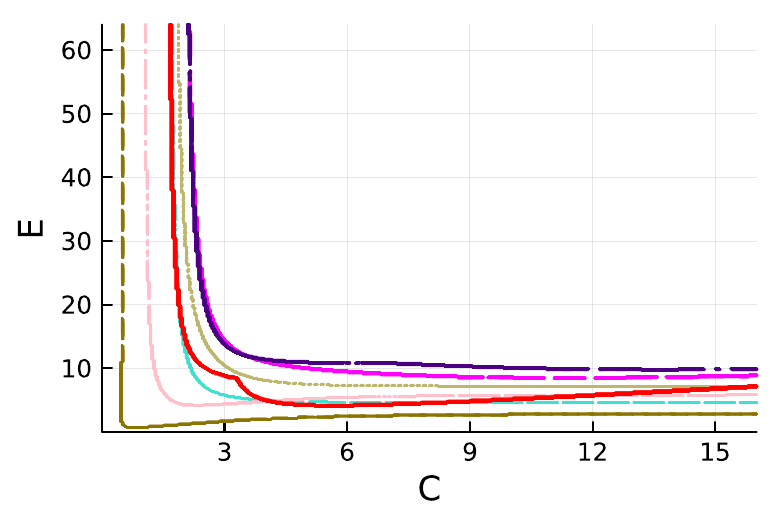}
		\centering
		$[\ADER,\GLB, k, A_k,D_{2\lceil k/2 \rceil}]$
	\end{minipage}
	\caption{Stability areas varying the order of all partial methods.}
	\label{fig: exa_difftermsim}
\end{figure}

We can conclude that increasing the order of the time-marching method results in higher values for both $C_0$ and $E_0$. However, the considered higher order finite differences for the spatial discretizations do not grant significant improvements. 
A special mention to high order IMEX ADER with equispaced nodes is necessary also here. The border $E_0$ vanishes, indeed, for order $>7$, see Figure~\ref{fig: ADER_eq_adv_diff_stab} (center). Moreover, also varying the spatial discretization we obtain the same results. 
%We also get this phenomenon by setting the advection stencil to order 6, but this problem can be overcome by the usage of high order diffusion stencils.\\
We remark that there are plenty of more options to test, which we do not include in this analysis, some of which can be found in the repository \cite{ourrepo}. Nevertheless, we presented all remarkable combinations of methods we tested and a summary of their results.

\begin{table}
	\centering
	\caption{Approximated border values $C_0$ (up to 2 decimals) and $E_0$ (up to 1 decimal) for Gauss--Lobatto methods with operators with optimal order $k$}\label{tab:CE_values}
	\begin{tabular}{|c||c|c||c|c||c|c|}\hline
				$k$&
		\multicolumn{2}{|c||}{$[\DeC,\GLB,k,A_k,D_{2\lceil k/2 \rceil}]$}&
		\multicolumn{2}{|c||}{$[\sDeC,\GLB,k,A_k,D_{2\lceil k/2 \rceil}]$}&
		\multicolumn{2}{|c|}{$[\ADER,\GLB,k,A_k,D_{2\lceil k/2 \rceil}]$}\\\hline
		&$\qquad C_0\qquad$&$E_0$&$\qquad C_0\qquad$&$E_0$&$\qquad C_0\qquad$&$E_0$\\\hline
		2& 0.50 &2.5&0.50&2.5&0.50&0.7\\
		3& 1.63 &6.1&1.69&5.1&1.63&4.5\\
		4& 1.04&6.9&1.43&4.9&1.04&4.2\\
		5& 1.74&8.8&2.31&6.6&1.74&7.2\\
		6& 1.60&4.1&2.33&4.2&1.60&4.1\\
		7& 1.94&9.5&3.12&7.5&1.94&8.5\\
		8&2.00&10.2&2.85&5.9&2.00&9.8\\ \hline
	\end{tabular}
\end{table}

As an example, we plot the stability region of $[\sDeC,\GLB,2, A_1,D_2]$ in Figure~\ref{fig: ADER_eq_adv_diff_stab}. We observe an asymptotic behavior both for $E\rightarrow \infty$ against a line $C=C_0$ and also some sort of border line $E<E_0$ for large values of $C$, which ensures stability for an arbitrary $C$, if $E\le E_0$. These are the desired values for $C_0$ and $E_0$.
In Table~\ref{tab:CE_values}, we study the operators with order $k$ matching for the time and spatial discretization for time schemes defined by GLB nodes. We display in that table the maximal values $C_0$ and $E_0$. 
We clearly see that they increase as the order increases. 
The only value that is not uniform among the methods is $E_0$ for ADER GLB of order 2, for which we have a restrictive bound. 
We also want to highlight, that $C_0$ matches inbetween the DeC and ADER methods for the same orders. This is probably due to their coinciding explicit stability regions for ODEs, as pointed out in \cite{Han_Veiga_2021}. 
	
	\section{PDE: analysis of advection-dispersion}
	\label{sec:PDE_adv_disp}
	In this section, we extend the analysis and results to observe the behavior of IMEX ADER and DeC methods onto the advection-dispersion equation
\begin{equation}
\label{eq: A-Disp_equation}
u_t(x,t) + au_x(x,t) + \beta u_{xxx}(x,t) = 0, \quad a\ge0, \ \beta \ge 0.
\end{equation}
\subsection{FD discretization}
\label{sec: disp_spatial_discretization}
First, we  introduce at this point the considered spatial discretizations for the advection-dispersion equation. Thereby, we consider the same discretization for the advection term, as introduced in Section~\ref{sec: spatial_discretization}.\\
For the dispersion term, we will consider the upwind scheme used in \cite{TanChenShu_ImEx_Stability} to test stability for the advection-dispersion equation \eqref{eq: A-Disp_equation}. It is of order 3 and given by
\begin{equation}
	\label{eq: shu_upwind_dispersion}
	\partial_{\Delta x}^3(u(x_j)) =\frac{1}{4{\Delta x}^3}\left( -w_{j-2} - w_{j-1}  + 10w_{j} - 14w_{j+1} + 7w_{j+2} - w_{j+3}\right).
\end{equation}
For higher orders, we have used the optimal $2r+1$ order formula on stencils of the type $[-r,r+1]$ with the tool provided in \cite{fdcc}.

We have also tested the methods with a central finite difference formula of order 2, always leading to less stable methods, hence, we will not include them in the following discussion.

%take two different types of methods into account: Central finite difference and upwind schemes. The underlying theory and assumptions are the same as we saw previously and their orders can be proven analogously. We consider the following central finite difference scheme of order $2$ 
%\begin{equation}
%	\label{eq: CFD-schemes_third_deriv}
%	\partial_{\Delta x}^3(u(x_j)) = \frac{1}{{\Delta x}^3}\left(-\frac{1}{2}w_{j-2} + w_{j-1} - w_{j+1} + \frac{1}{2}w_{j+2}\right)
%\end{equation}
%and 

\subsection{von Neumann analysis}
As previously done for the advection--diffusion problem, we will perform the von Neumann analysis by looking at the coefficients of the finite difference schemes, i.e.,
$$C=a\frac{\Delta t}{\Delta x},\qquad P= \beta\frac{\Delta t}{{\Delta x}^3}.$$
The procedure is analogous to the advection--diffusion one, with $C,\,P$ instead of $C,\,E$.
%So technically, we compute the methods in the same way and just adapt the spatial schemes from the diffusion term to the dispersion term. Remark that also the sign of this term changes. 
\subsubsection{Displaying stability}
To denote the considered methods, we use again the notation introduced for the advection-diffusion equation
	\begin{equation*}
[\TMM,\NODES,N, A_n,B_m],
\end{equation*}
where $B_m$ refers to the upwind $m$-th order stencils of type $[-r,r+1]$.
We proceed evaluating the amplification factor 
\begin{equation*}
G(k,\Delta x, \Delta t, a, \beta)=g(k,C,P)
\end{equation*}
to observe the stability region as a function of $C$ and $P$. 
In opposition to the advection--diffusion case, in \cite{TanChenShu_ImEx_Stability} only a CFL condition is found, even if, numerically, they observe larger stability regions with a little of dispersion.
We want to give a more comprehensive study of this behavior for different schemes and, as before, we look for meaningful coefficients that bounded by some constants give the stability. 
To find such coefficients, we proceed with an example.
\begin{example}\label{exa: disp_displaying_stability_upwind}
	In Figure~\ref{fig: disp_IIMEXDeC2/3_GLB_CvsE_upwind}, we display the stability areas for the $[\DeC, \GLB,2,A_1,B_3]$ and $[\DeC, \GLB,3, A_1,B_3]$ on the $(C,P)$ plane (left).
	In the IMEXDeC2 case, we note that for low $P$ a CFL constraint $C\leq 1$ guarantees stability, while for large $P$ we see a linear constraint of the type $P\gtrsim E_0 C $.
	In the IMEXDeC3 case, there is a further unstable region close to the $C=0$ axis. This extra unstable region is due to the fact that IMDeC3 is not A-stable and, hence, for low values of $C$ not enough numerical dissipation is brought to the system.
	
	Anyway, the linear constraint on the large $P$ motivates the following definition of 
	\begin{equation*}
		E_P:=\frac{C}{P}=\frac{\Delta ta}{\Delta x}\frac{\Delta x^3}{\beta\Delta t}=\frac{a \Delta x^2  }{\beta }.
	\end{equation*}
	Now, looking at the right plot for IMEXDeC2 in Figure~\ref{fig: disp_IIMEXDeC2/3_GLB_CvsE_upwind}, we observe that either $C\leq 1$ or $E_P\leq E_0 \approx 10^{-4}$ guarantee stability. 
	This is a peculiar result as $E_P=\frac{a \Delta x^2 }{\beta  }$ does not depend on the time discretization.
	The same does not hold for IMEXDeC3, where this area is stable only for large values of $C$, which leads to ridiculously small $\Delta x$ and large $\Delta t$.
	\begin{figure}
		\centering
		\begin{minipage}[t]{0.32\textwidth}
			\centering
			\includegraphics[width=\textwidth]{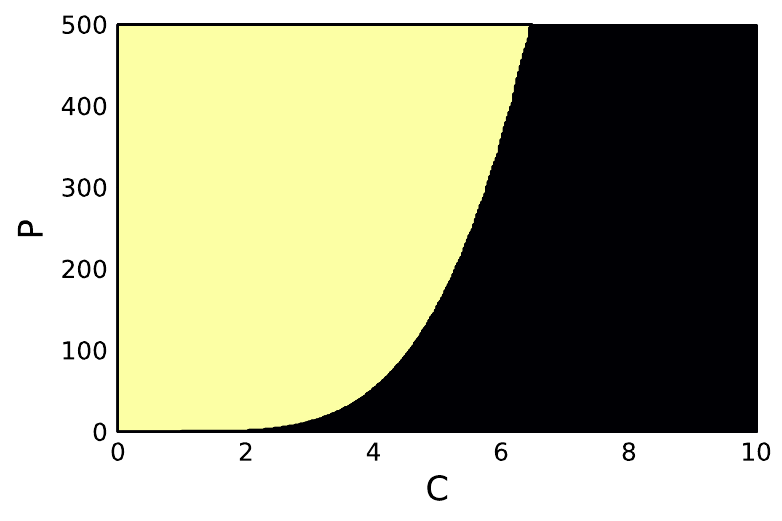}
			IMEX \DeC2 ($C$ vs $P$)
		\end{minipage}
		\begin{minipage}[t]{0.32\textwidth}
			\centering
			\includegraphics[width=\textwidth]{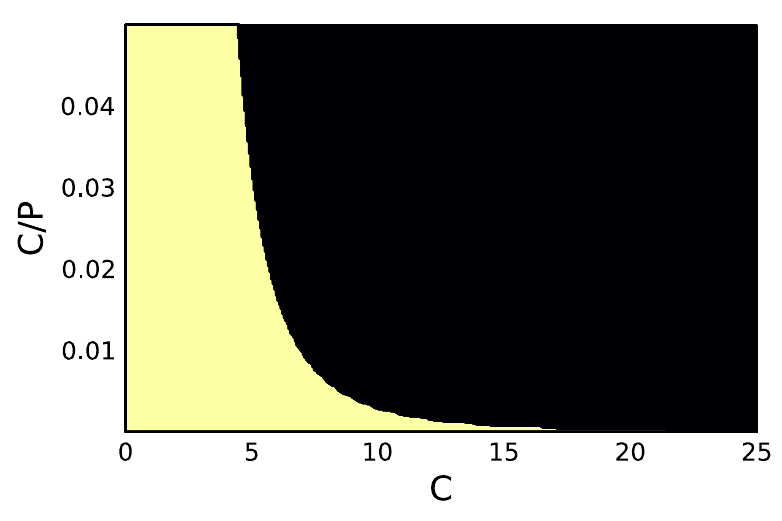}
			IMEX \DeC2 ($C$ vs $E_P$)
		\end{minipage}
		\begin{minipage}[t]{0.32\textwidth}
			\centering
			\includegraphics[width=\textwidth]{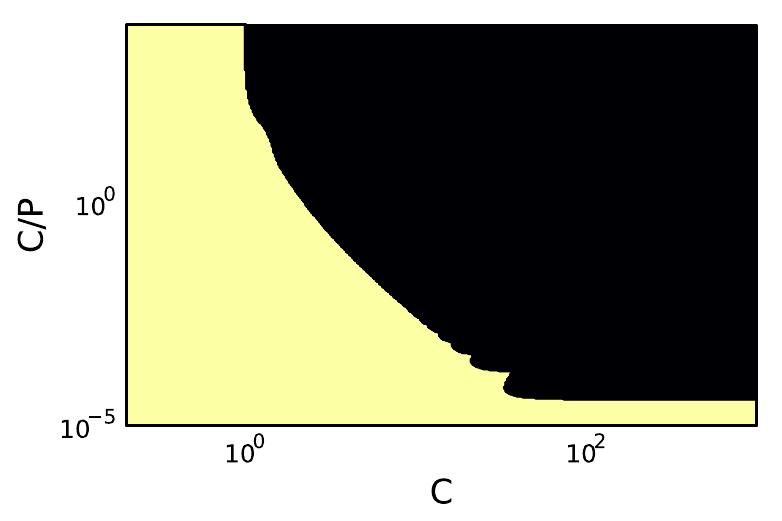}
			IMEX \DeC2 ($C$ vs $E_P$, log)
		\end{minipage}\\
		\begin{minipage}[t]{0.32\textwidth}
			\centering
			\includegraphics[width=\textwidth]{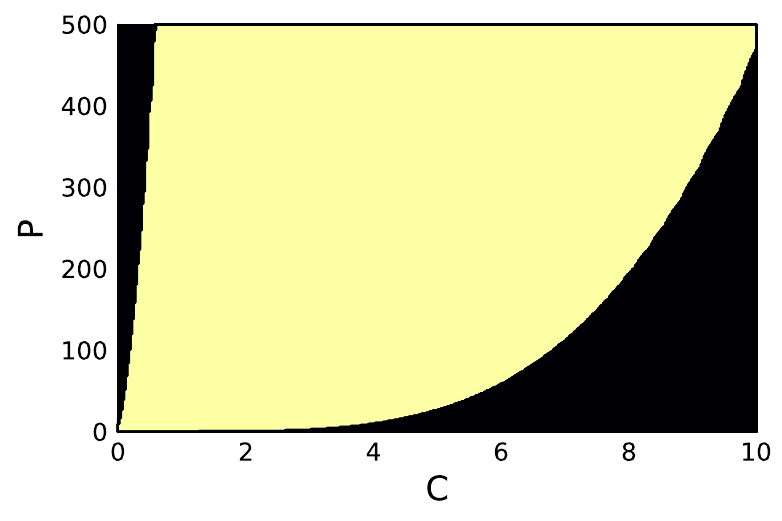}
			IMEX \DeC3 ($C$ vs $P$)
		\end{minipage} 
		\begin{minipage}[t]{0.32\textwidth}
			\centering
			\includegraphics[width=\textwidth]{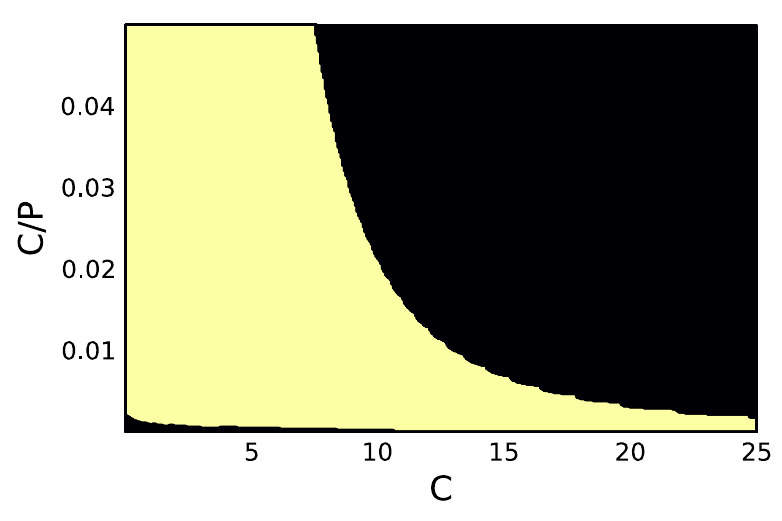}
			IMEX \DeC3 ($C$ vs $E_P$)
		\end{minipage}
		\begin{minipage}[t]{0.32\textwidth}
			\centering
			\includegraphics[width=\textwidth]{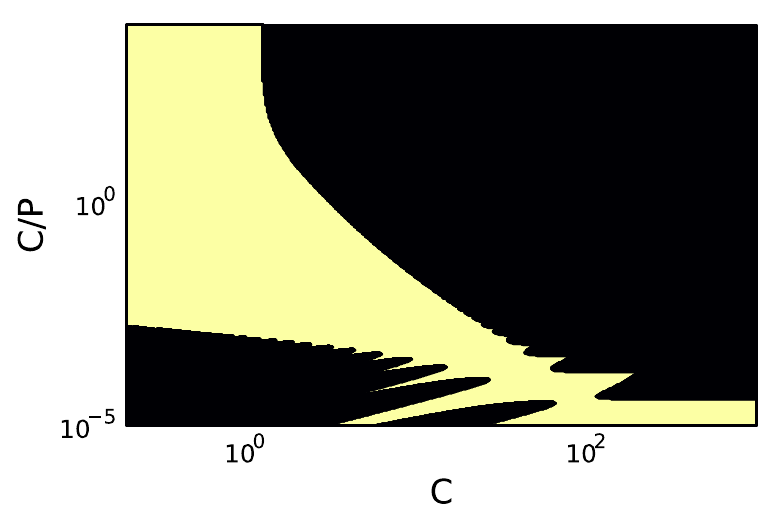}
			IMEX \DeC3 ($C$ vs $E_P$, log)
		\end{minipage}
		\caption{Stability areas for $[\DeC, \eq,2,A_1, B_3]$ and $[\DeC, \eq,3,A_1, B_3]$  with third order upwind dispersion operator}
		\label{fig: disp_IIMEXDeC2/3_GLB_CvsE_upwind}
	\end{figure}
\end{example}
These exemplary stability regions hold for most of the considered cases, i.e. all methods of order 2 do not have the instability areas for small $C$ and large $P$, as well as the IMEX ADER methods with equispaced nodes until order 4 and all IMEX ADER methods with Gauss-Lobatto nodes. 
Remark that these are exactly the methods which seem to be A-stable in their implicit ODE application as discussed in section~\ref{sec: stability_analysis_ODE}. 
All remaining methods possess this unfavorable stability region. 
%Furthermore, it can be observed that for methods corresponding to A-stable implicit ones the instability areas for very small $E_0$ appear for much smaller $E_0$ if the belonging implicit ODE method is A-stable. 
%This can be seen in the previous example and extends to all remaining considered methods.

%Note also that we tested other coefficients, for example $C$ vs $\frac{C^2}{P}$ and did not recognize any favorable restrictions in terms of regularity or borders, as for the advection-diffusion equation. 

\subsubsection{Results for IMEX DeC, sDeC and ADER}
In this section, we present the analysis results as displayed in Example~\ref{exa: disp_displaying_stability_upwind}, varying numerical methods. 
\begin{figure}
	\centering
	\includegraphics[width=0.4\textwidth,trim={230 340 30 22}, clip]{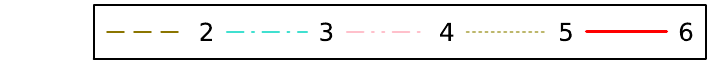}\\
	\begin{minipage}[t]{0.32\textwidth}
		\centering
		\includegraphics[width=\textwidth]{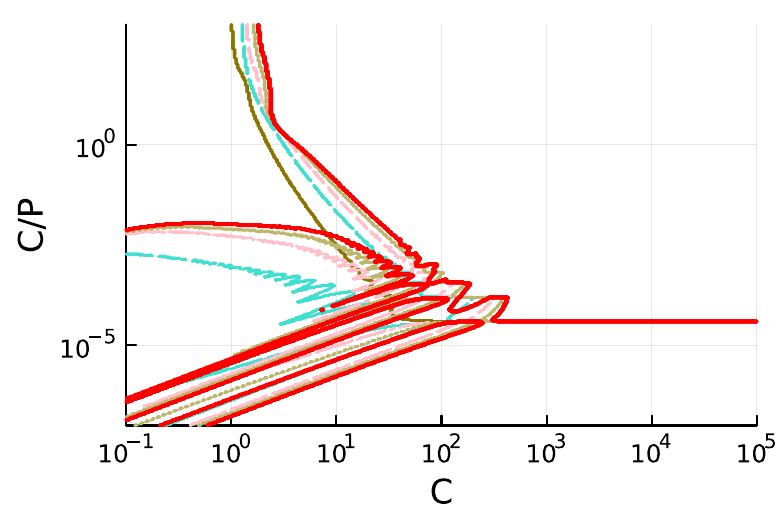}
		\small$[\DeC, \GLB,k, A_1,B_3]$\par
	\end{minipage}
	\begin{minipage}[t]{0.32\textwidth}
		\centering
		\includegraphics[width=\textwidth]{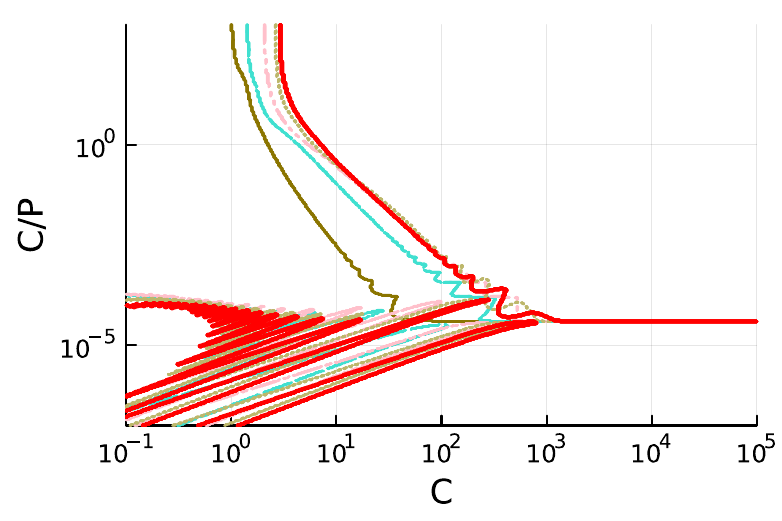}
		\small$[\sDeC, \GLB,k, A_1,B_3]$\par
	\end{minipage}
	\begin{minipage}[t]{0.32\textwidth}
		\centering
		\includegraphics[width=\textwidth]{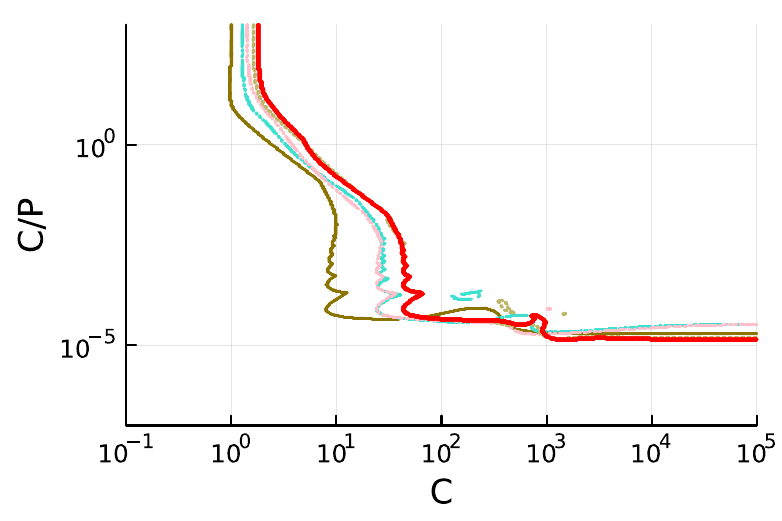}
		\small$[\ADER, \GLB,k, A_1,B_3]$\par
	\end{minipage}
%	\begin{minipage}[t]{0.32\textwidth}
%		\centering
%		\includegraphics[width=\textwidth]{pdf/pdepics/disp/IMEXDeC_equispaced_disp_TMM_2-6_newE.pdf}
%		\small$[\DeC, \eq,k,A_1,B_3]$\par
%	\end{minipage}
%	\begin{minipage}[t]{0.32\textwidth}
%		\centering
%		\includegraphics[width=\textwidth]{pdf/pdepics/disp/IMEXDeC_subtimesteps_equispaced_disp_TMM_2-6_newE.pdf}
%		\small$[\sDeC, \eq,k, A_1,B_3]$\par
%	\end{minipage}
%	\begin{minipage}[t]{0.32\textwidth}
%		\centering
%		\includegraphics[width=\textwidth]{pdf/pdepics/disp/IMEXADER_equispaced_disp_TMM_2-6_newE.pdf}
%		\small$[\ADER, \eq,k, A_1,B_3]$\par
%	\end{minipage}
	\caption{Stability areas for orders 2 to 6 with GLB nodes, the upwind scheme of \eqref{eq: shu_upwind_dispersion} for the dispersion and an first order backward scheme for the advection term}
	\label{fig: disp_allRK}
\end{figure}
We proceed now studying the stability regions increasing the order of the time scheme only, keeping fixed the advection and dispersion operators ($A_1$ and $B_3$), later on we also increase the accuracy of the advection and dispersion operators.
In Figure~\ref{fig: disp_allRK}, we can observe the stability regions changing the time scheme order from 2 to 6 for GLB nodes. 
For DeC methods of order larger than 2, we cannot  not provide bounds that guarantee the stability for small $C$ and $E_P$.
Anyway, away from this area, we observe stable regions for both $C\leq C_0$ with $C_0$ values similar to the ones of the advection--diffusion section, see Table~\ref{tab:CE_values}, and for $E_P\leq E_{P,0}$ with $E_{P,0}\approx 4\cdot 10^{-5}$ independently on the DeC method used.
In general, sDeC guarantees more stability in the region with $C\in [1,10]$ and $C/P\in [10^{-4},10^{-2}]$. The differences between equispaced and GLB are not so relevant.

On the other hand, IMEXADER with GLB nodes is very stable and there are clear bounds $C\leq C_0\leq 3$ and $E_P\leq E_{P,0}\leq 10^{-4}$ that guarantee stability. Moreover, there is a large stability area for large $C\leq 10$ and not so small $E_P$.
On the contrary, IMEXADER with equispaced points (figure available in the repository \cite{ourrepo}) for order more than 4 is much more unstable and only the are with both $C\leq C_0$ and $E_P \geq E_{P,0}$, which quite restrictive.

Again, the behavior of all these schemes reflects the A-stability property of the corresponding implicit methods.

\begin{figure}
	\centering
	\includegraphics[width=0.4\textwidth,trim={230 340 30 22}, clip]{pdf/pdepics/legends/colors_a-d_new_horiz_2-6_no_order.pdf}\\
	\begin{minipage}[t]{0.325\textwidth}
		\centering
		\includegraphics[width=\textwidth]{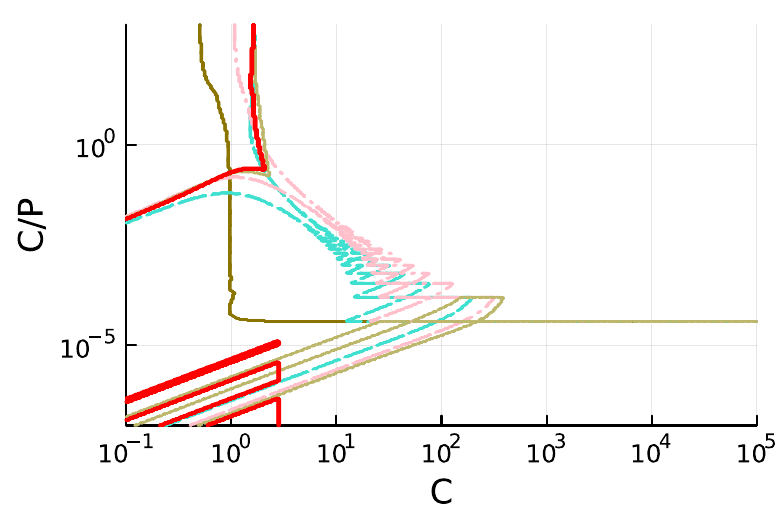}
		\small$[\DeC, \GLB,k,A_k,B_3]$\par
	\end{minipage}
	\begin{minipage}[t]{0.325\textwidth}
		\centering
		\includegraphics[width=\textwidth]{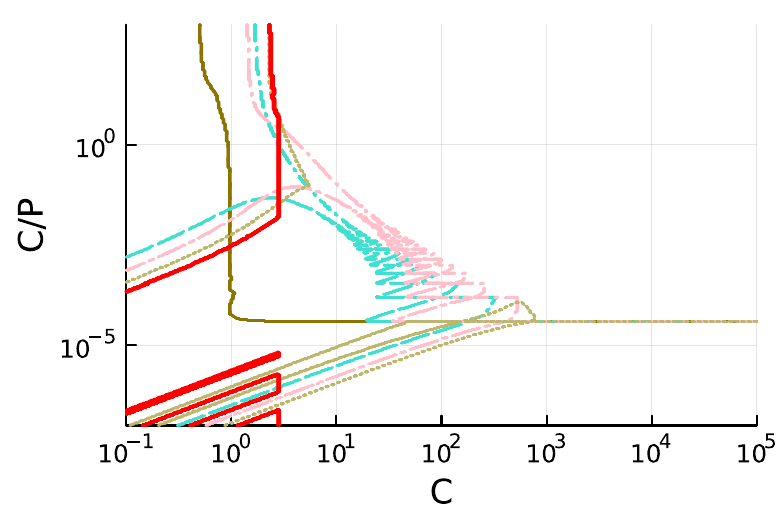}
		\small$[\sDeC, \GLB,k,A_k,B_3]$\par
	\end{minipage}
	\begin{minipage}[t]{0.325\textwidth}
		\centering
		\includegraphics[width=\textwidth]{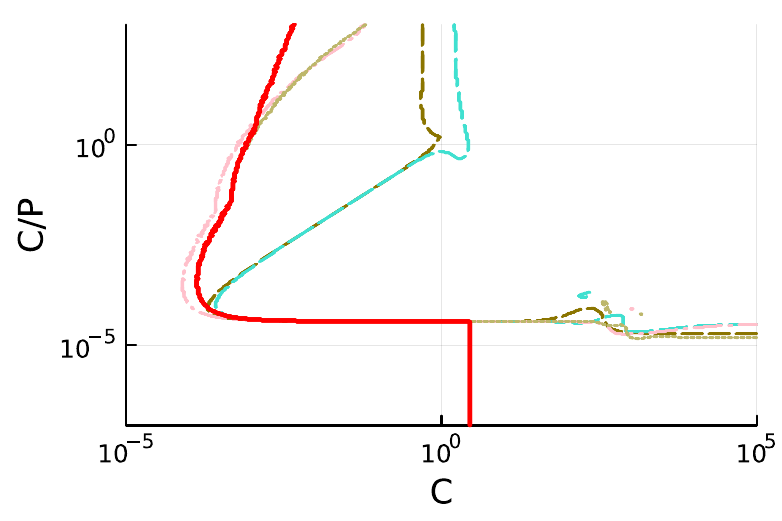}
		\small$[\ADER, \GLB,k,A_k,B_3]$\par
	\end{minipage}
	%	\\
	%	\begin{minipage}[t]{0.45\textwidth}
		%		\centering
		%		\includegraphics[width=\textwidth]{pdf/pdepics/disp/IMEXDeC_gaussLobatto_disp_advTMM_2-6_small.pdf}
		%		\small$[\DeC, \GLB,k,A_k,P_u]$(zoomed)\par
		%	\end{minipage}\hfill
	%	\begin{minipage}[t]{0.45\textwidth}
		%		\centering
		%		\includegraphics[width=\textwidth]{pdf/pdepics/disp/IMEXADER_gaussLobatto_disp_advTMM_2-6_small.pdf}
		%		\small$[\ADER, \GLB,k,A_k,P_u]$(zoomed)\par
		%	\end{minipage}
	\caption{Stability areas varying orders 2 to 6 of the advection scheme and time scheme.}
	\label{fig: disp_alladv_GLB}
\end{figure}
In Figure \ref{fig: disp_alladv_GLB}, we check the stability regions varying the advection and time order of accuracy for DeC, sDeC and ADER GLB methods. 
We find a loss of stability by increasing the order. We already see a slight reduction of our border $C_0$ for order 2 and 3. Going onto orders 4, 5 and 6, we obtain way larger unstable regions, in particular in the low $E_P$ region that was stable in the previous test.
For IMEXADER, we observe a strong reduction of the stability region in the low $C$ high $E_P$ region (observe the different scale) as well as in low $E_P$ values for order 6.

\begin{figure}
	\centering
	\includegraphics[width=0.4\textwidth,trim={230 340 30 22}, clip]{pdf/pdepics/legends/colors_a-d_new_horiz_2-6_no_order.pdf}\\
	\begin{minipage}[t]{0.325\textwidth}
		\centering
		\includegraphics[width=\textwidth]{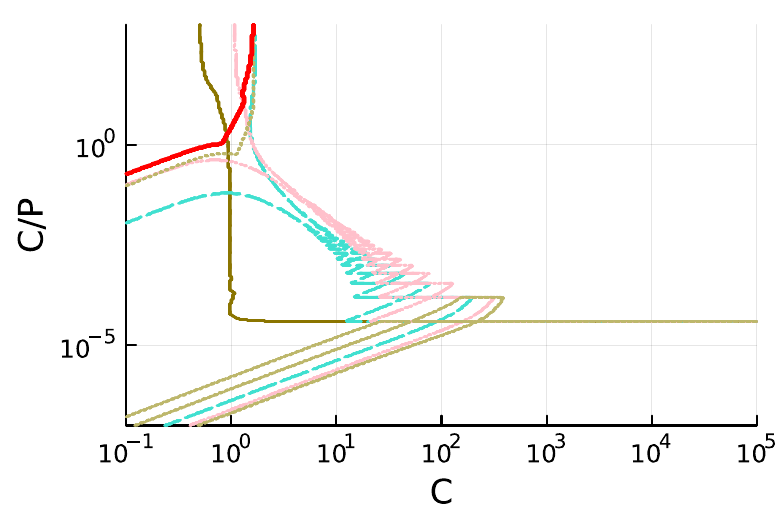}
		\small$[\DeC, \GLB,k, A_k,B_{\lceil k/2 \rceil}]$\par
	\end{minipage}
	\begin{minipage}[t]{0.325\textwidth}
		\centering
		\includegraphics[width=\textwidth]{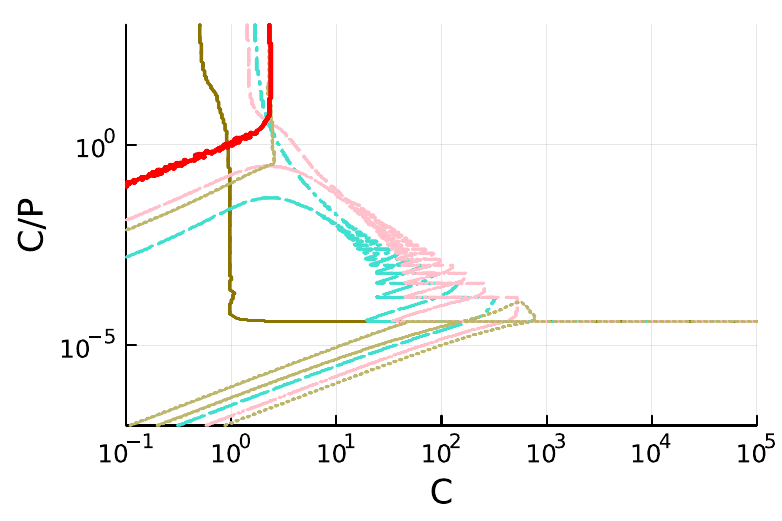}
		\small$[\sDeC, \GLB,k, A_k,B_{\lceil k/2 \rceil}]$\par
	\end{minipage}
	\begin{minipage}[t]{0.325\textwidth}
		\centering
		\includegraphics[width=\textwidth]{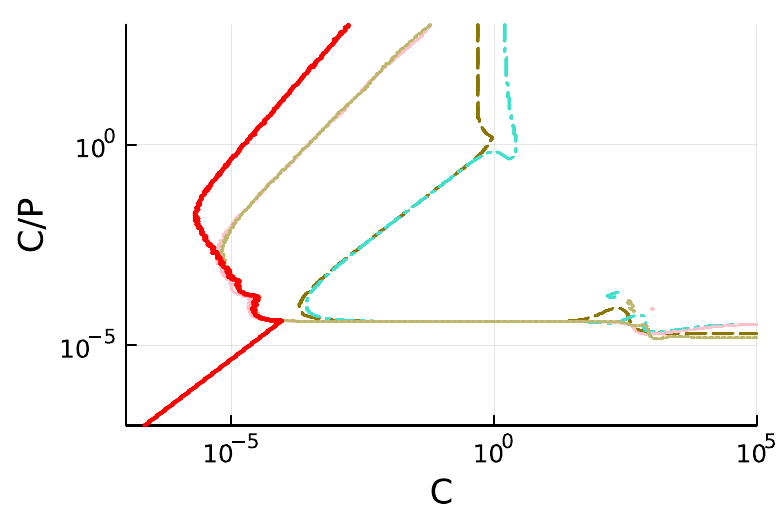}
		\small$[\ADER, \GLB,k, A_k,B_{\lceil k/2 \rceil}]$\par
	\end{minipage}
%	\begin{minipage}[t]{0.325\textwidth}
%		\centering
%		\includegraphics[width=\textwidth]{pdf/pdepics/disp/IMEXDeC_equispaced_disp_all_2-6_newE.pdf}
%		\small$[\DeC, \eq,k,A_k,B_{\lceil k/2 \rceil}]$\par
%	\end{minipage}
%	\begin{minipage}[t]{0.325\textwidth}
%		\centering
%		\includegraphics[width=\textwidth]{pdf/pdepics/disp/IMEXDeC_subtimesteps_equispaced_disp_all_2-6_newE.pdf}
%		\small$[\sDeC, \eq,k, A_k,B_{\lceil k/2 \rceil}]$\par
%	\end{minipage}
%	\begin{minipage}[t]{0.325\textwidth}
%		\centering
%		\includegraphics[width=\textwidth]{pdf/pdepics/disp/IMEXADER_equispaced_disp_all_2-6_newE.pdf}
%		\small$[\ADER, \eq,k, A_k,B_{\lceil k/2 \rceil}]$\par
%	\end{minipage}
	\caption{Stability areas for orders 2 to 6 with GLB (top) and equi (bottom) nodes, dispersion with stencil $[-\lceil (k+2)/2 \rceil +1 ,\lceil (k+2)/2\rceil]$, advection of order $k$ as in \eqref{eq: r-s_scheme} and time scheme of order $k$. Notice the difference in scales for $C$  between \ADER~and \DeC.}
	\label{fig: disp_allall}
\end{figure}
In Figure~\ref{fig: disp_allall}, we increase all together the order of all operators.
In particular, for a given order $k$ for the dispersion operator we use the optimal stencil with support $[-\lceil (k+2)/2 \rceil +1 ,\lceil (k+2)/2\rceil]$, similar to the upwinding of \eqref{eq: shu_upwind_dispersion}. To compute the coefficients of the dispersion operator, we have used the tool \cite{fdcc}. We are working with the dispersion stencil of order $3,\,5$ and $7$. 
The Fourier symbol of the stencils of order $5$ and $7$ take values very close to the imaginary axis also for quite large imaginary values. This means that schemes that are not A-stable will poorly perform on such higher orders. On the other hand, the dispersion operator of order 3 is only tangent to the imaginary axis, but it quickly has real values away from zero for large imaginary values. This will influence the stability regions.

We immediately see that the stability regions shrink and for high order DeC (greater than 5), we lose the stability region $E_P\leq E_{P,0}$. For ADER methods again the region with $C\leq C_0$ and moderate $E_P$ shrinks quite a lot and for order 6 the stability region $E_P\leq E_{P,0}$ essentially disappears. For the equispaced case, as for the time only case, from order 5 on there is no stable region for low $E_P$ \cite{ourrepo}.

We conclude that the observed IMEX methods combined with the finite difference stencils for the spatial discretization do not possess a spatial-independent condition on the time step (as for the diffusion case). 
Still, in most of the methods a classical stability region for $C\leq C_0$ and $E_P\geq E_{P,1}$, i.e., $P\geq C/E_{P,1}$, is observable, while a time independent stability region for $E_P\leq E_{P,0}$ is present only in few low order cases and it is really linked to the used spatial discretization.

	\section{Numerical tests}
	\label{sec:numerics}
	\subsection{ODE tests}
In this section, we will apply the introduced explicit, implicit and IMEX methods on ordinary differential equations to compare them with the theoretical results obtained before. Remark that our implementation takes usage of the linearized versions as in \eqref{eq: L1_linearized_ADER}. This is exact for linear systems, but has to be kept in mind while considering nonlinear ODEs. 
\subsubsection{Validation of the ImsDeC stability region}
\begin{figure}
	\centering
	\begin{minipage}[t]{0.45\textwidth}
		\includegraphics[width=\textwidth]{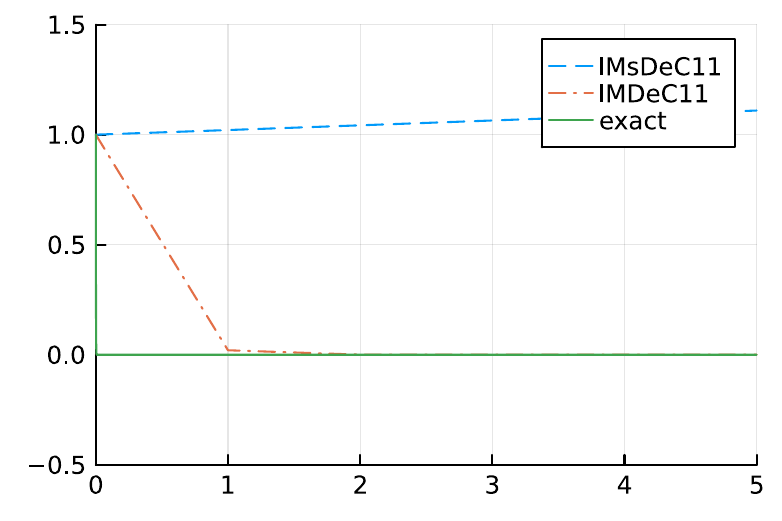}
	\end{minipage}
	\begin{minipage}[t]{0.45\textwidth}
		\includegraphics[width=\textwidth]{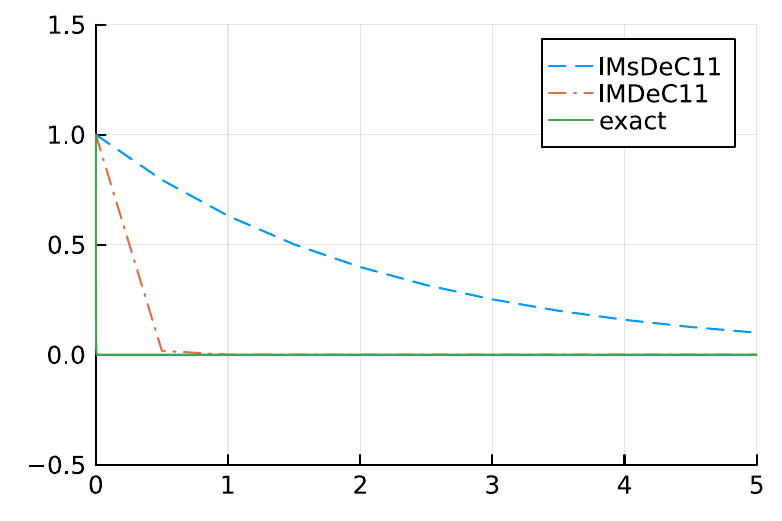}
	\end{minipage}\\
	\begin{minipage}[t]{0.45\textwidth}
		\includegraphics[width=\textwidth]{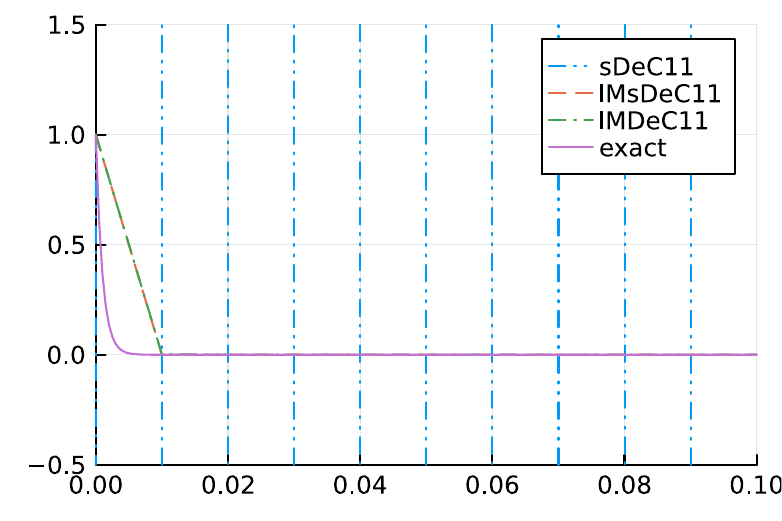}
	\end{minipage}
	\begin{minipage}[t]{0.45\textwidth}
		\includegraphics[width=\textwidth]{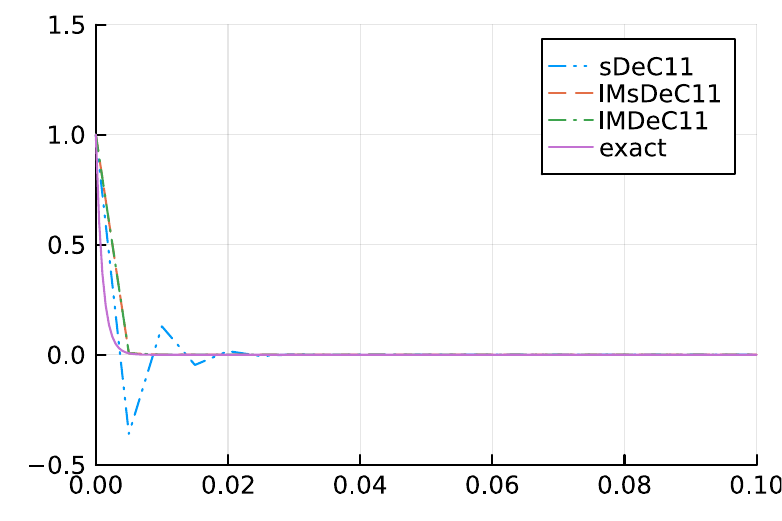}
	\end{minipage}
	\caption{Solving \eqref{eq: ImsDeC_linear_scalar} using ImsDeC11 with $h=1$ (top left), $h=0.5$ (top right), $h=0.01$ (bottom left) and $h=0.005$ (bottom right).}
	\label{fig: exaImsDeC}
\end{figure}
As seen in Figures~\ref{fig: ODEIMDeCADER} and \ref{fig: exaImsDeC_high}, the ImsDeC has an unexpected behavior by being the only considered pure implicit method here, which is not \textit{almost A-stable} for high orders. 
We want to validate this statement applying the method on some ODEs. \\
As example we take the ImsDeC11 with Gauss-Lobatto nodes, because it has the smallest stability region out of all ImsDeC methods. We observe that the limited, stable region has its left border approximately at $z=900+0i$. \\
Solving at first the linear, scalar ODE
\begin{equation}\label{eq: ImsDeC_linear_scalar}
\partial_t y(t) =-10^3y(t), y(0)=1,
\end{equation}
with different step sizes $h_1=1.0, h_2=0.5, h_3=0.02, h_4=0.01$, we expect the ImsDeC method for $h_1$ to be unstable and for the remaining 3 $h_i$ to be stable, because 
$z_1=\lambda\cdot h_1=-10^3\cdot 1=-10^3 \notin S,$ while $|z_i|=|\lambda h_i|<900$ and belong to $S$ for $i=2,3,4$,
where $S$ denotes the stability region of the ImsDeC11 with Gauss-Lobatto nodes.
The numerical results are shown in figure~\ref{fig: exaImsDeC}. We observe that for $h_1$, the method diverges from the decaying exact solution $y(t)=e^{-10^3t}$, which is in agreement with $z_1 \notin S$. 
In the second case for $h_2$, we see the solution curve mimicking the decaying behavior, which validate the fact that $z_2\in S$. As a comparison, we also plot the ImDeC11 with Gauss-Lobatto nodes which is known as almost A-stable and shows appropriate results. As seen in figure~\ref{fig: exaImsDeC}, of course for both $z_3$ and $z_4$, we obtain qualitative good solutions. 
We also display the sDeC11 in these figures and observe that their stability properties are still much worse than these of the ImsDeC11, as its stability region ends around $z\approx 6$, which make $h_c=\frac{6}{10^3}$ the threshold for a stable discretization.

We can summarize that the numerically calculated stability regions are validated through this experiment. 

\begin{figure}
	\centering
	\begin{minipage}[t]{0.45\textwidth}
		\includegraphics[width=\textwidth]{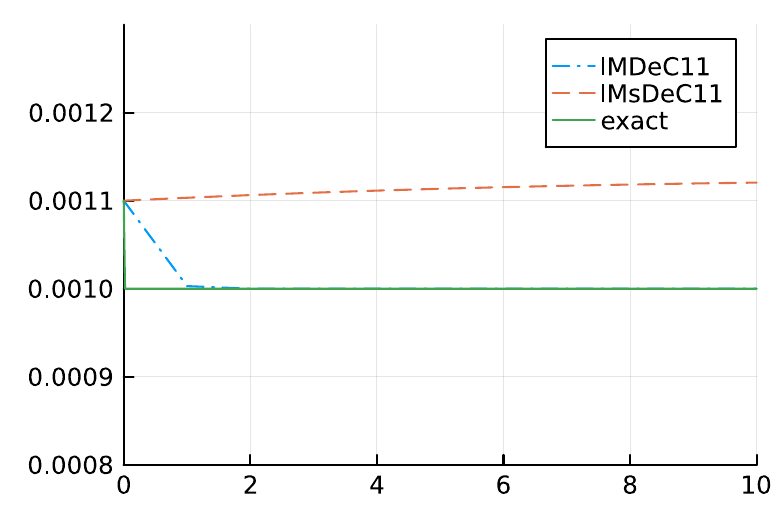}
	\end{minipage}
	\begin{minipage}[t]{0.45\textwidth}
		\includegraphics[width=\textwidth]{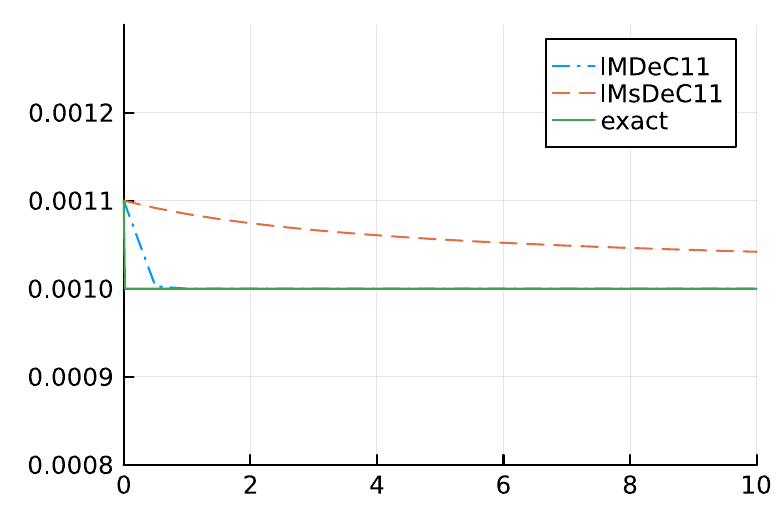}
	\end{minipage}
	\caption{Solving equation \eqref{eq: ImsDeC_nonlinear_scalar} using ImsDeC11 with $h=1.0$ (left) and $h=0.5$ (right).}
	\label{fig: exaImsDeC4}
\end{figure}
Moving to a nonlinear example, we test the ImsDeC11 with Gauss-Lobatto nodes for the nonlinear ODE
\begin{equation}\label{eq: ImsDeC_nonlinear_scalar}
\partial_t y(t)=-10^6 \lvert y(t) \rvert \cdot y(t)+1, \quad y(0)=\frac{1}{\sqrt{10^6}}
\end{equation} 
and compare it again with the analogue ImDeC. We observe in Figure \ref{fig: exaImsDeC4} on the left that the ImDeC handles the stiff equation well, while the ImsDeC again has an unstable behavior for the step $h_1=1$. This coincides with the stability region because for our nonlinear equation, an equivalent to $\lambda$ could be roughly obtained by linearizing the equation, so that $\lambda:=-10^6\lvert y(t) \rvert \approx -10^6\cdot \frac{1}{\sqrt{10^6}}=-10^3$ and therefore  $z_1\approx-10^3\notin S$. Nevertheless, after refining to $h_2=0.5$ we have $z_2\approx -500$ and we expect again stability for the ImsDeC, which can be seen on the right of Figure \ref{fig: exaImsDeC4}.
%The results of their explicit counterparts are not included in these Figures, because they diverge above $h=0.1$ and are therefore not comparable in the same Figure.

Summarizing we can observe that the non-A-stable ImsDeC performs worse than the ImDeC for stiff applications, but may still be applicable to mildly stiff problems providing better results than explicit sDeC.

\subsubsection{A stiff linear example for the IMEX methods}
We consider the second order ODE
\begin{equation}
\partial_{tt} y(t)=-2\partial_t y(t)-2501 y(t),\quad y(0)=1, \; \partial_y(0)=0,
\end{equation}
which can be rewritten as a system of ODEs
\begin{equation}\label{eq: ODE_linear_ImEx}
\begin{aligned}
\partial_t u_1(t)&=-2 u_1(t)-2501u_2(t), \quad &u_1(0)=1, \\
\partial_t u_2(t)&=u_1(t), &u_2(0)=0.
\end{aligned}
\end{equation}
The exact solution is given by 
\begin{equation*}
\vec{u}(t)=\left(\begin{array}{c}
y(t)\\
\partial_t y(t)
\end{array}\right)	=\left(\begin{array}{c}
\frac{1}{50}e^{-t}\left(\sin(50t)+50\cos(50t)\right) \\
-\frac{2501}{50}e^{-t}\sin(50t)
\end{array}\right),
\end{equation*}
i.e., it shows a rapid oscillation and a slow transient part.
Therefore, we separate the right-hand side of \eqref{eq: ODE_linear_ImEx} by 
\begin{equation*}
F_I(\vec{u}(t))=\left(\begin{array}{c}
-2501 u_2(t)\\
u_1(t)
\end{array}\right), \quad 
F_E(\vec{u}(t))=\left(\begin{array}{c}
-2u_1(t)\\
0
\end{array}\right)
\end{equation*}
to treat $F_I$ implicitly and $F_E$ explicitly.\\
The numerical solutions for IMEX ADER, IMEX DeC and to comparison ADER are shown in Figure \ref{fig: exaImEx_linear}, all methods of order 5 and use equispaced nodes. As before, we observe stability for the IMEX methods while the explicit ADER does not converge towards the exact solution.
Moreover, it is remarkable that the IMEX methods are able to catch the slow transient part even when the discretization scale does not allow to represent the fast oscillatory behavior.

\begin{figure}
	\centering
	\includegraphics[width=0.49\textwidth]{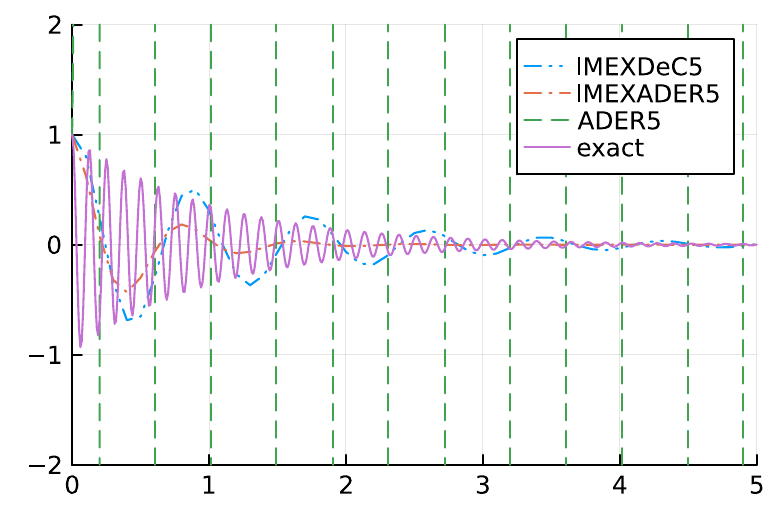}
	\caption{Solving equation \eqref{eq: ODE_linear_ImEx} with different methods using equispaced nodes with $h=0.1$.}
	\label{fig: exaImEx_linear}
\end{figure}
\subsection{PDE tests}
Finally, we assess the accuracy of the proposed schemes on the advection--diffusion equation \eqref{eq: A-D_equation}. We consider the domain $\Omega=[0,2\pi]$, $u_0(x)=\sin(x)$, $a=1$ and $C=\frac{a \Delta t}{\Delta x } =0.4$. We fix for all simulations $E=0.5$, leading to a variable $d=\frac{a^2 \Delta t}{E}$. We test ADER, DeC and sDeC with {\GLB} nodes for variable orders from 2 to 5 and meshes with sizes from $N=2^5$ to $N=2^{11}$. 
The spatial operators are generated using \cite{ranocha2021sbp}.
For all methods we obtain stable simulations. In Figure~\ref{fig:convergence_AD}, we observe how all methods converge to the exact solution $u(t,x)=e^{-dt}\sin(x-at)$ with the expected order of accuracy.

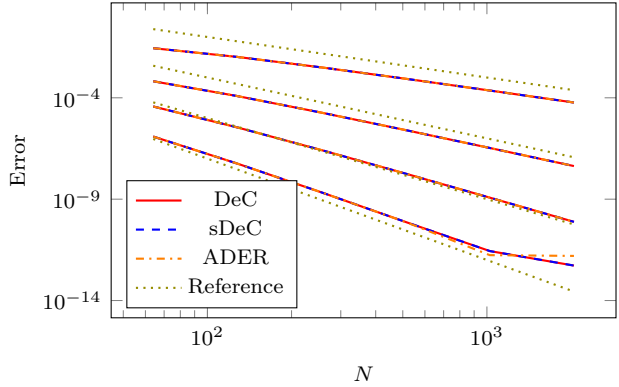
\begin{figure}
	\centering
	\begin{tikzpicture}  
		\begin{axis}[ymode=log,
			xmode=log,
			xlabel={$N$},
			ylabel={Error},
			legend pos=south west,%outer north east,
			width=0.5\textwidth,
			height=0.35\textwidth,
			style={font=\footnotesize}]
						
			\foreach \order in {2,...,5}{
				\addplot[mark size=1.3pt, red, thick] table [mark=triangle*,x=N, y=order\order, col sep=comma] {pdf/pdepics/numerics/dec_error.csv};
			}

			\foreach \order in {2,...,5}{
				\addplot[mark size=1.3pt, blue, thick, dashed] table [mark=triangle*,x=N, y=order\order, col sep=comma] {pdf/pdepics/numerics/sdec_error.csv};
			}
			
			\foreach \order in {2,...,5}{
				\addplot[mark size=1.3pt, orange, thick, dashdotted] table [mark=triangle*,x=N, y=order\order, col sep=comma] {pdf/pdepics/numerics/ader_error.csv};
			}

			\foreach \order in {2,...,5}{
				\addplot[mark size=1.3pt, olive, thick, dotted] table [mark=triangle*,x=N, y expr=1000*\thisrow{N}^-\order, col sep=comma] {pdf/pdepics/numerics/sdec_error.csv};
			}
			\legend{DeC,,,,sDeC,,,,ADER,,,,Reference,,,}
%			\addplot[dashdotted, mark size=1.3pt, red] table [x=Number, y=Eul_eig_norm, col sep=comma] {img/sod_1D_param_numerical/all_eigs_m.csv};
%			\addlegendentry{Eul POD $m$}
%			\addplot[densely dotted,mark size=1.3pt, red] table [x=Number, y=Eul_eig_norm, col sep=comma] {img/sod_1D_param_numerical/all_eigs_E.csv};
%			\addlegendentry{Eul POD $E$}
%			
%			\addplot[mark size=3pt, blue] table [mark=square*,x=Number, y=ALE_eig_norm, col sep=comma] {img/sod_1D_param_numerical/all_eigs_rho.csv};
%			\addlegendentry{ALE POD $\rho$}
%			\addplot[dashdotted, mark size=3pt, blue] table [x=Number, y=ALE_eig_norm, col sep=comma] {img/sod_1D_param_numerical/all_eigs_m.csv};
%			\addlegendentry{ALE POD $m$}
%			\addplot[densely dotted, mark size=3pt, blue] table [x=Number, y=ALE_eig_norm, col sep=comma] {img/sod_1D_param_numerical/all_eigs_E.csv};
%			\addlegendentry{ALE POD $E$}
		\end{axis}
	\end{tikzpicture}
	\caption{Error convergence on an advection diffusion problem for DeC, sDeC, ADER and dotted reference error line $N^{-K}$, for orders from $K=2$ to $K=5$ (from top to bottom).}\label{fig:convergence_AD}
\end{figure}

	\section{Conclusions}
	\label{sec:conclusion}

In our study, we have analyzed the implicit and implicit-explicit ADER and DeC methods in terms of their stability properties. To this end, we initially reformulated them as RK methods and investigated them based on the selected order, method, and quadrature nodes. Unlike our prior work \cite{Han_Veiga_2021}, which focused on explicit versions, we observed significant variations in stability behavior, ranging from A-stable to bounded stability regions. In general, the implicit (and implicit-explicit) ADER methodology demonstrated greater stability compared to the DeC framework.

After the ODE case, we further extended  our analysis to the PDE case, focusing on advection-diffusion and advection-dispersion equations inspired by previous works \cite{TanChenShu_ImEx_Stability, WangShuZhang_LDG1_2015}. For space discretization, we utilized up-to-date finite difference stencils and derived CFL-like stability conditions through von Neumann stability analysis. 
Notably, we expanded upon the investigation of \cite{TanChenShu_ImEx_Stability} by introducing two new auxiliary coefficients for the advection-diffusion equation. 
These coefficients yielded equivalent conditions to those in \cite{TanChenShu_ImEx_Stability}, but they do not depend on the spatial discretization. 
We established precise boundaries for relevant coefficients for advection-diffusion and advection-dispersion and offered recommendations regarding the suitability of specific schemes.

In the future, further potential research directions include exploring different space discretization methods and focusing on continuous and discontinuous Galerkin formulations \cite{ortleb2023stability, zbMATH07137361}. Additionally, investigating stability in the context of nonlinear problems would be desirable, particularly focusing on the entropy production of such schemes as suggested by previous works \cite{zbMATH07086321, zbMATH06928679, oeffner2020}, or
using the add-and-subtract version of the implicit ADER and DeC to study the stability in the same style of \cite{tan2022stability}.

\section*{Acknowledgments}

	P.Ö. is supported by the DFG within SPP 2410, project  525866748 (OE 661/5-1) and under the personal grant 520756621 (OE 661/4-1).  P.Ö. gratefully acknowledges the 		 	support of the Gutenberg Research College, JGU Mainz.
	D. T. is funded by a SISSA Mathematical Fellowship and is a member of the Gruppo Nazionale Calcolo Scientifico-Istituto Nazionale di Alta Matematica (GNCS-INdAM).
	L. P. is supported by the DFG within SPP 2410, project number 526031774 and the Daimler und Benz Stiftung (Daimler and Benz foundation), project number 32-10/22.\\
	D. T. and P. Ö. wish to express their gratitude to prof. Giovanni Russo, who brought to their attention various seminal works on the stability of IMEX schemes. Engaging in discussions with him has deepened our comprehension in this domain.

%\section*{CRediT authorship contribution statement}

%Philipp \"Offner:
%Conceptualization, Investigation, Methodology, Project administration,  
%Software, Supervision, Writing --- original draft,
%Writing --- review \& editing.
%Louis Petri: Data curation, Formal analysis, Investigation, Software,
%Visualization, Writing --- original draft, Writing --- review \& editing.
%Davide Torlo:
%Conceptualization, Formal analysis,  Investigation,  Methodology, Software, Supervision, Validation,
%Visualization, Writing --- original draft, Writing --- review \& editing.

\section*{Declaration of competing interest}

The authors declare that they have no known competing financial
interests or personal relationships that could have appeared to
influence the work reported in this paper.

\section*{Data availability}

Results for some of the proposed methods are listed online in our repository
\cite{ourrepo} and the code is available at request.

\bibliographystyle{elsarticle-num-names}
	\bibliography{literature}

\begin{thebibliography}{54}
\expandafter\ifx\csname natexlab\endcsname\relax\def\natexlab#1{#1}\fi
\providecommand{\url}[1]{\texttt{#1}}
\providecommand{\href}[2]{#2}
\providecommand{\path}[1]{#1}
\providecommand{\DOIprefix}{doi:}
\providecommand{\ArXivprefix}{arXiv:}
\providecommand{\URLprefix}{URL: }
\providecommand{\Pubmedprefix}{pmid:}
\providecommand{\doi}[1]{\href{http://dx.doi.org/#1}{\path{#1}}}
\providecommand{\Pubmed}[1]{\href{pmid:#1}{\path{#1}}}
\providecommand{\bibinfo}[2]{#2}
\ifx\xfnm\relax \def\xfnm[#1]{\unskip,\space#1}\fi
%Type = Article
\bibitem[{Pareschi and Russo(2000)}]{pareschi2000implicit}
\bibinfo{author}{L.~Pareschi}, \bibinfo{author}{G.~Russo},
\newblock \bibinfo{title}{Implicit-explicit {Runge-Kutta} schemes for stiff systems of differential equations},
\newblock \bibinfo{journal}{Recent trends in numerical analysis} \bibinfo{volume}{3} (\bibinfo{year}{2000}) \bibinfo{pages}{269--289}.
%Type = Article
\bibitem[{Dutt et~al.(2000)Dutt, Greengard, and Rokhlin}]{dutt2000dec}
\bibinfo{author}{A.~Dutt}, \bibinfo{author}{L.~Greengard}, \bibinfo{author}{V.~Rokhlin},
\newblock \bibinfo{title}{{Spectral Deferred Correction Methods for Ordinary Differential Equations}},
\newblock \bibinfo{journal}{BIT Numerical Mathematics} \bibinfo{volume}{40} (\bibinfo{year}{2000}) \bibinfo{pages}{241--266}. \URLprefix \url{http://dx.doi.org/10.1023/A:1022338906936}. \DOIprefix\doi{10.1023/A:1022338906936}.
%Type = Article
\bibitem[{Abgrall(2017)}]{abgrall2017dec}
\bibinfo{author}{R.~Abgrall},
\newblock \bibinfo{title}{High order schemes for hyperbolic problems using globally continuous approximation and avoiding mass matrices},
\newblock \bibinfo{journal}{Journal of Scientific Computing} \bibinfo{volume}{73} (\bibinfo{year}{2017}) \bibinfo{pages}{461--494}. \URLprefix \url{https://doi.org/10.1007/s10915-017-0498-4}. \DOIprefix\doi{10.1007/s10915-017-0498-4}.
%Type = Article
\bibitem[{Christlieb et~al.(2010)Christlieb, Ong, and Qiu}]{christlieb2010integral}
\bibinfo{author}{A.~Christlieb}, \bibinfo{author}{B.~Ong}, \bibinfo{author}{J.-M. Qiu},
\newblock \bibinfo{title}{Integral deferred correction methods constructed with high order {Runge}-{Kutta} integrators},
\newblock \bibinfo{journal}{Math. Comput.} \bibinfo{volume}{79} (\bibinfo{year}{2010}) \bibinfo{pages}{761--783}. \DOIprefix\doi{10.1090/S0025-5718-09-02276-5}.
%Type = Article
\bibitem[{Minion(2003)}]{minion2003dec}
\bibinfo{author}{M.~L. Minion},
\newblock \bibinfo{title}{Semi-implicit spectral deferred correction methods for ordinary differential equations},
\newblock \bibinfo{journal}{Commun. Math. Sci.} \bibinfo{volume}{1} (\bibinfo{year}{2003}) \bibinfo{pages}{471--500}. \URLprefix \url{https://projecteuclid.org:443/euclid.cms/1250880097}.
%Type = Article
\bibitem[{\"Offner and Torlo(2020)}]{offner2019arbitrary}
\bibinfo{author}{P.~\"Offner}, \bibinfo{author}{D.~Torlo},
\newblock \bibinfo{title}{Arbitrary high-order, conservative and positivity preserving {Patankar}-type deferred correction schemes},
\newblock \bibinfo{journal}{Applied Numerical Mathematics} \bibinfo{volume}{153} (\bibinfo{year}{2020}) \bibinfo{pages}{15–34}. \URLprefix \url{http://dx.doi.org/10.1016/j.apnum.2020.01.025}. \DOIprefix\doi{10.1016/j.apnum.2020.01.025}.
%Type = Article
\bibitem[{Abgrall et~al.(2022)Abgrall, Le~M{\'e}l{\'e}do, {\"O}ffner, and Torlo}]{abgrall2022relaxation}
\bibinfo{author}{R.~Abgrall}, \bibinfo{author}{{\'E}.~Le~M{\'e}l{\'e}do}, \bibinfo{author}{P.~{\"O}ffner}, \bibinfo{author}{D.~Torlo},
\newblock \bibinfo{title}{Relaxation deferred correction methods and their applications to residual distribution schemes},
\newblock \bibinfo{journal}{SMAI J. Comput. Math.} \bibinfo{volume}{8} (\bibinfo{year}{2022}) \bibinfo{pages}{125--160}. \DOIprefix\doi{10.5802/smai-jcm.82}.
%Type = Article
\bibitem[{Layton and Minion(2004)}]{layton2004conservative}
\bibinfo{author}{A.~T. Layton}, \bibinfo{author}{M.~L. Minion},
\newblock \bibinfo{title}{Conservative multi-implicit spectral deferred correction methods for reacting gas dynamics},
\newblock \bibinfo{journal}{Journal of Computational Physics} \bibinfo{volume}{194} (\bibinfo{year}{2004}) \bibinfo{pages}{697--715}.
%Type = Article
\bibitem[{Speck et~al.(2015)Speck, Ruprecht, Emmett, Minion, Bolten, and Krause}]{speck2015multi}
\bibinfo{author}{R.~Speck}, \bibinfo{author}{D.~Ruprecht}, \bibinfo{author}{M.~Emmett}, \bibinfo{author}{M.~Minion}, \bibinfo{author}{M.~Bolten}, \bibinfo{author}{R.~Krause},
\newblock \bibinfo{title}{A multi-level spectral deferred correction method},
\newblock \bibinfo{journal}{BIT Numerical Mathematics} \bibinfo{volume}{55} (\bibinfo{year}{2015}) \bibinfo{pages}{843--867}.
%Type = Inbook
\bibitem[{Toro et~al.(2001)Toro, Millington, and Nejad}]{ADERHistorical2}
\bibinfo{author}{E.~F. Toro}, \bibinfo{author}{R.~C. Millington}, \bibinfo{author}{L.~A.~M. Nejad}, \bibinfo{title}{Towards Very High Order {Godunov} Schemes}, \bibinfo{publisher}{Springer US}, \bibinfo{year}{2001}, p. \bibinfo{pages}{907–940}. \URLprefix \url{http://dx.doi.org/10.1007/978-1-4615-0663-8_87}. \DOIprefix\doi{10.1007/978-1-4615-0663-8_87}.
%Type = Article
\bibitem[{Schwartzkopff et~al.(2002)Schwartzkopff, Munz, and Toro}]{ADERHistorical1}
\bibinfo{author}{T.~Schwartzkopff}, \bibinfo{author}{C.~D. Munz}, \bibinfo{author}{E.~F. Toro},
\newblock \bibinfo{title}{{ADER}: A high-order approach for linear hyperbolic systems in 2d},
\newblock \bibinfo{journal}{J. Sci. Comput.} \bibinfo{volume}{17} (\bibinfo{year}{2002}) \bibinfo{pages}{231–240}. \URLprefix \url{https://doi.org/10.1023/A:1015160900410}. \DOIprefix\doi{10.1023/A:1015160900410}.
%Type = Article
\bibitem[{Titarev and Toro(2002)}]{titarev2002ader}
\bibinfo{author}{V.~A. Titarev}, \bibinfo{author}{E.~F. Toro},
\newblock \bibinfo{title}{{ADER}: Arbitrary high order {G}odunov approach},
\newblock \bibinfo{journal}{Journal of Scientific Computing} \bibinfo{volume}{17} (\bibinfo{year}{2002}) \bibinfo{pages}{609--618}.
%Type = Article
\bibitem[{Dumbser et~al.(2008)Dumbser, Balsara, Toro, and Munz}]{ADERModern}
\bibinfo{author}{M.~Dumbser}, \bibinfo{author}{D.~S. Balsara}, \bibinfo{author}{E.~F. Toro}, \bibinfo{author}{C.-D. Munz},
\newblock \bibinfo{title}{A unified framework for the construction of one-step finite volume and discontinuous {Galerkin} schemes on unstructured meshes},
\newblock \bibinfo{journal}{Journal of Computational Physics} \bibinfo{volume}{227} (\bibinfo{year}{2008}) \bibinfo{pages}{8209–8253}. \URLprefix \url{http://dx.doi.org/10.1016/j.jcp.2008.05.025}. \DOIprefix\doi{10.1016/j.jcp.2008.05.025}.
%Type = Article
\bibitem[{Gaburro et~al.(2023)Gaburro, {\"O}ffner, Ricchiuto, and Torlo}]{zbMATH07627644}
\bibinfo{author}{E.~Gaburro}, \bibinfo{author}{P.~{\"O}ffner}, \bibinfo{author}{M.~Ricchiuto}, \bibinfo{author}{D.~Torlo},
\newblock \bibinfo{title}{High order entropy preserving {ADER}-{DG} schemes},
\newblock \bibinfo{journal}{Appl. Math. Comput.} \bibinfo{volume}{440} (\bibinfo{year}{2023}) \bibinfo{pages}{21}. \DOIprefix\doi{10.1016/j.amc.2022.127644}.
%Type = Article
\bibitem[{Dumbser et~al.(2008)Dumbser, Enaux, and Toro}]{dumbser2007FVStiff}
\bibinfo{author}{M.~Dumbser}, \bibinfo{author}{C.~Enaux}, \bibinfo{author}{E.~F. Toro},
\newblock \bibinfo{title}{Finite volume schemes of very high order of accuracy for stiff hyperbolic balance laws},
\newblock \bibinfo{journal}{J. Comput. Phys.} \bibinfo{volume}{227} (\bibinfo{year}{2008}) \bibinfo{pages}{3971–4001}. \URLprefix \url{https://doi.org/10.1016/j.jcp.2007.12.005}. \DOIprefix\doi{10.1016/j.jcp.2007.12.005}.
%Type = Article
\bibitem[{Boscheri and Dumbser(2014)}]{boscheri2014direct}
\bibinfo{author}{W.~Boscheri}, \bibinfo{author}{M.~Dumbser},
\newblock \bibinfo{title}{{A direct Arbitrary-Lagrangian--Eulerian ADER-WENO finite volume scheme on unstructured tetrahedral meshes for conservative and non-conservative hyperbolic systems in 3D}},
\newblock \bibinfo{journal}{Journal of Computational Physics} \bibinfo{volume}{275} (\bibinfo{year}{2014}) \bibinfo{pages}{484--523}.
%Type = Article
\bibitem[{Micalizzi et~al.(2023)Micalizzi, Torlo, and Boscheri}]{micalizzi2023efficient}
\bibinfo{author}{L.~Micalizzi}, \bibinfo{author}{D.~Torlo}, \bibinfo{author}{W.~Boscheri},
\newblock \bibinfo{title}{Efficient iterative arbitrary high-order methods: an adaptive bridge between low and high order},
\newblock \bibinfo{journal}{Communications on Applied Mathematics and Computation}  (\bibinfo{year}{2023}) \bibinfo{pages}{1--38}.
%Type = Article
\bibitem[{{Han Veiga} et~al.(2024){Han Veiga}, Micalizzi, and Torlo}]{veiga2023improving}
\bibinfo{author}{M.~{Han Veiga}}, \bibinfo{author}{L.~Micalizzi}, \bibinfo{author}{D.~Torlo},
\newblock \bibinfo{title}{On improving the efficiency of {ADER} methods},
\newblock \bibinfo{journal}{Applied Mathematics and Computation} \bibinfo{volume}{466} (\bibinfo{year}{2024}) \bibinfo{pages}{128426}. \DOIprefix\doi{10.1016/j.amc.2023.128426}.
%Type = Article
\bibitem[{Han~Veiga et~al.(2021)Han~Veiga, \"Offner, and Torlo}]{Han_Veiga_2021}
\bibinfo{author}{M.~Han~Veiga}, \bibinfo{author}{P.~\"Offner}, \bibinfo{author}{D.~Torlo},
\newblock \bibinfo{title}{{DeC} and {ADER:} similarities, differences and a unified framework},
\newblock \bibinfo{journal}{Journal of Scientific Computing} \bibinfo{volume}{87} (\bibinfo{year}{2021}). \URLprefix \url{http://dx.doi.org/10.1007/s10915-020-01397-5}. \DOIprefix\doi{10.1007/s10915-020-01397-5}.
%Type = Article
\bibitem[{Abgrall and Torlo(2020)}]{abgrall2018asymptotic}
\bibinfo{author}{R.~Abgrall}, \bibinfo{author}{D.~Torlo},
\newblock \bibinfo{title}{High order asymptotic preserving deferred correction implicit-explicit schemes for kinetic models},
\newblock \bibinfo{journal}{SIAM Journal on Scientific Computing} \bibinfo{volume}{42} (\bibinfo{year}{2020}) \bibinfo{pages}{B816–B845}. \URLprefix \url{http://dx.doi.org/10.1137/19m128973x}. \DOIprefix\doi{10.1137/19m128973x}.
%Type = Book
\bibitem[{Hundsdorfer and Verwer(2003)}]{Hundsdorfer}
\bibinfo{author}{W.~Hundsdorfer}, \bibinfo{author}{J.~Verwer}, \bibinfo{title}{Numerical Solution of Time-Dependent Advection-Diffusion-Reaction Equations}, \bibinfo{publisher}{Springer Berlin Heidelberg}, \bibinfo{year}{2003}. \URLprefix \url{http://dx.doi.org/10.1007/978-3-662-09017-6}. \DOIprefix\doi{10.1007/978-3-662-09017-6}.
%Type = Article
\bibitem[{Liotta et~al.(2000)Liotta, Romano, and Russo}]{liotta2000central}
\bibinfo{author}{S.~F. Liotta}, \bibinfo{author}{V.~Romano}, \bibinfo{author}{G.~Russo},
\newblock \bibinfo{title}{Central schemes for balance laws of relaxation type},
\newblock \bibinfo{journal}{SIAM Journal on Numerical Analysis} \bibinfo{volume}{38} (\bibinfo{year}{2000}) \bibinfo{pages}{1337--1356}.
%Type = Article
\bibitem[{Tan et~al.(2021)Tan, Cheng, and Shu}]{TanChenShu_ImEx_Stability}
\bibinfo{author}{M.~Tan}, \bibinfo{author}{J.~Cheng}, \bibinfo{author}{C.-W. Shu},
\newblock \bibinfo{title}{Stability of high order finite difference schemes with implicit-explicit time-marching for convection-diffusion and convection-dispersion equations},
\newblock \bibinfo{journal}{Int. J. Numer. Anal. Model.} \bibinfo{volume}{18} (\bibinfo{year}{2021}) \bibinfo{pages}{362--383}. \URLprefix \url{www.global-sci.org/intro/article_detail/ijnam/18730.html}.
%Type = Book
\bibitem[{{\"O}ffner(2023)}]{offner2023approximation}
\bibinfo{author}{P.~{\"O}ffner}, \bibinfo{title}{Approximation and stability properties of numerical methods for hyperbolic conservation laws}, \bibinfo{publisher}{Springer Nature}, \bibinfo{year}{2023}.
%Type = Article
\bibitem[{Micalizzi and Torlo(2023)}]{torlo2022}
\bibinfo{author}{L.~Micalizzi}, \bibinfo{author}{D.~Torlo},
\newblock \bibinfo{title}{A new efficient explicit deferred correction framework: Analysis and applications to hyperbolic {PDEs} and adaptivity},
\newblock \bibinfo{journal}{Communications on Applied Mathematics and Computation}  (\bibinfo{year}{2023}). \URLprefix \url{http://dx.doi.org/10.1007/s42967-023-00294-6}. \DOIprefix\doi{10.1007/s42967-023-00294-6}.
%Type = Article
\bibitem[{Tang et~al.(2013)Tang, Xie, and Yin}]{tang2013high}
\bibinfo{author}{T.~Tang}, \bibinfo{author}{H.~Xie}, \bibinfo{author}{X.~Yin},
\newblock \bibinfo{title}{High-order convergence of spectral deferred correction methods on general quadrature nodes},
\newblock \bibinfo{journal}{J. Sci. Comput.} \bibinfo{volume}{56} (\bibinfo{year}{2013}) \bibinfo{pages}{1--13}. \DOIprefix\doi{10.1007/s10915-012-9657-9}.
%Type = Article
\bibitem[{Ong and Spiteri(2020)}]{dec_overview}
\bibinfo{author}{B.~W. Ong}, \bibinfo{author}{R.~J. Spiteri},
\newblock \bibinfo{title}{Deferred correction methods for ordinary differential equations},
\newblock \bibinfo{journal}{J. Sci. Comput.} \bibinfo{volume}{83} (\bibinfo{year}{2020}) \bibinfo{pages}{29}. \DOIprefix\doi{10.1007/s10915-020-01235-8}.
%Type = Article
\bibitem[{Boscarino et~al.(2018)Boscarino, Qiu, and Russo}]{zbMATH06856490}
\bibinfo{author}{S.~Boscarino}, \bibinfo{author}{J.-M. Qiu}, \bibinfo{author}{G.~Russo},
\newblock \bibinfo{title}{Implicit-explicit integral deferred correction methods for stiff problems},
\newblock \bibinfo{journal}{SIAM J. Sci. Comput.} \bibinfo{volume}{40} (\bibinfo{year}{2018}) \bibinfo{pages}{a787--a816}. \DOIprefix\doi{10.1137/16M1105232}.
%Type = Book
\bibitem[{Wanner and Hairer(1996)}]{wanner1996solving}
\bibinfo{author}{G.~Wanner}, \bibinfo{author}{E.~Hairer}, \bibinfo{title}{Solving ordinary differential equations {II}: Stiff and Differential-Algebraic Problems}, volume \bibinfo{volume}{375}, \bibinfo{publisher}{Springer Berlin Heidelberg}, \bibinfo{address}{Berlin}, \bibinfo{year}{1996}.
%Type = Book
\bibitem[{Stetter et~al.(1973)}]{stetter1973analysis}
\bibinfo{author}{H.~J. Stetter}, et~al., \bibinfo{title}{Analysis of discretization methods for ordinary differential equations}, volume~\bibinfo{volume}{23}, \bibinfo{publisher}{Springer}, \bibinfo{year}{1973}.
%Type = Article
\bibitem[{Scherer(1979)}]{scherer1979necessary}
\bibinfo{author}{R.~Scherer},
\newblock \bibinfo{title}{A necessary condition for {B}-stability},
\newblock \bibinfo{journal}{BIT Numerical Mathematics} \bibinfo{volume}{19} (\bibinfo{year}{1979}) \bibinfo{pages}{111--115}.
%Type = Book
\bibitem[{Van~Loan and Pitsianis(1993)}]{vanloan1993approximation}
\bibinfo{author}{C.~F. Van~Loan}, \bibinfo{author}{N.~Pitsianis}, \bibinfo{title}{Approximation with Kronecker products}, \bibinfo{publisher}{Springer}, \bibinfo{year}{1993}.
%Type = Misc
\bibitem[{\"Offner et~al.(2024)\"Offner, Petri, and Torlo}]{ourrepo}
\bibinfo{author}{P.~\"Offner}, \bibinfo{author}{L.~Petri}, \bibinfo{author}{D.~Torlo}, \bibinfo{title}{{IMEX\_DeC\_ADER} github repository}, \bibinfo{howpublished}{\url{https://github.com/accdavlo/IMEX_DeC_ADER.git}}, \bibinfo{year}{2024}.
%Type = Book
\bibitem[{Hairer et~al.(1987)Hairer, N{\o}rsett, and Wanner}]{hairer1987solving}
\bibinfo{author}{E.~Hairer}, \bibinfo{author}{S.~N{\o}rsett}, \bibinfo{author}{G.~Wanner}, \bibinfo{title}{Solving Ordinary Differential Equations {I}: Nonstiff Problems}, \bibinfo{publisher}{Springer-Verlag}, \bibinfo{address}{Berlin}, \bibinfo{year}{1987}.
%Type = Article
\bibitem[{Zhong(1996)}]{zhong1996additive}
\bibinfo{author}{X.~Zhong},
\newblock \bibinfo{title}{Additive semi-implicit {Runge--Kutta} methods for computing high-speed nonequilibrium reactive flows},
\newblock \bibinfo{journal}{Journal of Computational Physics} \bibinfo{volume}{128} (\bibinfo{year}{1996}) \bibinfo{pages}{19--31}.
%Type = Article
\bibitem[{Caflisch et~al.(1997)Caflisch, Jin, and Russo}]{caflisch1997uniformly}
\bibinfo{author}{R.~E. Caflisch}, \bibinfo{author}{S.~Jin}, \bibinfo{author}{G.~Russo},
\newblock \bibinfo{title}{Uniformly accurate schemes for hyperbolic systems with relaxation},
\newblock \bibinfo{journal}{SIAM Journal on Numerical Analysis} \bibinfo{volume}{34} (\bibinfo{year}{1997}) \bibinfo{pages}{246--281}.
%Type = Article
\bibitem[{Michel et~al.(2021)Michel, Torlo, Ricchiuto, and Abgrall}]{michel2021spectral}
\bibinfo{author}{S.~Michel}, \bibinfo{author}{D.~Torlo}, \bibinfo{author}{M.~Ricchiuto}, \bibinfo{author}{R.~Abgrall},
\newblock \bibinfo{title}{{Spectral analysis of continuous FEM for hyperbolic PDEs: influence of approximation, stabilization, and time-stepping}},
\newblock \bibinfo{journal}{Journal of Scientific Computing} \bibinfo{volume}{89} (\bibinfo{year}{2021}) \bibinfo{pages}{1--41}.
%Type = Article
\bibitem[{Michel et~al.(2023)Michel, Torlo, Ricchiuto, and Abgrall}]{michel2023spectral}
\bibinfo{author}{S.~Michel}, \bibinfo{author}{D.~Torlo}, \bibinfo{author}{M.~Ricchiuto}, \bibinfo{author}{R.~Abgrall},
\newblock \bibinfo{title}{{Spectral analysis of high order continuous FEM for hyperbolic PDEs on triangular meshes: influence of approximation, stabilization, and time-stepping}},
\newblock \bibinfo{journal}{Journal of Scientific Computing} \bibinfo{volume}{94} (\bibinfo{year}{2023}) \bibinfo{pages}{49}.
%Type = Phdthesis
\bibitem[{Torlo(2020)}]{torlo2020hyperbolic}
\bibinfo{author}{D.~Torlo}, \bibinfo{title}{Hyperbolic problems: high order methods and model order reduction}, Ph.D. thesis, Universit{\"a}t Z{\"u}rich, \bibinfo{year}{2020}.
%Type = Article
\bibitem[{Titarev and Toro(2007)}]{titarev2007analysis}
\bibinfo{author}{V.~Titarev}, \bibinfo{author}{E.~F. Toro},
\newblock \bibinfo{title}{{Analysis of ADER and ADER-WAF schemes}},
\newblock \bibinfo{journal}{IMA journal of numerical analysis} \bibinfo{volume}{27} (\bibinfo{year}{2007}) \bibinfo{pages}{616--630}.
%Type = Article
\bibitem[{Dematt{\'e} et~al.(2020)Dematt{\'e}, Titarev, Montecinos, and Toro}]{dematte2020ader}
\bibinfo{author}{R.~Dematt{\'e}}, \bibinfo{author}{V.~Titarev}, \bibinfo{author}{G.~Montecinos}, \bibinfo{author}{E.~Toro},
\newblock \bibinfo{title}{{ADER methods for hyperbolic equations with a time-reconstruction solver for the generalized Riemann problem: the scalar case}},
\newblock \bibinfo{journal}{Communications on Applied Mathematics and Computation} \bibinfo{volume}{2} (\bibinfo{year}{2020}) \bibinfo{pages}{369--402}.
%Type = Article
\bibitem[{Iserles(1982)}]{Iserles1982}
\bibinfo{author}{A.~Iserles},
\newblock \bibinfo{title}{Order stars and a saturation theorem for first-order hyperbolics},
\newblock \bibinfo{journal}{IMA Journal of Numerical Analysis} \bibinfo{volume}{2} (\bibinfo{year}{1982}) \bibinfo{pages}{49–61}. \URLprefix \url{http://dx.doi.org/10.1093/imanum/2.1.49}. \DOIprefix\doi{10.1093/imanum/2.1.49}.
%Type = Article
\bibitem[{Fornberg(1988)}]{fornberg_finite_difference}
\bibinfo{author}{B.~Fornberg},
\newblock \bibinfo{title}{Generation of finite difference formulas on arbitrarily spaced grids},
\newblock \bibinfo{journal}{Mathematics of Computation} \bibinfo{volume}{51} (\bibinfo{year}{1988}) \bibinfo{pages}{699–706}. \URLprefix \url{http://dx.doi.org/10.1090/s0025-5718-1988-0935077-0}. \DOIprefix\doi{10.1090/s0025-5718-1988-0935077-0}.
%Type = Book
\bibitem[{LeVeque(2007)}]{leveque2007finite}
\bibinfo{author}{R.~J. LeVeque}, \bibinfo{title}{Finite difference methods for ordinary and partial differential equations: steady-state and time-dependent problems}, \bibinfo{publisher}{SIAM}, \bibinfo{year}{2007}.
%Type = Article
\bibitem[{Wang et~al.(2015)Wang, Shu, and Zhang}]{WangShuZhang_LDG1_2015}
\bibinfo{author}{H.~Wang}, \bibinfo{author}{C.-W. Shu}, \bibinfo{author}{Q.~Zhang},
\newblock \bibinfo{title}{Stability and error estimates of local discontinuous {Galerkin} methods with implicit-explicit time-marching for advection-diffusion problems},
\newblock \bibinfo{journal}{SIAM Journal on Numerical Analysis} \bibinfo{volume}{53} (\bibinfo{year}{2015}) \bibinfo{pages}{206–227}. \URLprefix \url{http://dx.doi.org/10.1137/140956750}. \DOIprefix\doi{10.1137/140956750}.
%Type = Article
\bibitem[{Wang et~al.(2016)Wang, Shu, and Zhang}]{WangShuZhang_LDG_2016}
\bibinfo{author}{H.~Wang}, \bibinfo{author}{C.-W. Shu}, \bibinfo{author}{Q.~Zhang},
\newblock \bibinfo{title}{Stability analysis and error estimates of local discontinuous {Galerkin} methods with implicit–explicit time-marching for nonlinear convection–diffusion problems},
\newblock \bibinfo{journal}{Applied Mathematics and Computation} \bibinfo{volume}{272} (\bibinfo{year}{2016}) \bibinfo{pages}{237–258}. \URLprefix \url{http://dx.doi.org/10.1016/j.amc.2015.02.067}. \DOIprefix\doi{10.1016/j.amc.2015.02.067}.
%Type = Misc
\bibitem[{Taylor(2016)}]{fdcc}
\bibinfo{author}{C.~R. Taylor}, \bibinfo{title}{Finite difference coefficients calculator}, \bibinfo{howpublished}{\url{https://web.media.mit.edu/~crtaylor/calculator.html}}, \bibinfo{year}{2016}.
%Type = Article
\bibitem[{Ranocha(2021)}]{ranocha2021sbp}
\bibinfo{author}{H.~Ranocha},
\newblock \bibinfo{title}{{SummationByPartsOperators.jl}: {A} {J}ulia library of provably stable semidiscretization techniques with mimetic properties},
\newblock \bibinfo{journal}{Journal of Open Source Software} \bibinfo{volume}{6} (\bibinfo{year}{2021}) \bibinfo{pages}{3454}. \URLprefix \url{https://github.com/ranocha/SummationByPartsOperators.jl}. \DOIprefix\doi{10.21105/joss.03454}.
%Type = Article
\bibitem[{Ortleb(2023)}]{ortleb2023stability}
\bibinfo{author}{S.~Ortleb},
\newblock \bibinfo{title}{{On the stability of IMEX Upwind gSBP schemes for 1D linear advection-diffusion equations}},
\newblock \bibinfo{journal}{Communications on Applied Mathematics and Computation}  (\bibinfo{year}{2023}) \bibinfo{pages}{1--30}.
%Type = Article
\bibitem[{Ortleb(2020)}]{zbMATH07137361}
\bibinfo{author}{S.~Ortleb},
\newblock \bibinfo{title}{{{\(L^2\)}}-stability analysis of {IMEX}-{{\(( \sigma,\mu )\)}} {DG} schemes for linear advection-diffusion equations},
\newblock \bibinfo{journal}{Appl. Numer. Math.} \bibinfo{volume}{147} (\bibinfo{year}{2020}) \bibinfo{pages}{43--65}. \DOIprefix\doi{10.1016/j.apnum.2019.08.016}.
%Type = Article
\bibitem[{Lozano(2019)}]{zbMATH07086321}
\bibinfo{author}{C.~Lozano},
\newblock \bibinfo{title}{Entropy production by implicit {Runge}-{Kutta} schemes},
\newblock \bibinfo{journal}{J. Sci. Comput.} \bibinfo{volume}{79} (\bibinfo{year}{2019}) \bibinfo{pages}{1832--1853}. \DOIprefix\doi{10.1007/s10915-019-00914-5}.
%Type = Article
\bibitem[{Lozano(2018)}]{zbMATH06928679}
\bibinfo{author}{C.~Lozano},
\newblock \bibinfo{title}{Entropy production by explicit {Runge}-{Kutta} schemes},
\newblock \bibinfo{journal}{J. Sci. Comput.} \bibinfo{volume}{76} (\bibinfo{year}{2018}) \bibinfo{pages}{521--564}. \DOIprefix\doi{10.1007/s10915-017-0627-0}.
%Type = Article
\bibitem[{{\"O}ffner et~al.(2020){\"O}ffner, Glaubitz, and Ranocha}]{oeffner2020}
\bibinfo{author}{P.~{\"O}ffner}, \bibinfo{author}{J.~Glaubitz}, \bibinfo{author}{H.~Ranocha},
\newblock \bibinfo{title}{Analysis of artificial dissipation of explicit and implicit time-integration methods},
\newblock \bibinfo{journal}{Int. J. Numer. Anal. Model.} \bibinfo{volume}{17} (\bibinfo{year}{2020}) \bibinfo{pages}{332--349}. \URLprefix \url{www.math.ualberta.ca/ijnam/Volume-17-2020/No-3-20/2020-03-03.pdf}.
%Type = Article
\bibitem[{Tan et~al.(2022)Tan, Cheng, and Shu}]{tan2022stability}
\bibinfo{author}{M.~Tan}, \bibinfo{author}{J.~Cheng}, \bibinfo{author}{C.-W. Shu},
\newblock \bibinfo{title}{Stability of high order finite difference and local discontinuous {Galerkin} schemes with explicit-implicit-null time-marching for high order dissipative and dispersive equations},
\newblock \bibinfo{journal}{Journal of Computational Physics} \bibinfo{volume}{464} (\bibinfo{year}{2022}) \bibinfo{pages}{111314}.

\end{thebibliography}

\end{document}